\documentclass[11pt]{amsart}  
\usepackage{amscd}
\usepackage{amssymb,amsxtra} 
\usepackage[all,cmtip]{xy}

\DeclareMathAlphabet{\mathsl}{OT1}{cmr}{m}{sl}

\newcommand{\Q}{\mathbb{Q}}
\newcommand{\Z}{\mathbb{Z}}

\newcommand{\N}{\mathbb{N}}

\newcommand{\grB}{\mathsf{B}}

\newcommand{\grD}{\mathsf{D}}
\newcommand{\grE}{\mathsf{E}}

\newcommand{\grM}{\mathsf{M}}
\newcommand{\grN}{\mathsf{N}}

\newcommand{\grV}{\mathsf{V}}

\newcommand{\id}{\operatorname{\mathsl{id}}}

\renewcommand{\min}{\operatornamewithlimits{\mathsl{min}}}
\renewcommand{\max}{\operatornamewithlimits{\mathsl{max}}}
\renewcommand{\deg}{\operatorname{\mathsl{deg}}}
\renewcommand{\dim}{\operatorname{\mathsl{dim}}}
\renewcommand{\ker}{\operatorname{\mathsl{ker}}}

\renewcommand{\int}{\operatorname{\mathsl{int}}}
\renewcommand{\tilde}{\widetilde}

\DeclareMathOperator{\End}{\operatorname{\mathsl{End}}}
\DeclareMathOperator{\grEnd}{\operatorname{\mathsf{End}}}

\DeclareMathOperator{\gr}{\operatorname{\mathsf{gr}}}

\DeclareMathOperator{\ind}{\operatorname{\mathsl{ind}}}

\DeclareMathOperator{\im}{\operatorname{\mathsl{im}}}
\DeclareMathOperator{\Gal}{\operatorname{\mathcal{G}}}

\DeclareMathOperator{\inv}{\operatorname{\mathsl{inv}}}

\DeclareMathOperator{\charac}{\operatorname{\mathsl{char}}}

\DeclareMathOperator{\srk}{\operatorname{\mathsl{srk}}}

\def\hsp{\mspace{1mu}}
\newcommand{\DIM}[2]{[#1{\hsp:\hsp}#2]}
\newcommand{\IND}[2]{\lvert#1{\hsp:\hsp}#2\rvert}

\def\tsum{\textstyle\sum\limits}
\def\tprod{\textstyle\prod\limits}
\def\tbigoplus{\textstyle\bigoplus\limits}
\def\tbigcup{\textstyle\bigcup\limits}
\def\tbigsqcup{\textstyle\bigsqcup\limits}
\def\tbigcap{\textstyle\bigcap\limits}

\def\inv{^{-1}}

\def\half{\frac12}

\def\ov{\overline}

\newcommand{\jdotfont}{}
\font\jdotfont lcircle10  scaled 913 
\newcommand{\osbullet}{\jdotfont\char113} 
\newlength{\sbwd} 
\newcommand{\csbullet}{\kern.5\sbwd\osbullet\kern-.5\sbwd}

\newcommand{\alg}{\text{\textsl{alg}}}
\newcommand{\sep}{\text{\textsl{sep}}}

\renewcommand{\gcd}{\textsl{gcd}}

\newcommand{\divh}[1]{{\mathbb H}(#1)}

\newcommand{\chii}{\raise 1.5pt\hbox{$\chi$}}

\newdir{ (}{{}*!/-5pt/@^{(}}

\newcommand{\mat}{\mathbb M}

\newtheorem{theorem}{Theorem}[section]
\newtheorem{lemma}[theorem]{Lemma}
\newtheorem{proposition}[theorem]{Proposition}
\newtheorem{corollary}[theorem]{Corollary}

\theoremstyle{definition}
\newtheorem{definition}[theorem]{Definition}
\newtheorem{example}[theorem]{Example}

\newtheorem{setup}[theorem]{Setup}
\newtheorem{remark}[theorem]{Remark}

\theoremstyle{remark}

\numberwithin{equation}{section}

\newtheoremstyle{note}
  {3pt}
  {3pt}
  {}
  {}
  {\bf}
 {}
  {.01em}
  {}

\theoremstyle{note}

\begin{document}

\title{Value functions and Dubrovin valuation rings on simple algebras}

\author{Mauricio A. Ferreira}
\address{Departamento de Ci\^encias Exatas, Universidade Estadual de Feira
de Santana, Avenida Transnordestina, S/N, Novo Horizonte, Feira de
Santana, Bahia 44036-900 \ \ Brazil}
\email{maferreira@uefs.br}
\thanks{Some of the results in this paper are based on the first author's doctoral
dissertation written under supervision of the second author and Antonio Jos\' e
Engler. The first author was partially supported by FAPESP, Brazil 
\mbox{(Grant
06/00157-3)} during his graduate studies. This author would like to thank
the second author and UCSD for their hospitality during his visit in 2009.}
%
\author{Adrian R. Wadsworth}
\address{Department of Mathematics 0112, University of California, San Diego,
9500 Gilman Drive, La~Jolla, California 92093-0112  \ \ USA}
\email{arwadsworth@ucsd.edu}
\thanks{Some of the research for this paper was carried out during the second
author's visit to the University of Campinas, Brazil, during  July-August, 
2010, which was made possible by a grant from FAPESP, Brazil.  The second 
author would like 
thank Prof. Antonio Jos\'e Engler and the first author for their hospitality 
during that visit.} 





\begin{abstract}
In this paper we prove relationships between two generalizations
of commutative valuation theory for noncommutative central simple
algebras:  (1) Dubrovin valuation rings; and (2)~the value
functions called gauges introduced by Tignol and 
Wadsworth in \cite{TW1} and \cite{TW2}.  We show that if 
$v$ is a valuation on a field $F$ with associated valuation ring 
$V$ and $v$ is defectless in a central simple $F$-algebra  $A$, 
and $C$ is a subring of $A$, then the following are equivalent:
(a)~$C$~
 is the gauge ring of some minimal $v$-gauge
on $A$, i.e., a gauge with the minimal number of simple components of 
$C/J(C)$; (b)~$C$~is integral over $V$ with  
$C = B_1 \cap \ldots \cap B_\xi$ where each 
$B_i$ is a Dubrovin valuation ring of $A$ with center~$V$,  
and the~$B_i$~satisfy Gr\"ater's Intersection Property. 
Along the way we prove the existence of minimal gauges whenever
possible and we show how gauges on simple algebras are built from
gauges on central simple algebras.
\end{abstract}

\maketitle
\begin{center}
\end{center}

\section*{Introduction} 

Valuation theory has been a very useful tool in the study of 
finite-dimensional division algebras, particularly in the construction 
of examples, such as noncrossed products and division algebras with 
nontrivial reduced Whitehead group $SK_1$.   
(See \cite{W3} for a survey of valuation theory on division 
algebras.) But there has been some 
difficulty in applying valuation theory in noncommutative settings
because division algebras do not have many valuations and 
simple algebras that are not division algebras do not have valuations.
This has led to efforts to find structures similar to but less 
restrictive than valuations that would exist more widely.

One such approach was initiated by Dubrovin in \cite{Du1} and 
\cite{Du2}.  By generalizing the idea of places in commutative 
valuation theory, he defined what are now called Dubrovin 
valuation rings.   Such rings share many of the distinctive 
properties of commutative   valuation rings, and it is known that
for every central simple algebra $A$ over a field $F$  and every 
valuation ring $V$ of $F$ there is a Dubrovin valuation ring 
$B$ of $A$ with center $V$, and $B$ is unique up to isomorphism.
However, there is in general no valuation associated with 
such a $B$ (except in the integral case, see below).  The 
substantial theory of Dubrovin valuation rings is the topic
of the book \cite{MMU}.  

Another approach was initiated rather recently by Tignol and 
the second author in \cite{TW1} and \cite{TW2} by the introduction
of {\it gauges}, which are a kind of value function on a semisimple
algebra~$S$ finite-dimensional over a field $F$, but satisfying weaker
axioms than for a valuation.  (See \S1 below for the definition
of a gauge.) The theory of gauges is still developing,
but it has already become a useful complement to the classical 
valuation theory of division algebras.  A gauge $\alpha$ on~ 
$S$ always extends some valuation $v$ on $F$, so $\alpha$ is called 
a $v$-gauge.  Just as the valuation $v$ induces a filtration on~
$F$, so $\alpha$~induces a filtration on $S$ yielding an associated graded
ring $\gr_\alpha(S)$, which is a finite-rank semisimple graded
algebra over the  graded field $\gr_v(F)$.  The associated graded
algebra captures much of the essential information about $\alpha$ on 
$S$, but $\gr_\alpha(S)$ is often much easier to work with than 
 with $S$ itself.  Gauges are easy to construct in many cases,  
and they have good behavior with respect to tensor products and
scalar extensions of algebras.

The question naturally arises what kind of connections there
may be between Dubrovin valuation rings and gauges.  A limited
answer was provided in \cite{TW1}:  Morandi had shown in \cite{M2}
that a Dubrovin valuation ring $B$ of a central simple $F$-algebra $A$ 
has a special kind of associated value function $\mu_B$ if and only  
if $B$ is integral over its center $Z(B)$; when this 
occurs, $B$~determines~$\mu_B$ and vice versa since 
$B = \{ a\in A \mid \mu_B(a) \ge 0\}$.   
Here, $Z(B)$  is a valuation
ring $V$ of $F$, say with associated valuation~$v$.  In 
\cite[Prop.~2.5]{TW1}  it was shown that the gauge ring 
${R_\alpha= \{ a\in A \mid \alpha(a)\ge 0 \}}$ of a $v$-gauge 
$\alpha$ on $A$ is a Dubrovin valuation ring if and
only if its residue ring $R_\alpha /J(R_\alpha)$ is a simple
ring.  ($J(\raisebox{1pt}{$\scriptstyle\bullet$})$~denotes the Jacobson radical of 
\raisebox{1pt}{$\scriptstyle\bullet$}.) 
Moreover, when this occurs,  $R_\alpha$ is integral over 
its center $V$ and $\alpha = \mu_{R_\alpha}$.  Furthermore,
if~$v$~ is defectless in~$A$, then every Dubrovin valuation ring
$B$ of~$A$ with center $V$ and integral over~$V$ is the gauge
ring of the  $v$-gauge $\mu_B$ on $A$. 
Since it is known that gauge rings~ 
$R_\alpha$ are always integral over their centers, and 
Dubrovin valuation rings $B$ always satisfy $B/J(B)$ is simple, this 
result says that Dubrovin valuation rings and gauge rings
coincide \lq\lq whenever possible." 
When $V$ has rank 
\mbox{(= Krull dimension)} $1$, 
then every Dubrovin valuation ring of $A$ with center~$V$
is integral over $V$.  
But, when $V$ has rank $2$ or more, 
the Dubrovin valuation rings of $A$ with center~$V$ are very 
often not integral over $V$.  

We show in this paper that 
there are still significant, and somewhat surprising,
connections between the Dubrovin theory and gauges even when
the Dubrovin valuation rings are not integral over their centers.
We use the special intersections of Dubrovin valuation rings
analyzed by Gr\"ater in \cite{G}.  He showed that to 
every valuation ring $V$ of a field $F$ and every central 
simple $F$-algebra $A$ there is a subring $C$ of $A$,
with $C$ integral over $V$ and  determined
uniquely up to isomorphism, such that $C = B_1 \cap \ldots \cap
B_\xi$, where the $B_i$ are each Dubrovin valuation rings of 
$A$ with center $V$, and the~$B_i$ are related by satisfying
a special \lq \lq Intersection Property."  We dub such a 
$C$ a {\it Gr\"ater~ring} for~$V$ in $A$.  Gr\"ater proved
that among other nice properties $C$ is a noncommutative 
B\'ezout ring and that the $B_i$ are determined from 
$C$ as the localizations of $C$ with respect to its maximal ideals.
The number $\xi$ of $B_i$ in the intersection (= the number of maximal
ideals of $C$) is an invariant $\xi = \xi_{V,[A]}$ of $V$ and the 
Brauer class $[A]$ of $A$.  Gr\"ater called this number the 
\lq\lq extension number."  The same number had appeared earlier
in \cite{W} in the \lq\lq Ostrowski Theorem" for   Dubrovin 
valuation rings. The extension number equals $1$ if and only if 
some (hence every) Dubrovin valuation ring of~$A$ with center $V$
is integral over $V$. 

Let $\alpha$ be a $v$-gauge on  the central simple $F$-algebra $A$.  
The degree zero piece of  
the associated graded ring $\gr_\alpha(A)$, denoted $A_0^\alpha$,
coincides with $R_\alpha/J(R_\alpha)$, and is a semisimple ring
finite-dimensional over the residue ring $\ov F^{\,v}$ of the valuation  
$v$.  Let 
$$
\omega(\alpha) \, = \, \text{the number of simple components of 
the semisimple ring $A_0^\alpha$}.
$$
We  show in Th.~\ref{minimalgauge1} that $\omega(\alpha) \ge \xi_{V,[A]}$, 
where $V$ is the valuation ring of $v$.  We call $\alpha$ a {\it minimal~gauge}
if equality holds.
 We  show in 
Th.~\ref{teoremaprincipal} that if $\alpha$ is a minimal gauge on $A$ then 
its gauge ring $R_\alpha$ is a Gr\"ater ring. Conversely, 
we prove in Th.~\ref{teoremaprincipalvolta}, that if $v$ is defectless in 
$A$ then a Gr\"ater ring of $A$ with center $V$ is the gauge ring of  some minimal
$v$-gauge.   

Defect in valuation theory refers to the failure of
equality in the Fundamental Inequality.  The  background results we need on 
defect are given in \S\ref{subsec:defect}.  The defect is trivial for 
valuations of residue characteristic $0$ or of prime residue 
characteristic not dividing the index of a central simple algebra.
It is not hard to prove  that if a semisimple $F$-algebra 
$S$ has a $v$-gauge, then the valuation $v$ must be defectless in 
$S$ (see Prop.~\ref{prop:gaugeimpliesdefectless}).  In 
Th.~\ref{existenceminimalgauges} we prove the very  nontrivial converse
that if $v$ is defectless in~$S$, then $S$ necessarily has a minimal $v$-gauge.
This is of interest in itself, and is also essential for the proof 
of Th.~\ref{teoremaprincipalvolta}.  

Our approach to proving results about gauges and Gr\"ater rings
with respect to a central valuation~$v$ often involves working 
back from corresponding objects for a coarser valuation $w$.
A $v$-gauge $\alpha$ on a central simple $F$-algebra $A$ has a coarsening  
to a $w$-gauge $\beta$ on $A$ with an induced $v/w$-gauge on the 
residue ring  $A_0^\beta$, see Prop.~\ref{gaugealphazero}.  However, 
the semisimple $\ov F^{\,w}$-algebra $A_0^\beta$ need not be 
central simple.
Therefore, it has been necessary to determine how gauges on 
semisimple algebras are related to gauges for 
central simple algebras.
We do this in \S 2. Reduction from the semisimple case to the 
simple case is very easy. The simple case turns out to have a 
nice description
that takes some work to prove:  
 If $S$ is a finite-dimensional simple $F$-algebra, 
and $v$ is a
valuation on $F$, let $v_1, \ldots, v_r$ be the extensions of 
$v$ to the center $Z(S)$.  If~$\alpha$~is a $v$-gauge on 
$F$, we show in Th.~\ref{thm:gaugeismin} that there are uniquely determined 
$v_i$-gauges 
$\alpha_i$ on $S$ such that $\alpha = \min(\alpha_1, \dots, \alpha_r)$.
Moreover, in Th.~\ref{thm:compatible} we give a necessary and sufficient 
compatibility condition on $v_i$-gauges $\alpha_i$ so 
that $\min(\alpha_1, \ldots, \alpha_r)$ is a $v$-gauge.
This is a notable result  in its own right for  the 
still-developing theory of gauges. 

For valuations on fields, the valuation is determined 
(up to equivalence) by the valuation ring.  This is likewise
true for the Morandi value functions determined by 
Dubrovin valuation rings integral over their centers.
We give an example in \S5 to show that the corresponding property 
does not hold for minimal gauges.  In our example there are
infinitely many nonisomorphic minimal gauges on a central simple 
algebra all with the same gauge ring $R_\alpha$.

There is a book in preparation \cite{TW3} that will give a more extensive 
treatment of gauges than is available in the currently published literature.  
We have adapted \S\ref{subsec:defect} below on defect and the result in the Appendix
from that book because the material is essential for this paper and
there is no adequate published reference available.  Also, the existence of 
gauges for defectless algebras was first proved in that book.  What is needed here
to complete the results in \S 3 below is the existence of {\it minimal} gauges
for defectless algebras.  This is proved in \S 4 by an argument that is 
substantially different from the one in \cite{TW3}.

\section{Value functions on semisimple algebras}

For the convenience of the reader, we recall a few basic properties of 
gauges proved in \cite{TW1} and \cite{TW2}. First, a few words of notation:
If $R$ is any ring (with $1$), we write: $R^\times$ for the group of units
of $R$; $Z(R)$ for the center of $R$; $J(R)$ for the Jacobson radical of
$R$; and $\mat _k(R)$ for the $k\times k$-matrix ring over $R$.
Fix  for now a divisible totally ordered abelian group $\Gamma$, chosen to be 
sufficiently large to contain the
values of all valuations and the degrees of all gradings we consider.
If $v\colon F \to \Gamma\cup \{\infty\}$ is a
valuation on a field $F$, let $V$ be the corresponding valuation
ring, $V = \{ x\in F \, | \, v(x) \geq 0\}$; then $J(V) =
\{ x\in F \, | \, v(x) > 0\}$, which is the unique maximal 
ideal of $V$. We write~
$\overline{F}^{\,v}$ for the residue field $V /J(V)$ and
 $\Gamma_{v}$~for the value group~
$v(F^\times)$, which is a subgroup of $\Gamma$. 
We say that $v$ and $V$ are trivial if $V = F$.  For basic assumed
background on valuations on fields we refer to Bourbaki 
\cite[Ch. 6]{B} or Engler-Prestel \cite{EP}. 
For valuations on division rings we use analogous notation and
terminology to that for fields.  A good reference for valuation 
theory on division rings is the paper~\cite{JW}.   

Let $D$ be a division ring finite-dimensional over its center. A valuation
$v\colon D \rightarrow \Gamma\cup \{\infty \}$ defines a filtration on $D$: for
each $\gamma \in \Gamma$, we set
\[
D_{\geq \gamma} \,=\, \{ d \in D \; | \; v(d) \geq \gamma \},
\quad D_{> \gamma} \,=\, \{ d \in D \; | \; v(d)  > \gamma \}, \quad
\text{ and } \quad D_{\gamma} \,=\, D_{\geq \gamma} / D_{>\gamma}.
\]
The associated graded ring for $v$ on $D$ is
$$
\gr_v(D) \,= \,  \tbigoplus_{\gamma \in \Gamma} D_{\gamma}. 
$$
It is a \emph{graded division ring} because every nonzero homogeneous
element in $\gr(D)$ is invertible.
(If $D$ is commutative, then $\gr(D)$
is also commutative and is then called a \emph{graded field}.)
The grade set $\Gamma_{\gr(D)}$ of $\gr(D)$ is 
$
 \{\gamma\in \Gamma\;|\; D_\gamma\ne \{0\} \}, 
$
which is a subgroup of $\Gamma$ coinciding  with $\Gamma_v$.  Note
that $D_0$~is a division ring.  Also, since $\Gamma_v$ is
torsion-free, $D^\times = \bigcup_{\gamma\in\Gamma_v}D_\gamma 
\setminus
\{0\}$. Now, let $\grM = \bigoplus_{\gamma\in \Gamma}\grM_\gamma$ be
a graded
right $\gr(D)$-module, i.e., $\grM$ is a right  $\gr(D)$-module with 
$ {\grM_\gamma\!\cdot\! D_\delta} \subseteq \grM_{\gamma + \delta}$ for all~
$\gamma, \delta \in \Gamma$.  Then, $\grM$ is a free $\gr(D)$-module
with well-defined rank and a homogeneous base.  We therefore
call~$\grM$ a right graded $\gr(D)$-vector space, and write
$\dim_{\gr(D)}\grM$ for the rank of $\grM$ as   a $\gr(D)$-module.
If $\grN$~ is another graded $\gr(D)$-vector space, then a
$\gr(D)$-homomorphism
$\psi \colon \grM\to \grN$ is called a graded homomorphism if
$\psi(\grM_\gamma) \subseteq \grN_\gamma$ for all $\gamma\in \Gamma$.
We write
$\grM \cong_g \grN$  if
there is a graded isomorphism $\grM \to \grN$. If 
$m \in \grM\setminus \{0\}$ is homogeneous,
then there is a unique $\gamma\in \Gamma$ with $ m \in \grM_\gamma$;
we call~$\gamma$ the degree of $m$, and write $  \gamma = \deg m$.
For any 
$\delta \in \Gamma$, we write $\grM(\delta)$ for the 
$\gr(D)$-vector space obtained from $\grM$ by shifting the grading by 
$\delta$:
\begin{equation} \label{eq:shift}
\grM(\delta) \, = \,\tbigoplus\limits_{\gamma\in \Gamma}\grM(\delta) _\gamma
\qquad \text{where} \qquad \grM(\delta)_\gamma \, = \, \grM_{\gamma+ \delta}.
\end{equation}

Now let $M$ be a right $D$-vector space. A $v$-\emph{value function} on
$M$ is a map $\alpha\colon M \rightarrow \Gamma\cup \{\infty \}$ such that 
\begin{itemize}
\item[(i)]
$\alpha(x) = \infty$ if and only if $x =0$;
\item[(ii)]
 $\alpha(x+y) \geq \min \big(\alpha(x), \alpha(y)\big)$
for $x,y \in M;$
\item[(iii)] $\alpha(xd) = \alpha(x) + v(d)$ for all
$x \in M$ and $d \in D$.
\end{itemize}
We write $\Gamma_{\alpha}$ to denote the value set
$\alpha(A\setminus \{0\})$, which need not be a group, but it is a
union of cosets of $\Gamma_{v}$.
The $v$-value function $\alpha$ is called a $v$-\emph{norm} if
$M$ is finite-dimensional and contains a 
$D$-vector space base $(x_i)_{i=1}^n$ such that
\begin{equation}\label{definicaonorma}
\alpha\big(\tsum_{i=i}^n x_id_i\big) =
\min\limits_{1\leq i \leq n}\big(v(d_i) + \alpha(x_i)\big) \quad
\text{for } d_1, \ldots, d_n \in D.
\end{equation}
Such a base $(x_i)_{i=1}^n$ is called a {\it splitting base} of $M$ for $\alpha$.
Just as with valuations, the value function~$\alpha$ defines a filtration
on $M$: for each
$\gamma \in \Gamma$, we set
$$
M_{\geq \gamma} \,=\, \{ x \in M \; | \; \alpha(x) \geq \gamma \}, \quad
M_{> \gamma} \,=\, \{ x \in M \; | \; \alpha(x) > \gamma \},
\quad \text{and} \quad M_{\gamma} \,=\, M_{\geq \gamma} / M_{>\gamma}.
$$
(When we need to specify the value function, we write
$M_{\geq \gamma}^\alpha, M_{> \gamma}^\alpha$, and $M_\gamma^\alpha$.)
The associated graded object
\[
\gr_\alpha(M) =  \tbigoplus_{\gamma \in \Gamma} M_{\gamma}
\]
is a graded right $\gr(D)$-vector space with
$\dim_{\gr(D)}\gr_\alpha(M) \le \dim _DM$ (see \cite[Cor.~2.3]{RTW}). 
 We will write  $\gr(M)$ instead of $\gr_\alpha(M)$ when the value
function  $\alpha$ is clear.
Every nonzero element $x \in M$ has an image $\tilde {x}^{\,\alpha}$ in~
$\gr(M)$, defined by
$\tilde {x}^{\,\alpha} = x+ M_{>\alpha(x)} \in M_{\alpha(x)}$. We will
often
write simply $\tilde x$ instead of $\tilde {x}^{\,\alpha}$ when
  $\alpha$~is clear. A $D$-base $(x_i)_{i=1}^k$ of $M$ is
a splitting base of $M$ for $\alpha$ if and only if
$(\tilde x_i)_{i=1}^k$ is a $\gr(D)$-base of~$\gr(M)$
(see \cite[Cor.~2.3]{RTW}). Hence, 
$\alpha$ is a $v$-norm on $M$ if and only if
$\dim_{\gr(D)}\gr(M) = \dim_{D}M$.

A $v$-value function $\alpha$ on a finite-dimensional $F$-algebra $A$ is
\emph{surmultiplicative} if $\alpha(1) = 0$ and
$\alpha(xy) \geq \alpha(x) + \alpha(y)$ for $x,y \in A.$
For such an $\alpha$, set
\begin{equation}\label{eq:defRalpha}
R_\alpha \,=\, \{ x \in A \: | \: \alpha(x) \geq 0 \}
\qquad\text{and} \qquad
J_\alpha = \{ x \in A \: | \: \alpha(x) > 0\}.
\end{equation}
It is clear from the axioms for $\alpha$ that
$R_\alpha$ is a subring of $A$
and $J_\alpha$ is a two-sided ideal
of $R_\alpha$.  
One can check further that $1+ J_\alpha \subseteq R_\alpha^\times$;
hence $J_\alpha$ lies in the Jacobson radical $J(R_\alpha)$.
Thus, if $R_\alpha/J_\alpha$ is semisimple, then
$J_\alpha = J(R_\alpha)$. Also, 
if $F = Z(A)$ and $\Gamma_\alpha$ lies in the 
divisible hull of $\Gamma_v$ then  $Z(R_\alpha) = R_\alpha \cap F = V$,
the valuation ring of $v$,
 and
$J(R_\alpha) \cap F = J(V)$. Moreover, if in addition
  $\alpha$ is  a $v$-norm, then it follows from \cite[Lemma~1.20 and
Th.~3.1]{TW1} that $R_\alpha$ is  integral over $V$. Because
$\alpha$ is surmultiplicative, the
multiplication in $A$ induces a multiplication on the graded $\gr(F)$-vector
space $\gr_\alpha(A)$ such that for $x,y \in A$
\begin{equation} \label{multiplicacaograduada}
{\tilde x}\,{\tilde y}\,=\, xy + A_{>\alpha(x)+ \alpha(y)} \,=\, \left \{
\begin{array}{cl}
\widetilde{xy} & \text{\; if \:} \alpha(xy) = \alpha(x) + \alpha(y),\\
0 & \text{\; if \:} \alpha(xy) > \alpha(x) + \alpha(y).
\end{array}\right.
\end{equation}
Thus, $\gr_\alpha(A)$ is a graded algebra over $\gr(F)$.
We write $\DIM{\gr_\alpha(A)}{\gr(F)}$ for $\dim_{\gr(F)}\gr_\alpha(A)$.

Let $F$ be a field with a valuation $v$ and valuation ring $V$.
A surmultiplicative $v$-value function $\alpha$ on a finite-dimensional
$F$-algebra $A$ is called a $v$-\emph{gauge} if $\alpha$ is a $v$-norm
and $\gr_\alpha(A)$ is a graded semisimple $\gr(F)$-algebra, i.e.,
$\gr_\alpha(A)$ has no nonzero homogeneous nilpotent ideal. In this
case, $R_\alpha$ is  called the \emph{gauge ring} of $\alpha$.
If $\alpha$ is a $v$-gauge on a semisimple $F$-algebra $A$ and 
$\alpha'$ is a $v$-gauge on a semisimple $F$-algebra $A'$ we say 
that $\alpha$ and $\alpha'$ are isomorphic $v$-gauges if there 
is an $F$-algebra isomorphism $\eta \colon A \to A'$ such that 
$\alpha'\circ \eta = \alpha$.

For finite-dimensional graded algebras over a graded field there 
are structure theorems analogous to the Wedderburn theorems
in the ungraded context.  The graded theory is developed in~
\mbox{\cite[\S1]{HW}}.

\begin{example}\label{ex;endgauge}
One of the basic constructions of gauges is that of End-gauges
on endomorphism rings  determined by norms on vector spaces, which we now
recall.
Let $D$ be a finite-dimensional division $F$-algebra and let $M$ be a
finite-dimensional right $D$-vector space. Suppose
the valuation $v$ on~$F$ extends to a
valuation $w$ on $D$ and let $\alpha$ be a $w$-norm on $M$.
Then there is a well-defined surmultiplicative $v$-value function
$\End(\alpha)$
on the endomorphism ring $\End_D(M)$, given 
 by
\begin{equation}\label{endgauge}
\End(\alpha)(f) \,=\, \min_{m\, \in\, M \setminus\{0\}}
\big( \alpha(f(m)) -\alpha(m)\big)
\end{equation}
(see \cite[Prop.~1.19]{TW1}).
Moreover,
$$
\gr_{\End(\alpha)}(\End_D(M)) \, \cong_g \,
\grEnd_{\gr_w(D)}\big(\gr_\alpha(M)\big).
$$
(The endomorphism ring $\grEnd_{\gr_w(D)}\big(\gr_\alpha(M)\big)$ has
a natural grading in which a map ${g \colon\gr_\alpha(M)
\to \gr_\alpha(M)}$ has degree $\gamma$ if $g(M_\delta) \subseteq
M_{\gamma+ \delta}$ for all $\delta \in \Gamma$.)
Thus $\gr_{\End(\alpha)}(\End_D(M))$ is graded simple, i.e., it has
no proper nonzero homogeneous ideals.
It follows by dimension count that $\End(\alpha)$ is a $v$-gauge if and only if $w$ on $D$
is defectless over $v$, i.e.,
${\DIM{\gr_w(D)}{\gr_v(F)}= \DIM DF}$ (see \cite[Prop.~1.19]{TW1}).
Since $\dim_D M <\infty$ the grade set $\Gamma_\alpha
= \alpha(M\setminus\{0\})$ consists of finitely many cosets
of $\Gamma_w$, say $\Gamma_\alpha  = \bigsqcup _{i =1}^k\gamma_i + \Gamma_w$,
a disjoint union.  Then $\gr_\alpha(M)$ has a canonical graded
$\gr(D)$-vector space decomposition
$$
\gr_\alpha(M) \, = \, \tbigoplus\limits_{i = 1}^k \grM_i,
\qquad \text{where }\ \grM_i \, = \,
\tbigoplus\limits_{\delta\, \in\, \gamma_i
+ \Gamma_w} M_\delta \, \cong_g \, \gr(D) \otimes_{D_0}M_{\gamma_i}.
$$
An endomorphism of $\gr_\alpha(M)$ of degree $0$ sends each
$M_\delta$ to itself.  Thus,  the degree-$0$ component of
$\gr_{\End(\alpha)}(\End_D(M))$ is
\begin{equation}\label{eq:end0}
\End_D(M)_0 \,\cong\, \big(\grEnd_{\gr(D)}(\gr(M))\big)_0 \,\cong \,
\tprod\limits
_{i=1}^k\big(\grEnd_{\gr(D)}(\grM_i)\big)_0 \, \cong\,
\tprod_{i = 1}^k \End_{D_0}(M_{\gamma_i}).
\end{equation}
Thus, $\big(\End_{\gr(D)}(\gr(M))\big)_0$  is semisimple with exactly
$k$ simple components, each Brauer equivalent to the division ring $D_0$.
Note that the number $k$ of simple components equals the number of 
cosets of $\Gamma_w$ in $\Gamma_\alpha$.

The graded ring $\grEnd_{\gr(D)}(\gr(M))$ can be viewed as a 
matrix ring with shifted grading as follows:  Let $(m_1,\ldots,
m_n)$ be any splitting base of $M$ with respect to $\alpha$,
and let $\gamma_i = \alpha(m_i)$, for each $i$.  Then 
$(\tilde{m_1}, \ldots, \tilde {m_n})$ is a homogeneous base
of the $\gr(D)$-vector space $\gr(M)$ with $\deg \tilde{m_i}
= \gamma_i$.  Let $\grD = \gr(D)$.  Each $\tilde{m_i}$ spans
a $1$-dimensional graded $\grD$-subspace of $\grM$, and 
$\tilde{m_i}\grD \cong_g \grD(\gamma_i)$, which
 is $\grD$~with its grading shifted by 
$\gamma_i$, as in \eqref{eq:shift}.  Thus, as graded
$\grD$-vector spaces, ${\gr(M)= \bigoplus_{i=1}^n\tilde{m_i}\grD 
\cong_g \bigoplus_{i=1}^n \grD(\gamma_i)}$. Then, as 
graded $\gr(F)$-algebras,
\begin{equation}\label{eq:shiftmatrix}
\grEnd_{\gr(D)}(\gr(M)) \, \cong_g \, \grEnd_\grD\big(\grD(\gamma_1) \oplus \ldots
\oplus \grD(\gamma_n)\big)  \, \cong_g \, \mat_n(\grD)(\gamma_1, \ldots,
\gamma_n),
\end{equation}
where $\mat_n(\grD)(\gamma_1, \ldots,\gamma_n)$ is the matrix ring 
$\mat_n(\grD)$ but with grading shifted so that its $\delta$-component
cosists of the matrices with each $ij$-entry $\grD_{\gamma_j-\gamma_i
+\delta}$, for all $\delta\in \Gamma$.  Since 
$\grD_{\gamma_j-\gamma_i+\delta}\ne \{0\}$ if and only if 
$\delta \in \gamma_i - \gamma_j + \Gamma_w$, we have
\begin{equation}\label{eq:Gammaalpha}
\Gamma_{\End(\alpha)} \, = \, \Gamma_{\grEnd_{\gr(D)}(\gr(M))}
\, = \tbigcup_{i = 1} ^n\tbigcup _{j = 1}^n \gamma_i - \gamma_j + \Gamma_w.
\end{equation} 
\end{example}

We will need in \S4 the following more general construction of 
End-gauges:

\begin{lemma}\label{lemma1existence}
Let $A$ be a semisimple $F$-algebra with a $v$-gauge $\alpha$. 
Let $M$ be a free right $A$-module with base $(m_1, \ldots, m_n)$. 
Take any $\gamma_1, \ldots , \gamma_n \in \Gamma$, and let 
$\eta\colon M \rightarrow \Gamma \cup \{ \infty\}$ be 
the \lq\lq$\alpha$-$v$-norm" defined by
\begin{equation}\label{eq:etadef}
\eta\big(\tsum_{i=1}^n m_ia_i\big) \,=\, 
\min\limits_{1 \leq i \leq n} \big(\gamma_i + \alpha(a_i)\big), 
\quad \text{ for all  } \;a_1, \ldots, a_n \in A.
\end{equation}
Let $E = \End_A(M)$ and let $\psi$ be the $v$-value function 
$\End(\eta)$ on $E$ defined by
\[
\psi(f) \,=\, \min\limits_{1 \leq i \leq n}
\big(\eta(f(m_i)) - \eta(m_i)\big).
\]
Then,
\begin{equation}\label{eq:psif}
\psi(f) = \min\limits_{m \in M \setminus\{0\}}
\big(\eta(f(m)) - \eta(m)\big),
\end{equation}
and $\psi$ is a $v$-gauge on $E$. Moreover, $\gr_\eta(M)$ is a free right 
$\gr_\alpha(A)$-module with base 
$(\widetilde{m}_1, \ldots, \widetilde{m}_n)$, and 
$\gr_\psi(E) \cong_g \grEnd_{\gr_\alpha(A)}(\gr_\eta(M)).$
\end{lemma}
\begin{proof}
Note that for all $m, m'\in M$ and $a \in A$,
\[
\eta(ma) \,\geq\, \eta(m) + \alpha(a) \quad \text{and} 
\quad \eta(m+m') \,\geq\, \min (\eta(m),\eta(m')).
\]
Now, take any $f \in E$ and nonzero $m \in M$, say 
$m = \tsum_{i=1}^n m_ia_i$ with $a_i \in A$. Then,
\begin{eqnarray*}
\eta(f(m)) - \eta(m) 
& = & 
\eta\big(\tsum_{i=1}^n f(m_i)a_i\big) - \eta(m) \ \geq\ 
 \min\limits_{1 \leq i \leq n}
\big(\eta(f(m_i)) + \alpha(a_i)\big) - \eta(m)\\
                     & = & 
\min\limits_{1 \leq i \leq n}
\big[\big(\eta(f(m_i)) - \eta(m_i)\big) + 
\big(\eta(m_i) + \alpha(a_i)\big)\big]- \eta(m)\\
                     & \geq &
 \min\limits_{1 \leq i \leq n}\big(\eta(f(m_i)) - \eta(m_i)\big) 
+ \min\limits_{1 \leq i \leq n} \big(\eta(m_i) + \alpha(a_i)\big)- \eta(m)\\
                     & = & 
\min\limits_{1 \leq i \leq n}\big(\eta(f(m_i)) - \eta(m_i)\big) 
\ =\  \psi(f).
\end{eqnarray*}
Thus, $\eta(f(m)) - \eta(m) \geq \psi(f)$. Hence,
\[
\psi(f) \,=\, \min\limits_{1 \leq i \leq n}
\big(\eta(f(m_i)) - \eta(m_i)\big) \,\geq\, 
\min\limits_{m \,\in M \setminus\{0\}}\big(\eta(f(m)) - \eta(m)\big) 
\,\geq \,\psi(f),
\]
so equality holds throughout, proving \eqref{eq:psif}. 

 Note that for the identity map $\id_M$, clearly 
$\psi(\id_M) = 0$.  Also, let $f,g \in E$. 
Then,
\begin{eqnarray*}
\psi(f \circ g) 
& = & 
\min\limits_{m\, \in M \setminus\{0\}}\big(\eta(f(g(m))) - \eta(m)\big)\\
                & = & 
\min\limits_{m\, \in M \setminus\{0\}}
\big[\big(\eta(f(g(m))) - \eta(g(m))\big) + 
\big(\eta(g(m))- \eta(m)\big)\big]\\
                 & \geq & 
\min\limits_{m \,\in M \setminus\{0\}}\big(\eta(f(g(m))) - \eta(g(m))\big) 
+ \min\limits_{m \,\in M \setminus\{0\}} \big(\eta(g(m))- \eta(m)\big)\\
                 & \geq & 
\psi(f) + \psi(g).
\end{eqnarray*}
Thus, $\psi$ is surmultiplicative.

Now consider the graded structures.  Note that as $\alpha$ is a $v$-norm, 
$\eta$ is a $v$-value function on $M$. So, $\gr_\eta(M)$ is a graded 
$\gr_v(F)$-vector space. Formula \eqref{eq:etadef} shows 
$\eta(ma) \geq \eta(m) + \alpha(a)$ for all $m \in M$ and $a \in A$. 
Hence, the right $A$-module structure of $M$ induces a well-defined 
$\gr_\alpha(A)$-module action on $\gr_\eta(M)$ such that
\[
\widetilde{m}\,\widetilde{a} \,=\, ma + M_{>\eta(m)+ \alpha(a)} \,= \,
\left \{
\begin{array}{cl}
\widetilde{ma} & \text{\; if \:} \eta(ma) = \eta(m) + \alpha(a);\\
0 & \text{\; if \:} \eta(ma) > \eta(m) + \alpha(a).
\end{array}\right.
\]
In particular, $\widetilde{m_i}\,\widetilde{a} = \widetilde{m_ia}$. Let 
$(a_j)_{j=1}^r$ be a splitting base of $A$ for $\alpha$. By formula 
\eqref{eq:etadef}, $(m_ia_j)_{i,j=1}^{n,r}$~is a splitting base of 
$M$ for $\eta$. Hence,
\[
\gr_\eta(M)\,=\,  
\tbigoplus_{i=1}^n\tbigoplus_{j=1}^r  \widetilde{m_ia_j} \gr_v(F) \,=\,
  \tbigoplus_{i=1}^n\tbigoplus_{j=1}^r  
\widetilde{m_i}\,\widetilde{a_j} \gr_v(F) \, =\,  
\tbigoplus_{i=1}^n \widetilde{m_i} \big( \tbigoplus_{j=1}^r  
\widetilde{a_j} \gr_v(F)\big) \,=\,  \tbigoplus_{i=1}^n  \widetilde{m_i} 
\gr_\alpha(A).
\]
Therefore, $\gr_\eta(M)$ is a free $\gr_\alpha(A)$-module of rank $n$ with 
base $(\widetilde{m_1}, \ldots , \widetilde{m_n}).$

Take any nonzero $f \in E$. Formula \eqref{eq:psif} above shows that 
$\eta(f(m)) \geq \psi(f) + \eta(m)$, for all $m \in M$, with equality for 
some $m \in M \setminus \{0\}$. From this one can see that $f$ induces a 
well-defined map $\widehat{f}\colon \gr_\eta(M) \rightarrow \gr_\eta(M)$ 
defined on homogeneous elements by 
\[
\widehat{f}(\widetilde{m}) \,=\, \left \{
\begin{array}{cl}
\widetilde{f(m)} & \text{\; if \:} \eta(f(m)) = \psi(f) + \eta(m);\\
0 & \text{\; if \:} \eta(f(m)) > \psi(f) + \eta(m).
\end{array}\right.
\]
Routine calculations show that $\widehat{f}$ is a $\gr_v(F)$-vector space 
endomorphism of $\gr_\eta(M)$ which maps each $M_\gamma$ into 
$M_{\gamma+ \psi(f)}$. The definition of $\psi$ shows that 
$\widehat{f}(\widetilde{m_i})\neq 0$ for some $i$; so $\widehat{f} \neq 0$. 
Moreover, for any $m \in M$ and $a \in A$, it is easy to check that
\[
\widehat{f}(\widetilde{m}\hsp\widetilde{a} ) \,=\, \left \{
\begin{array}{cl} 
\widetilde{f(ma)} & \text{\; if \:} 
\eta(f(ma)) = \psi(f) + \eta(m) + \alpha(a)\\
0 & \text{\; if \:} \eta(f(ma)) > \psi(f) + \eta(m)+ \alpha(a)
\end{array}\right.  \ =\ \widehat{f}(\widetilde{m})\hsp\widetilde{a}, 
\]
from which it follows that 
$\widehat{f} \in \grEnd_{\gr_\alpha(A)}(\gr_\eta(M))$ and $\widehat{f}$ 
is homogeneous of degree $\psi(f)$. Moreover, for nonzero $g \in E$ and 
any $m \in M,$
\[
\widehat{g} \circ \widehat{f}(\widetilde{m}) \,=\, \left \{
\begin{array}{cl}
\widetilde{g(f(m))} & \text{\; if \:} 
\eta(g(f(m))) =\psi(g)+ \psi(f) + \eta(m),\\
0 & \text{\; if \:} \eta(g(f(m))) >\psi(g)+ \psi(f) + \eta(m),
\end{array}\right.
\]
hence,
\begin{equation}\label{eq:gcircf}
\widehat{g} \circ \widehat{f}\,=\, \left \{
\begin{array}{cl}
\widehat{g \circ f} & \text{\; if \:} 
\psi(g \circ f)\, =\,\psi(g)+ \psi(f);\\
0 & \text{\; if \:} \psi(g \circ f) \,>\,\psi(g)+ \psi(f).
\end{array}\right.
\end{equation}
Also, if $\psi(g)> \psi(f)$, then $\widehat{f+g} = \widehat{f}$. Therefore, 
$\widehat{f}$ depends only on the image $\widetilde{f}$ of $f$ in 
$\gr_\psi (E)$. Thus, there is a well-defined map 
$\iota\colon \gr_{\psi}(E) \rightarrow \grEnd_{\gr_\alpha(A)}(\gr_\eta(M))$ 
given on homogeneous elements by $\iota(\widetilde{f}) = \widehat{f}$ for 
$f \in E$. Using \eqref{eq:gcircf}, one can check that $\iota$ is a graded 
$\gr_v(F)$-algebra homomorphism, and $\iota$ is injective since it is 
injective on each homogeneous component of $\gr_\psi(E)$.

For surjectivity of $\iota$, take any $j,k \in \{1, \ldots, n\}$ and any 
nonzero $a \in A$. Define $h \in E$ by $h(m_k) = m_ja$ and $h(m_i) = 0$ 
for $i \neq k$. Then, as $\eta(m_ja) = \gamma_j + \alpha(a)$, we have 
$\psi(h) = \gamma_j - \gamma_k + \alpha(a)$  and 
$\widehat{h}(\widetilde{m_k}) = \widetilde{m_ja} = 
\widetilde{m_j}\hsp \widetilde{a}$ while $\widehat{h}(\widetilde{m_i}) = 0$ 
for $i \neq k$. Of course, $\widehat{h} = \iota(\hsp\widetilde{h}\hsp) \in 
\im(\iota)$. Because maps such as $\widehat{h}$ generate 
$\grEnd_{\gr_\alpha(A)}(\gr_\eta(M))$ as an abelian group, the map $\iota$ 
is surjective, so a graded isomorphism. Because $\gr_\alpha(A)$ is graded 
semisimple (as $\alpha$ is a $v$-gauge) and $\gr_\eta(M)$ is a free 
$\gr_\alpha(A)$-module, $\grEnd_{\gr_\alpha(A)}(\gr_\eta(M))$ is graded 
semisimple by the graded Wedderburn theory (see \cite[\S1, especially 
Prop.~1.3]{HW}). 
 Hence, $\gr_\psi(E)$ is graded semisimple. Furthermore,
\begin{eqnarray*}
[\gr_\psi(E): \gr_v(F)]
& = &
[\grEnd_{\gr_\alpha(A)}(\gr_\eta(M)):\gr_v(F)]  \ = \  
n^2[\gr_\alpha(A): \gr_v(F)]\\
                       & = &
n^2[A: F]  \ = \   [E:F].
\end{eqnarray*}
Thus, $\psi$ is a surmultiplicative $v$-norm on $E$ with $\gr_\psi(E)$ 
graded semisimple, showing that $\psi$ is a $v$-gauge on $E$.
\end{proof}

The following proposition 
shows how to construct gauges on direct products of 
algebras.

\begin{proposition}\label{gaugesemisimple}
Let $A_1, \ldots , A_k$ be finite-dimensional simple $F$-algebras with
respective $v$-gauges $\alpha_1, \ldots , \alpha_k$, and let
$A= A_1 \times \ldots \times A_k$. The map
$\alpha\colon A \rightarrow \Gamma \cup \{\infty\}$ defined by
\begin{equation}\label{gaugeminsemisimple}
\alpha(x_1, \ldots , x_k) \,=\, \min\big(\alpha_1(x_1), \ldots,
\alpha_k(x_k)\big), \quad \text{for} \quad x_1 \in A_1, \ldots, x_k
\in A_k,
\end{equation}
is a $v$-gauge on $A$ and there exists a canonical identification of
$\gr(F)$-algebras
\[
\gr_\alpha (A) \,=\,
\gr_{\alpha_1} (A_1) \times \ldots \times \gr_{\alpha_k}(A_k).
\]
\end{proposition}
\begin{proof}
It is easily checked that $\alpha$ is a surmultiplicative $v$-value
function on $A$. For $\gamma \in \Gamma$, we have
$\alpha(x_1, \ldots , x_k) \geq \gamma$ if and only if
$\alpha_i(x_i) \geq \gamma$ for $i = 1,\ldots, k$. Hence
$A_{\geq \gamma}= A_{1,{\geq \gamma}} \times \ldots
\times A_{k,{\geq \gamma}}$. Similarly, $A_{>\gamma}= A_{1,{> \gamma}}
\times \ldots \times A_{k,{> \gamma}}$. Therefore,
\begin{equation} \label{identificacao}
\gr_\alpha (A) \,=\, \gr_{\alpha_1} (A_1) \times \ldots \times
\gr_{\alpha_k}(A_k).
\end{equation}
By counting dimensions, we have
\[
[\gr_\alpha(A): \gr(F)]  \,=\,  \sum_{i=1}^k [\gr_{\alpha_i}(A_i):\gr(F)].
\]
Since $[\gr_{\alpha_i}(A_i):\gr(F)] = [A_i:F]$ for $i = 1,\ldots, k$, it
follows that $[\gr_{\alpha}(A):\gr(F)] = [A:F]$. Finally, the projections
of a homogeneous nilpotent two-sided ideal of $\gr_\alpha(A)$ are
homogeneous nilpotent two-sided ideals of each $\gr_{\alpha_i}(A_i)$,
which is trivial since $\gr_{\alpha_i}(A_i)$ is assumed semisimple.
Therefore, $\gr_\alpha(A)$ is also semisimple.
\end{proof}

Every $v$-gauge on $A$ as in Prop.~\ref{gaugesemisimple} is given by the
formula \eqref{gaugeminsemisimple}; more precisely, if
${\beta: A \rightarrow \Gamma \cup \{\infty\}}$ is a $v$-gauge, then for
each $i=1,\ldots, k$ the map $\beta_i\colon
A_i \rightarrow \Gamma \cup \{\infty\}$ defined by
\begin{equation}\label{betai}
\beta_i(x) \,=\, \beta(0, \ldots,0,x,0, \ldots,0) \qquad (x
\text{ in the $i$-th position})
\end{equation}
is a $v$-gauge on $A_i$ and
\[
\beta(x_1, \ldots , x_k) \,=\, \min\big(\beta_1(x_1), \ldots,
\beta_k(x_k)\big), \quad \text{for} \quad x_1 \in A_1,
\ldots, x_k \in A_k,
\]
(see \cite[Prop.~1.6]{TW1} for a proof).

\subsection{Composition of gauges}

Let $v\colon F\to\Gamma\cup\{\infty\}$ be a valuation on a field $F$,
where $\Gamma$ is a divisible totally ordered abelian group, and let $V$
be the valuation ring of $v$.  A valuation $v'$ on $F$ is said to be 
equivalent to $v$ if its valuation ring is also $V$.  Recall that
a valuation $w$ on $F$ is said to be a {\it coarsening} of $v$ if
its valuation ring $W$ contains $V$.  When this occurs, the maximal 
ideal $J(W)$ of~$W$ is a prime ideal of $V$, and $W$ is the 
localization $V_{J(W)}$ of $V$ at $J(W)$.  As is well-known,
the map $w \mapsto J(W)$ gives a one-to-one 
correspondence between the equivalence classes of coarsenings
of $v$ and the set of prime ideals of $V$.  Given $v$ and $w$, the 
ring $U = V/J(W)$ is a valuation ring of $\ov F^{\,w} = 
W/J(W)$, and a valuation $u$ on $\ov F^{\, w}$ with ring
$U$ is called the {\it residue valuation} determined by 
$v$ and $w$; we sometimes denote  
$u$ by $v/w$.  Its residue field is $\ov {\ov F^{\,w}}^{\,u}
= \ov F^{\,v}$.  From the perspective of $w$,
the valuation $v$ is called a {\it refinement} of $w$,
and $v$ is determined up to equivalence by 
$w$~and~$u$, since $V = \pi_w\inv(U)$ where $\pi_w\colon
W \to \ov F^{\,w}$ is the canonical projection; we call
$v$ the composite of $w$~and~$u$, and write $v = u*w$. 

We now look at coarsenings from the perspective of the 
valuation $v \colon F \to \Gamma\cup \{\infty\}$.  
Let $\Delta \subseteq \Gamma$ be any convex subgroup, 
i.e., a subgroup of $\Gamma$ satisfying 
 for all $\gamma\in \Gamma$, $\delta \in \Delta$,
if $0 \le \gamma \le \delta$ then 
$\gamma \in \Delta$.  So, $\Delta$
is a divisible group.  
Let $\Lambda = \Gamma/\Delta$, which is 
a divisible totally ordered abelian group under
the ordering induced from  $\Gamma$.
Let $\varepsilon\colon \Gamma \to \Lambda$ be the canonical map,
which we extend to $\Gamma\cup\{\infty\}$ by setting
$\varepsilon(\infty) = \infty$.  By composing
$v$ with $\varepsilon$, we obtain a  valuation $w$ on $F$
\[
w\, =\, \varepsilon\circ v\colon F\, \longrightarrow\,\Lambda\cup\{\infty\},
\]  
which is a coarsening of $v$.
Let $W$ be the valuation ring of $w$. The residue valuation
$u = v/w$ is  
$$
u \colon \ov F^{\,w} \,\longrightarrow\,\Delta\cup\{\infty\} \qquad 
\text{ given by }
\ \ u(c+J(W)) \,=\, \begin{cases} v(c) & \text{ if }w(c) = 0,\\
\infty &\text{ if } w(c) > 0,
\end{cases} 
\ \ \text{ for all } \ c\in W.
$$  
Note that all coarsenings of $v$ (up to equivalence) are obtainable 
this way:  If $w'$ is a coarsening of~
$v$, let $\Delta_0 = \{v(c) \mid w(c) = 0\}$, and let 
$\Delta'$ be the convex hull of $\Delta_0$ in $\Gamma$.  Then 
$w'$ is equivalent to the coarsening of $v$ determined by
$\Delta'$. 

 Let
\begin{equation}\label{eq:defdivhull}
\divh{\Gamma_v} \, = \Gamma_v \otimes_\Z \Q, 
\end{equation}
which is the
divisible hull of $\Gamma_v$.
Recall that the ordering on $\Gamma_v$ extends uniquely to 
$\divh{\Gamma_v}$, making the latter into a divisible totally 
ordered abelian group.  Moreover, 
we may view $\divh{\Gamma_v}$ as a subgroup of $\Gamma$, since
there is a unique monomorphism
$\divh{\Gamma_v} \to \Gamma$ extending  the inclusion 
$\Gamma_v \hookrightarrow \Gamma$.  Since valuations extending $v$ on 
algebraic extensions of $F$ all have value groups lying in 
$\divh{\Gamma_v}$,  there is generally no loss  for us to 
assume $\Gamma = \divh{\Gamma_v}$.  Recall that the convex subgroups
of $\divh{\Gamma_v}$ are in one-to-one correspondence with 
 the convex subgroups of $\Gamma_v$, which are in one-to-one 
correspondence with prime ideals of $V$, which are in turn 
in one-to-one correspondence with equivalence classes of coarsenings
of $v$.

Now let $A$ be a finite-dimensional $F$-algebra and let
$\alpha\colon A \rightarrow \Gamma\cup \{\infty \}$
be a surmultiplicative $v$-value function. With $\Delta$,
$\Lambda$,
$\varepsilon$, and $w$ as above, the composition of $\alpha$
with $\varepsilon$ yields a surmultiplicative $w$-value function
\[
\beta\,=\,\varepsilon \circ \alpha\colon A \,\longrightarrow\, \Lambda \cup
\{\infty \}.
\]
If $\alpha$ is a $v$-gauge, then $\beta$ is a $w$-gauge, by
\cite[Prop.~4.3]{TW2}. In this case, $\beta$ is called a \emph{coarsening}
of~$\alpha$. 
If $\Gamma_\alpha \subseteq\divh{ \Gamma_v}$,
then $w = \beta|_F$ determines $\Delta \cap \divh{\Gamma_v}$, 
which determines $\beta$; we then call $\beta$
the  {\it $w$-coarsening} of $\alpha$.

\begin{proposition}\label{coarsergaugering}
Let $\alpha$ be a $v$-gauge on a central simple $F$-algebra $A$ such
that $\Gamma_\alpha$
lies in the divisible hull $\divh{ \Gamma_v}$ of $\Gamma_v$. Let $w$ be
any valuation on $F$ which is a coarsening  of $v$, and let~
$W$~be the valuation ring of $w$. Let $\beta$ be the
$w$-coarsening of $\alpha$. Then the gauge ring $R_\beta$ is a central
localization of $R_\alpha$ by $P = J(W)$, that is,
$R_\beta = {R_\alpha\!\cdot \! V_P}$.
\end{proposition}
\begin{proof}
For each $x \in A$, if
$\alpha(x) \geq 0$, then $\beta(x) = \varepsilon (
\alpha(x)) \geq 0.$ Thus, $R_\alpha \subseteq R_\beta.$ Since
$W = V_P \subseteq R_\beta$, we have $R_\alpha\!\cdot\! V_P \subseteq R_\beta$. For
the reverse inclusion, let $b \in
R_\beta$. Since $\beta(x) > 0$ implies $\alpha(x) > 0$, we only have to
consider the case $\beta(x) = 0$. In this case, $\alpha(x) \in \Delta$.
Since $\Gamma_\alpha \subseteq \divh{\Gamma_v}
$,
 there exists a positive
integer $n$ such that
$n\alpha(x) \in \Gamma_v$. Thus, $n\alpha(x) \in \Gamma_v \cap \Delta$.
Let $c \in F$ such that $v(c) = n\hsp|\alpha(x)|.$ Hence, we have
$w(c)=\varepsilon(v(c))=0$. Since
$-n\hsp|\alpha(x)|\leq\alpha(x)\leq  n\hsp|\alpha(x)|$, it follows that
$\alpha(cx) = v(c) + \alpha(x) \geq 0.$  Therefore,
$x~=~(xc)c^{-1}~\in~{R_\alpha\!\cdot\!V_P}$. 
\end{proof}

 For the coarsening $\beta$ of the $v$-gauge $\alpha$ on $A$ as above,  
let $A_0^\beta$ be the degree zero part of $\gr_\beta(A)$, which is a
finite-dimensional semisimple $\overline{F}^{\,w}$-algebra. 
Thus,
\begin{align*}
A_0^\beta \, &= \, A_{\ge 0}^\beta\,\big/\, A_{>0}^\beta\\
&=\,\{x \in A\mid \alpha(x)\in \Delta  
\text{ or } 
\alpha(x) >\delta\text{ for all $\delta \in \Delta$}\}\,\big/\,
\{x\in A\mid \alpha(x) >\delta\text{ for all $\delta \in \Delta$}\}.
\end{align*}

For $u = v/w$,
we can define a
$u$-value function on $A_0^\beta$:
\begin{equation} \label{alphazero}
\alpha_{0}\colon
A^{\beta}_0\,\longrightarrow\,\Delta\cup\{\infty\}\quad\text{by} \quad
x+A^{\beta}_{>0}\,\mapsto\,
\begin{cases}
  \alpha(x)&\text{if $\beta(x)= 0$},\\
  \infty&\text{if $\beta(x)>0 $}.
\end{cases}
\end{equation}
This is well-defined by \cite[Lemma 4.1]{TW2}. 
Note that  $\Gamma_{\alpha_0} = \Delta\cap \Gamma_v$.    
For any $\delta \in  \Delta$
we have 
$$
(A_0^\beta)_\delta \, = 
\,(A_0^\beta)_{\ge\delta} \,\big/\, (A_0^\beta)_{>\delta} \,
= \,\big( A^\alpha_{\ge \delta}/ A^\beta_{>0}\big)  \big /
\big( A^\alpha_{>\delta}/ A^\beta_{>0}\big)
\, \cong \,A^\alpha_{\ge \delta}/ A^\alpha_{>\delta} \, = \,
A^\alpha_\delta.
$$
Hence,
\begin{equation} \label{gradedalphazero}
\gr_{\alpha_0}(A^\beta_0) \,\cong_g\,
\tbigoplus_{\delta \in   \Delta} A_\delta^\alpha,
\end{equation}
which is easily checked to be a graded ring isomorphism.
Thus, we may view $\gr_{\alpha_0}(A^\beta_0)$
as  a graded subring of $\gr_\alpha(A) =
\bigoplus_{\gamma \in \Gamma}
A^\alpha_{\gamma}.$

\begin{proposition} \label{gaugealphazero}
The value function $\alpha_0$ is a $u$-gauge on $A_0^\beta$.
\end{proposition}
\begin{proof} 
We  know from \cite[Prop.~4.3]{TW2} that $\alpha_0$ is a $u$-norm.
Moreover,  $\alpha_0$ is surmultiplicative since $\alpha$
is surmultiplicative.  Thus, it
remains only to prove that $\gr_{\alpha_0}(A_0^\beta)$ is graded
semisimple. Since $\gr_\alpha(A)$ is graded semisimple, it suffices
 to consider the case when $\gr_\alpha(A)$ is
graded simple. In this case, by the  graded version of Wedderburn's
Theorem
(see \cite[Prop.~1.3]{HW}), we can identify $\gr_\alpha(A)$ with
$\grEnd_\grD(\grV)$, where $\grD$ is a graded division ring and $\grV$ is
a finite-dimensional right graded $\grD$-vector space.
(The grading on $\grEnd_\grD(\grV)$ is given by:  for $\eta \in \Gamma$,
an $f \in \grEnd_\grD(\grV)$ is homogeneous of degree~$\eta$ if and only
if
$f(V_\gamma) \subseteq V_{\gamma + \eta}$ for all $\gamma \in \Gamma$.)
Let
$\{b_1, \ldots, b_n\}$ be a homogeneous $\grD$-base of $\grV$, and let~
$\gamma_i = \mathsl{deg}(b_i)$. Since $\Gamma_\grV$ is a finite union of
cosets of $\Gamma_\grD$, we can write $\{b_1, \ldots, b_n\}$ as a disjoint
union $\bigsqcup_{\ell = 1}^k S_r$, where $b_i$ and
$b_j$ are in the same $S_r$ if and only if
$\gamma_i - \gamma_j \in \Gamma_\grD + \Delta$.
Let~$\grV_i$ be the graded $\grD$-vector subspace of $\grV$ generated by
$S_i$.
We have a direct sum decomposition $\grV =
\bigoplus_{i=1}^k \grV_i$. By \eqref{gradedalphazero}, each
nonzero homogeneous
 $f \in \gr_{\alpha_0}(A_0^\beta)$ is a homogeneous element of
$\grEnd_\grD(\grV)$ satisfying also  $\deg(f) \in \Delta$. Hence,
$f$ maps each $\grV_i$ to itself. Thus, we can write
\begin{equation}
\gr_{\alpha_0}(A^\beta_0) \,=\, \tprod_{i=1}^k \grB_i \qquad \text{where}
\qquad \grB_i \,=\,
\tbigoplus_{\gamma \in \Delta}\grEnd_\grD(\grV_i)_\gamma.
\end{equation}
The proof is completed by showing that each $\grB_i$ is a simple graded
algebra.
Let
\[
\grD' \,=\, \tbigoplus_{\gamma \in \Delta} \grD_\gamma,
\]
which is a graded division subring of $\grD$. Now fix an index
$i\in \{1, \ldots, k\}$.
The grade set of $\grV_i$
lies in some  coset $\lambda_i + (\Gamma_\grD + \Delta)$. Let
\[
\grV'_i \,=\, \tbigoplus_{\gamma \,\in\, \lambda_i + \Delta} \grV_\gamma
\,\subseteq\,
\tbigoplus_{\gamma \in \lambda_i + \Gamma_\grD +\Delta} \grV_\gamma
\, = \,\grV_i,
\]
so $\grV_i'$  is a graded $\grD'$-vector subspace of $\grV_i$. Let
$\Gamma_\grD = \bigsqcup_{j \in J} \delta_j + (\Delta \cap \Gamma_\grD)$,
a disjoint union of cosets of $\Delta \cap \Gamma_D$. Then we also have
$\Gamma_\grD + \Delta = \bigsqcup_{j \in J} \delta_j + \Delta$, which is
again a disjoint union. For each $j \in J$, pick some $d_j \in
\grD_{\delta_j} \setminus \{0\}$. Then, as $\Gamma_{\grD'd_j} = \delta _j
+ (\Delta \cap \Gamma_\grD)$, we have
$$
\grD \, = \, \tbigoplus\limits_{j \in J} \grD'd_j.
$$
Also, as $\grV_i'\,d_j =
\bigoplus _{ \gamma \in \lambda_i + \delta_j + \Delta}\grV_\gamma$,
$$
\grV_i \, = \, \tbigoplus\limits_{j \in J} \grV_i'\,d_j \, = \,
\tbigoplus\limits_{j \in J}   \grV_i'\otimes _ {\grD'}\grD'd_j
\, = \,  \grV_i\otimes_{\grD'}\grD .
$$
 Any map in $\grB_i$ sends $\grV_i'$ to itself,
since $\Gamma_{\grB_i}\subseteq \Delta$.
 Thus, there exists a homomorphism of graded rings
$\psi\colon\grB_i \rightarrow \grEnd_{\grD'}(V_i')$,
given  by $g \mapsto g|_{V_i'}$.
This map $\psi$ has an inverse given by  sending $h \in
\grEnd_{\grD'}(\grV_i')$ to
$h\otimes\id_\grD  \in \grEnd_\grD(\grV_i'\otimes_{\grD'}\grD)$.
Note that if $h$ is homogeneous, then $\deg(h) \in \Delta$,
as $\Gamma_{\grV_i'} \subseteq \lambda_i + \Delta$.
Hence,  $h\otimes\id_\grD$ is homogeneous
with $\deg(h\otimes\id_\grD)= \deg(h) \in \Delta$,
showing that $h\otimes\id_\grD \in \grB_i$.
Therefore, $\grB_i \cong_g \grEnd_{\grD'}(\grV_i')$, which is a
simple graded algebra.
\end{proof}

\subsection{Defectlessness of valuations in semisimple algebras}
\label{subsec:defect}

We now develop the notion of defectlessness of a valuation $v$ on $F$ in a
finite-dimensional semisimple $F$-algebra $A$. 
We will see in \S4 that defectlessness is the condition
required for the existence of $v$-gauges on~$A$.
First, we review the
concept of defect on division algebras. 
A good reference for the division algebra case is 
\cite{M1}. Let $D$ be a finite-dimensional 
division algebra over a field~$F$. For a valuation $w$ on $D$, extending
a valuation $v$ on $F$, we have the \lq\lq fundamental inequality"
\begin{equation}\label{eq:divisioninequality}
\DIM{D}{F} \, \geq \,  \DIM{\,\overline{D}}{\overline{F}\,} \cdot
\IND{\Gamma_{w}}{\Gamma_v}.
\end{equation}
When  equality holds in \eqref{eq:divisioninequality}, we say the
valuation $w$ on $D$ is \emph{defectless over $F$}.
If $v$ extends uniquely to $Z(D)$, we define the \emph{defect}
$\partial_{D/F}$ of $D$ over $F$ by
\begin{equation}\label{defectdivision}
\partial_{D/F}\, =\ \frac{\DIM{D}{F}}{\DIM{\,\overline{D}}{\overline{F}\,}
  \cdot \IND{\Gamma_w}{\Gamma_v}}.
\end{equation}
In particular, $\partial_{D/Z(D)}$ is always defined; we call it
simply the \emph{defect of $D$} and use the simpler notation
$\partial_D$ for $\partial_{D/Z(D)}$. When $\partial_{D/F}$ is defined
we have
\begin{equation}
\label{eq:transdefect}
\partial_{D/F} \, =\,  \partial_D \cdot \partial_{Z(D)/F}
\end{equation}
by the transitivity formulas for residue degrees and  and indices
of value groups. In fact,
 for $\overline{p}=\charac
\overline{F}$, we have  $\partial_{D/F}=\overline{p}^\ell$ for some
integer $\ell\geq0$ if $\overline{p}\neq0$, and $\partial_{D/F}=1$
if $\overline{p}=0$. This result is known as \emph{Ostrowski's Theorem}
and was proved by Draxl in \cite[Th.~ 2]{D} for $v$ Henselian and in
general by Morandi in \cite[Th.~3]{M1}.  Hence in particular, 
if $\ov p = 0$ or if $\ov p \ne 0$ and $\ov p \nmid \DIM DF$,
then $\partial_{D/F} = 1$. 

\begin{def}
\label{def:vdeflessin}
  Let $A$ be a (finite-dimensional) semisimple $F$-algebra, let
  $(F_h,v_h)$ be a Henselization of $(F,v)$, and let
  $A_h=A\otimes_FF_h$. Since $F_h$ is a separable extension of $F$,
  the $F_h$-algebra $A_h$ is semisimple; hence, it has a
  decomposition into simple components 
  \[
  A_h \,\cong \,\mat_{n_1}(D_1)\times \ldots\times \mat_{n_r}(D_r)
  \]
  for some integers $n_1$, \ldots, $n_r$ and some division algebras
  $D_1$, \ldots, $D_r$ over $F_h$. We say that $v$ is
\emph{defectless in
    $A$} if for each $i=1$, \ldots, $r$ the unique valuation on $D_i$
  extending $v_h$ is defectless over~$F$, i.e., $\partial_{D_i/F_h}=1$
  for $i=1$, \ldots, $r$. 
\end{def}

It is clear from the definition that $v$ is defectless in $A$ if and
only if $v$ is defectless in each simple component of $A$, and that this
condition holds if and only if $v_h$ is defectless in $A_h$. We single
out two particular cases:
\begin{itemize}
\item[$\scriptstyle{\bullet}$]
  If $K$ is a finite-degree field extension of $F$ and $v_1, \ldots, v_n$
are all the extensions of $v$ of $F$ to~$K$, then $v$ is
  defectless in $K$ if and only if equality holds in the Fundamental
  Inequality, i.~e.
  \[
  \DIM KF \,=\, \tsum_{i=1}^n
\DIM{\,\overline{K}^{\,v_i}}{ \,\overline{F}^{\,v}\,}
\cdot | \Gamma_{v_i}: \Gamma_{v}|.
  \]
  This  follows readily from   Th.~\ref{tensorhens} in the 
Appendix. 
\item[$\scriptstyle{\bullet}$]
  If $A$ is a central simple $F$-algebra, then $v$ is defectless in
  $A$ if and only if the valuation on the division algebra associated
  to $A_h$ is defectless over $F_h$. In particular, if $v$ is
  defectless in $A$, it is also defectless in every algebra
  Brauer-equivalent to $A$.
\end{itemize}

In simple algebras that are not central, we have the following
reduction to the central case:

\begin{proposition}
  \label{prop:deflessredtocent}
  Let $A$ be a simple $F$-algebra, and let $v_1$, \ldots, $v_r$ be the
  valuations on $Z(A)$ extending $v$. The following conditions are
  equivalent:
  \begin{enumerate}
  \item[(a)]
  $v$ is defectless in $A$;
  \item[(b)]
  $v$ is defectless in $Z(A)$ and $v_1$, \ldots, $v_r$ are each defectless
  in $A$.
  \end{enumerate}
\end{proposition}

\begin{proof}
  For each $i=1$, \ldots, $r$, let $(Z_i,v_{i,h})$ be a Henselization
  of $(Z(A),v_i)$. From Th.~\ref{tensorhens} in the Appendix, we have
  \[
  Z(A)_h\,\cong \,Z_1\times\ldots\times Z_r.
  \]
  Since $Z(A_h)=Z(A)_h$, the number of prime components of $A_h$ is
  $r$, and we have division algebras $D_1$, \ldots, $D_r$ with centers
  $Z_1$, \ldots, $Z_r$ respectively such that 
  \[
  A_h\,\cong\, \mat_{n_1}(D_1)\times\ldots\times \mat_{n_r}(D_r)
  \]
  for some integers $n_1$, \ldots, $n_r$. Now, we have
  $\partial_{D_i/F_h} = \partial_{D_i}\cdot \partial_{Z_i/F_h}$
  (cf.~\eqref{eq:transdefect}); hence $\partial_{D_i/F_h}=1$ if and
  only if $\partial_{D_i}=1$ and $\partial_{Z_i/F_h}=1$. The
  equivalence of (a) and (b) follows.
\end{proof}

When $\Gamma_F\cong\Z$, $v$ is defectless in any central simple
$F$-algebra and $v$ is defectless in a semisimple $F$-algebra $A$
if and only if it is defectless in $Z(A)$. 
Also, by  Ostrowski's Theorem, if $\charac\ov F=0$,
then $v$ is defectless in every
semisimple $F$-algebra;
if $\charac\ov F=\ov p\neq0$, then $v$ is defectless in every
central simple $F$-algebra $A$ whose index $\ind(A)$ is not
divisible by $\ov p$.

\begin{lemma}\label{lem:deflessinK}
Let $w$ be any coarsening of the valuation $v$ on $F$, and let
$K$ be a finite-degree field extension of $F$.
If $v$ is defectless in $K$, then $w$ is defectless in~$K$.
\end{lemma}
\begin{proof}
Let $w_1$, \ldots , $w_\ell$ be the extensions of $w$ to $K$.  For
$j \in\{ 1,\ldots,
\ell\}$, let $v_{1j}, v_{2j}$,  \ldots, $v_{k_jj}$ be the extensions of $v$ to $K$ that
are refinements of $w_j$.  Thus, the $v_{ij}$ are all the extensions of $v$ to
$K$.  Let $v/w$ denote the valuation on $\ov F^{\,w}$ induced by
$v$ and likewise $v_{ij}/w_j$  the valuation on $\ov K^{\,w_j}$ induced by
$v_{ij}$. Note that $v_{1j}/w_j, v_{2j}/w_j$, \ldots, $v_{k_jj}/w_j$ are all the extensions of $v/w$
to $\ov K^{\,w_j}$.
For each $i,j$ there is a commutative diagram of value groups
with exact rows:
$$
  \xymatrix{
  0\ar[r]
& \Gamma_{v/w}  \ar[r] \ar[d]
&\Gamma_{v} \ar[r]  \ar[d]
& \Gamma_{w}\ar[r] \ar[d]
& 0\\
  0\ar[r]
&  \Gamma_{v_{ij}/w_j} \ar[r]
& \Gamma_{v_{ij}} \ar[r]
& \Gamma_{w_j} \ar[r]
& 0
  }
$$
Because the map $\Gamma_{w} \to \Gamma_{w_j}$ is injective,
the Snake Lemma  yields  a short
exact sequence of cokernels of the
columns:
$$
  \xymatrix{
  0\ar[r]
& \Gamma_{v_{ij}/w_j}
\big/\Gamma_{v/w}\ar[r]
&\Gamma_{v_{ij}} \big/\Gamma_{v} \ar[r]
& \Gamma_{w_j}\big/\Gamma_{w}\ar[r]
& 0
  }
$$
Hence,
\begin{equation}\label{diag:valgpseq}
\big|\Gamma_{v_{ij}}\!:\! \Gamma_{v}\big| \, = \,
\big|\Gamma_{v_{ij}/w_j}\!: \Gamma_{v/w}\big| \,
\big|\Gamma_{w_j}\!:\! \Gamma_{w}\big|.
\end{equation}

Since $v$ is defectless in $K$, we have
$$
\DIM KF \, = \, \tsum_{j = 1}^\ell \tsum _{i = 1}^{k_j} \big[\, \ov K^{\,v_{ij}}
\!: \! \ov F^{\,v}\big] \, \big|\Gamma_{v_{ij}}\!:\! \Gamma_{v}\big|.
$$
Equation  \eqref{diag:valgpseq} together with the Fundamental Inequality
for each~$\ov K^{\,w_j}$ over~$\ov F^{\,w}$ and for $K$ over~$F$ then yield,
\begin{align*}
\DIM KF \, &= \,  \tsum_{j = 1}^\ell \tsum _{i = 1}^{k_j}
\big[\, \ov K^{\,v_{ij}}\!: \! \ov F^{\,v}\big] \,
\big|\Gamma_{v_{ij}/w_j}\!: \Gamma_{v/w}\big| \,
\big|\Gamma_{w_j}\!:\! \Gamma_{w}\big|\\
&=\,\tsum_{j = 1}^\ell\bigg( \tsum _{i = 1}^{k_j}
\big[\,\ov{\ov K^{\,w_j}}^{\,v_{ij}/w_j}: \ov{\ov F^{\,w}}^{\,v/w}\,\big] \,
\big|\Gamma_{v_{ij}/w_j}\!: \Gamma_{v/w}\big|
\bigg)
\big|\Gamma_{w_j}\!:\! \Gamma_{w}\big|\\
&\le \, \tsum_{j = 1}^\ell
\big[\, \ov K^{\,w_j}\!: \! \ov F^{\,w}\,\big] \,
\big|\Gamma_{w_j}\!:\! \Gamma_{w}\big| \ \le \ \DIM KF.
\end{align*}
The last inequality must therefore be an equality, showing that
$w$ is defectless in~$K$.
\end{proof}

\begin{proposition} 
\label{prop:coarsdefless}
Let $w$ be any coarsening of the valuation $v$ on $F$.
Let $A$ be a semisimple $F$-algebra. If $v$ is defectless in $A$,
then $w$ is defectless in $A$.
\end{proposition}

\begin{proof}
Assume first that $A$ is central simple over $F$.  Let
$(F_{h,v}, v_h)$ be a Henselization of $(F,v)$.
Let $w'$ be the valuation on $F_{h,v}$
with ring ${W\!\cdot\!V_h}$.  So, $w'$ is the 
extension of $w$ which is a coarsening of $v_h$.  Since $v_h$ is
Henselian,
its coarsening $w'$ is also Henselian, by \cite[Cor.~4.1.4, p.~90]{EP}, 
so
$(F_{h,v}, w')$ contains a Henselization
$(F_{h,w},w_h)$ of $(F,w)$.  It follows from \cite[Th.~2]{M1}
that
$w'$ is  inertial (=~unramified) over $w_h$.
Let $D_{h,v}$ (resp.~$D_{h,w}$) be the central division algebra over
$F_{h,v}$ (resp.~$F_{h,w}$) associated to $A\otimes_F F_{h,v}$
(resp.~$A\otimes_F F_{h,w}$). Since $v$ is defectless in $A$, $D_{h,v}$
is defectless for $v_h$;  it is then also defectless for the coarser
valuation $w'$ by \cite[Lemma~1]{M4}.  Then, by
\cite[Remark~3.4]{JW} 
applied to the inertial extension $(F_{h,v}, w')$ of $(F_{h,w}, w_h)$, 
$D_{h,w}$ is defectless for $w_h$.
Hence, $w$ is defectless in $A$, as desired. 

Now assume only that $A$ is simple.  Let $K = Z(A)$, and let $v_1,$
 \ldots, $v_r$ be the extensions of $v$ to~$K$, and $w_1,$ \ldots,
$w_\ell$ the extensions of $w$ to $K$.  For $i \in \{ 1, \ldots, r\}$,
let $j(i) \in \{ 1, \ldots, \ell\}$ be the index such that $w_{j(i)}$
is the $w$-coarsening of $v_i$.  Since $v$ is defectless in
$A$, $v$ is defectless in $K$ and each $v_i$ is defectless in $A$.
 Then, $w$ is defectless in $K$ by
Lemma~\ref{lem:deflessinK}, and each $w_{j(i)}$ is
defectless in $A$ by the central simple case just considered.
  Since $w_{j(1)},$ \ldots, $w_{j(r)}$ are all the extensions of
$w$ to~$K$, it follows by Prop.~\ref{prop:deflessredtocent}
that $w$ is defectless in $A$.
This completes the proof for $A$~simple, and the general case for $A$
semisimple follows easily by considering the simple components of $A$.
\end{proof}

The following result says that defectlessness is a necessary condition
for the existence of a $v$-gauge on an arbitrary semisimple $F$-algebra.
We will see in Th.~\ref{existenceminimalgauges} below that this necessary condition is also
sufficient.

\begin{proposition}\label{prop:gaugeimpliesdefectless}
Let $A$ be any semisimple $($finite-dimensional$)$ $F$-algebra.
 If $A$ has a $v$-gauge, then $v$ is defectless in $A$.
\end{proposition}

\begin{proof}
$A$ has a $v$-gauge if and only if each simple component of $A$
has a $v$-gauge, by Prop.~\ref{gaugesemisimple}.
Also, by definition,
$v$ is defectless in $A$ if and only if $v$ is defectless in each simple
component of~$A$.  Thus, we may assume that $A$ is simple.
Let $K= Z(A)$, and let $v_1,$ \ldots, $v_r$ be the extensions of~
$v$ to the field $K$.  Let $(F_h, v_h)$ be a Henselization of
$(F,v)$, and let $(K_{h,v_i}, v_{i,h})$ be a Henselization of~
$(K,v_i)$.  Let $\alpha$ be a $v$-gauge on $A$. Then the restriction
$\alpha|_K$ is a surmultiplicative $v$-norm on~$K$ with
$\gr_{\alpha|_K}(K)$ graded semisimple, since it is a central graded
subalgebra of the graded semisimple algebra $\gr_\alpha(A)$.
Hence $\alpha|_K$ is a $v$-gauge on $K$, so $v$ is defectless in
$K$ by \cite[Cor.~1.9]{TW1}.
Moreover, $\alpha\otimes v_h$ is a
$v_h$-gauge on $A\otimes_FF_h$ and
\[
A\otimes_FF_h\ \cong\ (A\otimes_{K}K_{h,v_1}) \times\ldots\times
(A\otimes_{K}K_{h,v_r}).
\]
By Prop.~\ref{gaugesemisimple} each
$A\otimes_{K}K_{h,v_i}$ carries a gauge, hence the corresponding
central division algebra over~$K_{h,v_i}$ is defectless by
\cite[Th.~3.1]{TW1}.
\end{proof}

\section{Gauges on simple algebras}

Throughout this section, let $v$ be an arbitrary valuation on a field $F$,
and let $A$ be a
(finite-dimensional) simple $F$-algebra. Let $K$ be the center of $A$,
so $K$ is a finite-degree field extension of $F$.  Let $v_1, \ldots, v_r$
be all the extensions of $v$ to $K$, and let $V_i$ be the
valuation ring of $v_i$.  Our goal
is to characterize $v$-gauges on $A$ in terms of $v_i$-gauges for
$i = 1,\ldots, r$.  This will be achieved in Th.~\ref{thm:gaugeismin}
and Th.~\ref{thm:compatible}.

\begin{proposition}\label{prop:uniqueext}
Let $L$ be a field with $F\subseteq L \subseteq K$, and suppose
$v$ has a unique extension to a valuation $v_L$ of $L$.  If $A$ has a
$v$-gauge
$\alpha$, then $\alpha|_L = v_L$, which is a $v$-gauge on $L$, and
$\alpha$ is
a $v_L$-gauge.  Thus, whenever $v_L$ is a $v$-gauge on~$L$,
the $v$-gauges on $A$
are the same as the $v_L$-gauges on $A$.
\end{proposition}

\begin{proof}
The restriction $\alpha|_L$ of $\alpha$ to $L$ is clearly
surmultiplicative, and is a $v$-norm on~$L$ by
\cite[Prop.~2.5]{RTW} since $L$ is an $F$-subspace of $A$.
Moreover, as $L\subseteq Z(A)$, we have $\gr_{\alpha\rvert_L}(L) \subseteq
Z(\gr_\alpha(A))$.  Because $\gr_\alpha(A)$ is graded semisimple, it
contains no nonzero central homogeneous nilpotent elements.  Therefore,
the commutative $\gr(F)$-algebra $\gr_{\alpha|_L} (L)$ is semisimple,
and hence $\alpha|_L$ is a $v$-gauge on $L$.  Because $v_L$ is the unique
extension of $v$ to~$L$, \cite[Cor.~1.9]{TW1} shows that
$\alpha|_L = v_L$. Hence, for $c\in L^\times$, we have
$$
\alpha(c\inv) \, = \, v_L(c\inv) \,=\, - v_L(c) \,=\, -\alpha(c).
$$
So, a short computation (cf.~ \cite[Lemma~1.3]{TW1}) shows that
 $\alpha(ca) = \alpha(c) + \alpha(a)$
for all $a\in A$.  This proves that the $v$-value function~$\alpha$ on
$A$ is actually a $v_L$-value function.  Since  $v_L$ and $\alpha$
are $v$-norms, we have
\begin{align*}
\DIM LF \, \DIM{\gr_\alpha(A)}{\gr_{v_L}(L)}\, &=\,
\DIM{\gr_{v_L}(L)}{\gr_v(F)}\, \DIM{\gr_\alpha(A)}{\gr_{v_L}(L)}\, \\
&= \, \DIM{\gr_\alpha(A)}{\gr_v(F)} \, = \, \DIM AF \, = \,
\DIM LF \,\DIM AL.
\end{align*}
Hence, $\DIM{\gr_\alpha(A)}{\gr_{v_L}(L)} = \DIM AL$, showing that
$\alpha$ is a $v_L$-norm.  The other conditions needed for $\alpha$ to
be a $v_L$-gauge hold because it is a $v$-gauge.  Conversely,
whenever $v_L$ is a $v$-gauge (hence a $v$-norm) and $\beta$ is a
$v_L$-gauge on $A$, then
$$
\DIM{\gr_\beta(A)}{\gr_v(F)} \, = \, \DIM{\gr_\beta(A)}{\gr_{v_L}(L)} \,
\DIM{\gr_{v_L}(L)}{\gr_v(F)}\, = \, \DIM AL \, \DIM LF \, = \, \DIM AF,
$$
so $\beta$ is also a
$v$-norm and hence a $v$-gauge on $A$.
\end{proof}

\begin{theorem}\label{thm:gaugeismin}
Let $\alpha$ be a $v$-gauge on $A$.  Then there exist $v_i$-gauges $\alpha_i$ on
$A$ for $i = 1,\ldots, r$ such that
\begin{equation*}
\alpha(a) \, =\,  \min\big(\alpha_1(a), \ldots, \alpha_r(a)\big) \quad \text{for
all } a\in A.
\end{equation*}
Furthermore,
\begin{equation}\label{eq:grisocomps}
\gr_\alpha(A) \,\cong_g\, \gr_{\alpha_1}(A) \times \ldots\times
\gr_{\alpha_r}(A).
\end{equation}
Hence, the semisimple $\ov F$-algebra $\gr_\alpha(A)_0$ has at least
$r$ simple components. Moreover, the $\gr_{\alpha_i}(A)$ are the
graded simple components of $\gr_\alpha(A)$ and
$$
\DIM{\gr_{\alpha_i}(A)}{\gr(F)} \ = \ \DIM{A}{K} \ \DIM{\gr_{v_i}(K)}{\gr(F)}.
$$
\end{theorem}

We call the $\alpha_i$ of the theorem the {\it $v_i$-component} of
$\alpha$,
for $i = 1,\ldots, r$.  We will see in Cor.~\ref{cor:compunique}
 below that
the $\alpha_i$ are uniquely determined by $\alpha$.

\begin{proof}
 Let $(F_h, v_h)$ be the Henselization of $(F,v)$, and let
$$
B \, = \, A\otimes_F F_h \qquad\text{and}\qquad L  \,= \,Z(B)
\,=\, K\otimes_F F_h.
$$
Then, $L$ is a direct product of finitely many fields. Let
$e_1, \ldots, e_r$ be the primitive idempotents of~$L$, so
$$
L\, = \, L_1\times\ldots\times L_r\qquad \text{where} \qquad L_i \,=\, e_iL,\
$$
and each $L_i$ is a field. The $L_i$ are indexed by the $v_i$, as
we will explain below.
Since ${B \cong A\otimes _K L}$, the ring
$B$ is a product of  algebras
$$
B\,  = \, B_1\times \ldots \times B_r\qquad \text{where}\qquad
B_i \,= \,e_iB\, \cong A \otimes_K L_i.
$$
So, each $B_i$ is a central simple $L_i$-algebra.
We identify $K, \, A, \, F_h$ with their isomorphic copies
$K\otimes 1,\ A\otimes 1,\  1\otimes F_h$ in $B$. But, we do
{\it not} identify them with their isomorphic copies
$e_iK,\, e_iA, \, e_iF_h$ in $B_i$.
For each $i$ we have canonical inclusions
$$
p_i\colon A \hookrightarrow B_i, \ \ a \mapsto e_i(a\otimes 1) \ \ \quad
\text{and} \ \ \quad q_i\colon F_h \hookrightarrow B_i, \ \
c \mapsto e_i(1 \otimes c).
$$
Thus, $B_i$ has
subalgebras $p_i(A)$, $p_i(K)$, and~$L_i$, with
\begin{equation*}
B_i \, = \, p_i(A) \otimes_{p_i(K)}L_i \,\cong \, A\otimes_K L_i,
\qquad \text{hence,} \qquad \DIM{B_i}{L_i}\, = \, \DIM{A}{K}.
\end{equation*}

Each field $L_i$ is a compositum of fields, $L_i = {p_i(K) \!\cdot\! 
q_i(F_h)}$.  The Henselian valuation $v_h$ on $F_h$ has an isomorphic
(Henselian) valuation $v_h\circ q_i\inv$ on $q_i(F_h)$, which extends
uniquely to a Henselian valuation $w_i$ on $L_i$.  This pulls back to
a valuation $w_i\circ p_i$ on $K$ which extends $v$ on $F$.  The proof of
 Th.~\ref{tensorhens} in the Appendix shows
that the valuations $w_1\circ p_1,$ \ldots,
$w_r \circ p_r$ are all distinct and are all the extensions of $v$ to $K$.
Thus, after renumbering the $e_i$ if necessary, we can assume
$w_i\circ p_i= v_i$ for $i = 1,\ldots,r$.  That is,
\begin{equation}
\label{eq:videf}
v_i(d) \,=\, w_i(e_i(d\otimes 1))\qquad \text{for all $d\in K$}.
\end{equation}
From Th.~\ref{tensorhens}, we have also
that $(L_i, w_i)$ is a Henselization of~$(K, v_i)$.

 Let $\beta=
\alpha \otimes v_h$, which is a $v_h$-gauge on $B$ with
\begin{equation}\label{eq:gralpha'alpha}
\gr_{\beta}(B) \,\cong_g \,\gr_\alpha(A) \otimes_{\gr (F)}\gr(F_h)
\,\cong_g \,\gr_\alpha(A)
\end{equation}
 by \cite[Cor.~1.26]{TW1}.  Let
$\beta_i = \beta|_{B_i}$, which is a $v_h$-gauge
on $B_i$ via the embedding $q_i\colon F_h \to B_i$.  By
Prop.~\ref{gaugesemisimple},
\begin{equation}\label{eq:alpha'}
\beta(b) \, =\, \min_{1\le i\le r}\big( \beta_i(e_ib)\big) \qquad
\text{for all } \  b\in B,
\end{equation}
and
\begin{equation}\label{eq:gralpha'}
\gr_{\beta}(B)\,\cong_g\, \tprod\limits_{i=1}^r\gr_{\beta_i}(B_i).
\end{equation}
Since $w_i$ is the unique extension of the Henselian valuation $v_h$ to
$L_i$, Prop.~\ref{prop:uniqueext} above shows that each~$\beta_i$ is a
$w_i$-gauge. The structure theorem \cite[Th.~3.1]{TW1} for gauges on simple algebras
with respect to Henselian valuations shows that
$\beta_i$ is an End-gauge as in Ex.~\ref{ex;endgauge}; hence,
$\gr_{\beta_i}(B_i)$ is graded simple.

Define $v$-value functions $\alpha_1,\ldots, \alpha_r$ on $A$ by
$$
\alpha_i(a) \, = \, \beta(p_i(a)) \, =\, \beta_i(e_i(a\otimes 1)).
$$
Then, each $\alpha_i$ is surmultiplicative as $\beta_i$ is
surmultiplicative, and for  all $a\in A$,
\begin{equation}\label{eq:alphaformula}
\alpha(a) \, = \, \beta(a\otimes 1)  \,= \, \min_{1\le i\le r}
\big( \beta_i(e_i(a\otimes 1))\big)  \,= \, \min_{1\le i \le r}
\big(\alpha_i(a)\big).
\end{equation}
Furthermore, as $\beta_i$ is a $w_i$-value function, for all
$c\in K$ and $a\in A$,
\begin{align*} 
\alpha_i(ca) \,& = \,\beta_i(e_i(ca\otimes 1))  \, =
\, \beta_i\big( e_i(c\otimes 1) \cdot e_i(a\otimes 1)\big)\\ 
&=\, w_i(e_i(c\otimes 1)) \,+\, \beta_i(e_i(a\otimes 1)) \,
=\, v_i(c) + \alpha_i(a).
\end{align*}
Thus, $\alpha_i$ is a $v_i$-value function on $A$. The following
diagram shows the algebras related to $B_i$ and the associated
value functions being considered here.
\begin{equation}\label{diag:Bi}
\begin{split}
\xymatrix{
&B_i,\beta_i\\
A,\alpha_i\ar@{ (->}[ur]^{p_i}&&L_i,w_i\ar@{ - }[ul] \\
&K,v_i\ar@{ - }[ul] \ar@{ (->}[ur]^{p_i}&&F_h,v_h\ar@{_{(}->}[ul]_{q_i}
}
\end{split}
\end{equation}

Now, for all $\gamma\in
\Gamma$, the definition of the
$\beta_i$ and \eqref{eq:alpha'} and \eqref{eq:alphaformula}
above show that for each $i$ we have a
commutative diagram
\begin{equation*} 
\begin{CD}
A^{\hsp{\scriptstyle\alpha}\phantom{\beta}}_{\scriptscriptstyle{\ge}
\scriptstyle\gamma}
@>>> A^{\alpha_i}_{\scriptscriptstyle{\ge} \scriptstyle\gamma}\\
@VVV @VVV \\
B^{\beta}_{\scriptscriptstyle{\ge} \scriptstyle\gamma} @>{e_i\cdot}>>
B_{i,\,\scriptscriptstyle{\ge} \scriptstyle\gamma}^{\,\beta_i}
\end{CD}
\qquad\text{ given by }\qquad
\begin{CD}
a @>>> a\\
@VVV @VVV \\
a\otimes 1 @>>> e_i(a\otimes 1)
\end{CD}.
\end{equation*}
There is a corresponding commutative diagram with $>\gamma$ replacing
$\ge \gamma$, hence an induced commutative diagram of corresponding
factor groups;   these together yield a commutative
diagram of graded $\gr(F)$-algebra homomorphisms:
\begin{equation}\label{cd:graded}
\begin{CD}
\gr_\alpha(A) @>>>\tprod\limits_{i=1}^r \gr_{\alpha_i}(A) \\
@V{\cong}VV @VVV\\
\gr_{\beta}(B) @>{\cong}>>\tprod\limits_{i=1}^r\gr_{\beta_i}(B_i)
\end{CD}
\end{equation}
Here, the top map is injective by \eqref{eq:alphaformula}; the left map
is the  isomorphism of \eqref{eq:gralpha'alpha}; the right map is
injective since  for each $i$, the definition of $\alpha_i$ shows that
${\gr_{\alpha_i}(A)
\to \gr_{\beta_i}(B)}$ is injective; and the bottom map is the
 isomorphism of \eqref{eq:gralpha'}.  Therefore, all the maps in this
diagram must be isomorphisms.  Hence, for each~$i$,
${\gr_{\alpha_i}(A) \cong_g
\gr_{\beta_i}( B_i)}$, which is graded simple as
we saw above after~\eqref{eq:gralpha'}.  Since
the Henselization $(L_i,w_i)$ (resp.~$(F_h,v_h)$) is
an immediate extension of
$(K, v_i)$ (resp.~$(F,v)$), we have
$$
\DIM{L_i}{F_h} \,\hsp \ge\ \DIM{\gr_{w_i}(L_i)}{\gr_{v_h}(F_h)}
\ = \ \DIM{\gr_{v_i}(K)}{\gr(F)}.
$$
So, as $\beta_i$ is a $v_h$-norm,
\begin{align}\label{eq:gralphai}
\begin{split}
\DIM{\gr_{\alpha_i}(A)}{\gr_{v_i}(K)}\ \DIM{\gr_{v_i}(K)}{\gr(F)} \ &= \
\DIM{\gr_{\alpha_i}(A)}{\gr(F)} \\
&= \ \DIM{\gr_{\beta_i}(B_i)}{\gr_{v_h}(F_h)}\ = \ \DIM{B_i}{F_h} \\
&= \ \DIM{B_i}{L_i}\, \DIM{L_i}{F_h} \ = \ \DIM{A}{K}\,\DIM{L_i}{F_h}\\
&\ge \ \DIM{A}{K}\ \DIM{\gr_{v_i}(K)}{\gr(F)},
\end{split}
 \end{align}
and hence $\DIM{\gr_{\alpha_i}(A)}{\gr_{v_i}(K)} \ge \DIM AK$.
Since the reverse inequality holds for any $v_i$-value function, we have
$\DIM{\gr_{\alpha_i}(A)}{\gr_{v_i}(K)} = \DIM AK$.  Hence,
$\alpha_i$ is a $v_i$-norm on $A$; with the
graded simplicity noted above, this yields that $\alpha_i$
is a $v_i$-gauge.  Furthermore, equality holds in
\eqref{eq:gralphai}, yielding
 $\DIM{\gr_{\alpha_i}(A)}{\gr(F)} = \DIM AK \,\DIM{\gr_{v_i}(K)}{\gr(F)}$.
The isomorphism for $\gr_\alpha(A)$ in the theorem is the top
isomorphism in the commutative diagram~\eqref{cd:graded}.
Since each $\gr_{\alpha_i}(A)$ is graded simple,
the $\gr_{\alpha_i}(A)$~are the graded simple components of
$\gr_\alpha(A)$.
Because  $\gr_\alpha(A)$ has $r$ graded simple components,
its degree zero part must have at least $r$ simple components.
\end{proof}

Let $v$ be a valuation on some division algebra $D$, and let
$\alpha$, $\beta$, $\eta_1,$ \ldots,
$\eta_r$ be $v$-value functions on a finite-dimensional $D$-vector space
$M$. We write
${\alpha \le \beta}$\  if $\alpha(z) \le \beta (z)$ for all
$z\in M$.  Likewise, we write $ \alpha = \min\big(\eta_1,
\ldots, \eta_r\big)$ if  ${\alpha(z) = \min\big(\eta_1(z), \ldots,
\eta_r(z)\big)}$ for all $z\in M$.  It is easy to construct examples
of $v$-norms $\alpha,$ $\beta$ on $M$
with $\alpha \le \beta$ and $\alpha \ne \beta$.  By contrast, we will see
in Cor.~\ref{cor:gaugeineq} below 
for gauges on semisimple algebras that $\alpha\le \beta$ implies
$\alpha = \beta$.  This will be proved by showing a minimality property
 characterizing the components of a gauge on a simple algebra.

\begin{lemma}\label{surmultprod.lem}
 Let $A$ and $B$ be $F$-algebras with respective surmultiplicative
$v$-value
functions $\alpha$ and~$\beta$. Suppose there is an $F$-algebra
homomorphism $f\colon A \rightarrow B$ such that
\begin{equation}\label{inequalityf}
 \beta(f(a))\, \geq \,\alpha(a) \quad \text{for all } a \in A.
\end{equation}
 Then $f$ induces a graded $\gr(F)$-algebra homomorphism
 \[
 \widehat{f}\,\colon \gr_\alpha(A) \,\longrightarrow \,\gr_\beta(B)
 \]
 such that  $\widehat{f}(\tilde x) =
f(x) + B_{> \alpha(x)} \in B_{\alpha(x)}$ for all $x \in A$. Moreover,
$\widehat{f}$ is injective if and only if equality holds in
\eqref{inequalityf}.
 \end{lemma}
\begin{proof}
For any $\delta \in \Gamma$, we have $f(A_{\geq \delta}) \subseteq
B_{\geq \delta}$ and $f(A_{> \delta}) \subseteq B_{> \delta}$, hence
$f$ induces a map ${(\widehat{f})_\delta\colon A_\delta
\rightarrow B_\delta}$ given by
$(\widehat{f})_\delta(x+ A_{> \delta}) = f(x) + B_{> \delta}.$ Then set
$\widehat{f} = \bigoplus_{\delta \in \Gamma} (\widehat{f})_\delta
\colon \gr_\alpha(A) \to \gr_\beta(B)$.
Clearly, $\widehat{f}$ is a graded $\gr(F)$-vector space homomorphism.
To see that $\widehat f$ is multiplicative,
 it suffices to check this for homogeneous elements.  That is,
for any nonzero $x,y \in A$ we need
\begin{equation}\label{fxytilda}
\widehat{f}({\tilde x}\, {\tilde y}) \, =\,
 \widehat{f}(\tilde x) \widehat{f}(\tilde y).
\end{equation}
When the left expression in \eqref{fxytilda} is nonzero, it equals
$\widetilde{f(xy)}$, and when the right expression is nonzero it equals
$\widetilde{f(x)} \widetilde{f(y)} = \widetilde{f(x)f(y)} =
\widetilde{f(xy)}$. So, equality indeed holds in \eqref{fxytilda} when
each side is nonzero. Now we have
\begin{equation}\label{betafxy}
\beta(f(xy))\, \geq\, \alpha(xy) \,\geq\, \alpha(x) + \alpha(y)
\end{equation}
and
\begin{equation}\label{betafxy2}
\beta(f(xy)) \,=\,  \beta(f(x)f(y)) \,\geq\ \beta(f(x))+ \beta(f(y))
\,\geq \,\alpha(x) + \alpha(y).
\end{equation}
The left expression in \eqref{fxytilda} is nonzero if and only if
$\alpha(xy) = \alpha(x) + \alpha(y)$ (so
$\tilde x \,\tilde y = \widetilde{xy} \neq 0)$ and
$\beta(f(xy)) = \alpha(xy)$, i.e., equality holds
throughout~\eqref{betafxy}. The right expression in \eqref{fxytilda} is
nonzero if and only if $\beta(f(x)) = \alpha(x)$ and
$\beta(f(y)) = \alpha(y)$, and
$\beta(f(x)f(y)) = \beta(f(x)) +  \beta(f(x))$, i.e., equality holds
throughout \eqref{betafxy2}. Each of these conditions holds if and only if
$ \beta(f(xy)) = \alpha(x) + \alpha(y)$. Thus, we have equality in
\eqref{fxytilda} in all cases. Now note that $\widehat{f}(\tilde x) = 0$
if and only if $\beta(f(x)) > \alpha(x)$. Since $\ker(\widehat{f})$
is a homogeneous two-sided ideal of $\gr_\alpha(A)$, the stated condition
for injectivity of $\widehat{f}$ holds.
\end{proof}

\begin{theorem}\label{thm:gaugeineq}
Let $\alpha$ be any $v$-gauge on the simple
$F$-algebra $A$, and, as
 in Th.~\ref{thm:gaugeismin} above,
let $\alpha_i$
be the $v_i$-component of $\alpha$
for $i = 1, \ldots, r$.
 Suppose $\eta$ is a $v_k$-gauge on~$A$, for some $k$.
If~$\alpha \le \eta$,  then ${\eta = \alpha_k}$.
\end{theorem}

\begin{proof}
Pick a homogeneous base $\big(b_1, \ldots, b_n\big)$ of
$\gr_{\alpha_k}(A)$
as a graded $\gr_{v_k}(K)$-vector space, and let $\gamma_j = \deg(b_j)$.
Then, pick $a_1, \ldots, a_n\in A$ with each
${\tilde {a_j}^{\!\alpha}\mapsto
(0,\ldots, 0, b_j, 0, \ldots, 0)}$ ($b_j$ in the $k$-th position) under
the
isomorphism $\gr_\alpha(A) \cong_g \prod_{i = 1}^r \gr_{\alpha_i}(A)$ of
Th.~\ref{thm:gaugeismin}.  This means that for all $j = 1, \ldots, n$
and $i = 1,\ldots, r$,
$$
\alpha_i(a_j) \,>\,\gamma_j \ \text{ for } i\ne k, \quad \alpha(a_j)
\,=\, \alpha_k(a_j)
\,=\, \gamma_j, \quad \text{ and }\quad \tilde {a_j}^{\alpha_k} =
b_j\text{ in }
\gr_{\alpha_k}(A).
$$
Since $\big(\tilde {a_1}^{\alpha_k}, \ldots, \tilde {a_n}^{\alpha_k}\big)$
is a homogeneous base of the graded $\gr_{v_k}(K)$-vector space 
$\gr_{\alpha_k}(A)$,  ${(a_1, \ldots, a_n)}$ is a splitting base of the $K$-vector
space $A$ for the $v_k$-gauge $\alpha_k$ (see the comments  
preceding \eqref{eq:defRalpha} above).
  Furthermore,
${\alpha_k(a_j)  = \alpha(a_j) \le \eta(a_j)}$, for all $j$.  Therefore,
for any $a \in A$, writing $a = \sum_{j=1}^n a_jc_j$ with $c_j \in K$,
we have
\begin{align*}
\alpha_k(a) \, &= \, \min_{1\le j \le n}\big(\alpha_k(a_j) + v_k(c_j)\big)
\, \le \, \min_{1\le j \le n}\big(\eta(a_j) + v_k(c_j)\big) \\
&\le \, \eta \big(\tsum\limits_{j = 1}^n a_jc_j\big) \, = \, \eta(a).
\end{align*}
Thus, $\alpha_k \le \eta$ as $v_k$-gauges on $A$.

Since $\alpha_k \le \eta$,
Lemma~\ref{surmultprod.lem} (with $f = \id_A$) shows that there is a
well-defined graded
$\gr_{v_k}(K)$-algebra homomorphism $\varphi\colon
\gr_{\alpha_k}(A) \to \gr_\eta(A)$ given by
${\tilde a^{\alpha_k} \mapsto
a +A^\eta_{>\alpha_k(a)}\in A^\eta_{\alpha_k(a)}}$ for all $a\in A$.
But, as $\alpha_k$ is a $v_k$-gauge on the central simple $K$-algebra
$A$, \cite[Cor.~3.7]{TW1} shows that $\gr_{\alpha_k}(A)$
is a simple graded algebra.  Hence $\varphi$ must be injective.
Again by Lemma~\ref{surmultprod.lem} we have $\alpha_k = \eta$.
\end{proof}

\begin{corollary}\label{cor:compunique}
Let $\alpha$ be any $v$-gauge on
the simple $F$-algebra $A$,  and,
as
in Th.~\ref{thm:gaugeismin} above,
let
$\alpha_i$ be the
$v_i$-component of $\alpha$ for $i = 1,  \ldots, r$.
Suppose $\alpha = \min \big(\eta_1, \ldots,\eta_r\big)$
for some $v_i$-gauges $\eta_i$.  Then each $\eta_i = \alpha_i$.
\end{corollary}

\begin{proof}
For each $i$, we have $\alpha \le \eta_i$.  Hence, $\eta_i = \alpha _i$
by the preceding theorem.
\end{proof}

\begin{corollary}\label{cor:gaugeineq}
Let $\alpha$ and $\eta$ be $v$-gauges on a semisimple $F$-algebra
$C$.  If $\alpha \le \eta$, then $\alpha = \eta$.
\end{corollary}

\begin{proof}
It suffices to check this  for the restrictions of $\alpha$
and $\eta$ on the simple components of $C$.  Therefore, we may assume that
$C$ is a simple $F$-algebra.  Then, let $v_1, \ldots, v_r$ be all the
extensions of $v$ to $Z(C)$, and let $\alpha_i$ (resp.~$\eta_i$)
be the $v_i$-component of $\alpha$ (resp. $\eta$).  Since, for each $i$
we have
$\alpha \le \eta \le \eta_i$, Th.~\ref{thm:gaugeineq} shows that
$\alpha_i = \eta_i$.  Hence,
$$
\alpha\, = \, \min\big(\alpha_1, \ldots, \alpha_r\big)  \,=\, \min
\big(\eta_1, \ldots, \eta_r\big) \, = \, \eta.
$$
\end{proof}

\begin{corollary}\label{cor:coarscomponents}
Let $\alpha$ be a $v$-gauge on the simple $F$-algebra $A$,
with $v_i$-components~$\alpha_i$, for
${i = 1,\ldots, r}$, and suppose $\Gamma_\alpha$ lies in
the divisible hull of $\Gamma_v$.  Let $w$ be any
valuation on $F$ which is a coarsening of $v$, and let
$w_1, \ldots, w_\ell$ be all the extensions of $w$ to $K$.  Let $\beta$~be
the $w$-coarsening of $\alpha$, and let $\beta_j$ be the $w_j$-component
of $\beta$ for $j = 1,\ldots, \ell$.  For $i \in \{1,\ldots, r\}$, let
$j(i)\in \{1,\ldots, \ell\}$ be the index such that $w_{j(i)}$
is the $w$-coarsening of $v_i$ $($i.e., $W_{j(i)} = {W\!\cdot\! V_i})$.  
Then, $\beta_{j(i)}$~is the
$w_{j(i)}$-coarsening of $\alpha_i$.
\end{corollary}

\begin{proof}
Let $\alpha_{i,w}$ denote the $w_{j(i)}$-coarsening of $\alpha_i$.
Let $\Delta$ be the convex subgroup  of $\Gamma$ associated to~$w$.
Recall
that $\beta = \varepsilon\circ \alpha$, where $\varepsilon\colon
\Gamma \to \Gamma /\Delta$ is the canonical surjection. This $\varepsilon$
 is compatible with
the orderings on $\Gamma$ and $\Gamma/\Delta$.  Likewise, $w_{j(i)} =
\varepsilon \circ v_i$ and $\alpha_{i,w} = \varepsilon\circ \alpha_i$.
Now fix any $i$.
Since~$\alpha \le  \alpha_i$,
$$
 \beta\, = \, \varepsilon\circ \alpha \, \le \,
\varepsilon\circ \alpha_i \,=\, \alpha_{i,w},
$$
i.e., $\beta\le \alpha_{i,w}$. Since $\alpha_{i,w}$ is a
$w_{j(i)}$-gauge, Th.~\ref{thm:gaugeineq} shows that
$\alpha_{i,w} = \beta_{j(i)}$.
\end{proof}

We next determine when given $v_i$-gauges $\eta_1, \ldots, \eta_r$
yield a $v$-gauge as $\min\big(\eta_1, \ldots, \eta_r\big)$. For
 $i,j\in \{1,\ldots,r\}$ let $v_{ij}$ denote the finest common
coarsening
of $v_i$~and~$v_j$. That is, $v_{ij}$ is the valuation on $K$ with
associated valuation ring $V_{ij}
={V_i\!\cdot\!V_j}$.

\begin{theorem}\label{thm:compatible}\label{coarseningtheorem}
Suppose $v$ is defectless in the simple $F$-algebra $A$.
For $i = 1,\ldots ,r$,
let $\eta_i$ be a $v_i$-gauge
on $A$ with $\Gamma_{\eta_i}\subseteq \divh{\Gamma_v}$.
 Let $\alpha = \min\big(\eta_1, \ldots, \eta_r)$.
Then, $\alpha$ is a $v$-gauge on $A$ if and only if
$\eta_i$ and $\eta_j$ have the same $v_{ij}$-coarsening for all
pairs $i$, $j$.
When this occurs, each~$\eta_i$ is the $v_i$-component of~$\alpha$.
\end{theorem}

\begin{proof}  $\Rightarrow$  Suppose our $\alpha$ is a $v$-gauge
on $A$.  By Cor.~\ref{cor:compunique}, each $\eta_i$ is the
$v_i$-component  of $\alpha$.  Fix  indices $i, \, j$,  let
$w = v_{ij}|_F$, and let $\beta$ be the $w$-coarsening of $\alpha$.
The $v_{ij}$-coarsenings
of $\eta_i$ and $\eta_j$ must be the same,
since by Cor.~\ref{cor:coarscomponents} they coincide with the
$v_{ij}$-component
of $\beta$.

$\Leftarrow$  Suppose  each
$\eta_i$ and $\eta_j$ have the same $v_{ij}$-coarsening.
Now, $\alpha = \min\big(\eta_1, \ldots, \eta_r\big)$ is clearly
a surmultiplicative $v$-value function on $A$.  Consider the
graded $\gr(F)$-algebra homomorphism
\begin{equation}\label{gradedpsi}
\Psi\colon \gr_\alpha(A) \,\longrightarrow\, \tprod\limits_{i = 1}^r
\gr_{\eta_i}(A) \ \ \  \text{given by}\ \ \
a+ A_{>\alpha(a)}^{\alpha_{\phantom 1}} \,\mapsto\,
\big(a + A_{> \alpha(a)}^{\eta_1},
\ldots, a+A_{>\alpha(a)}^{\eta_r}\big).
\end{equation}
Then, $\Psi$ is well-defined and injective because 
$\alpha = \min\big(\eta_1, \ldots, \eta_r\big)$.
We will show below that $\Psi$ is an isomorphism.
It then follows that $\gr_\alpha(A)$ is semisimple,
as each $\gr_{\eta_i}(A)$ is semisimple.
Moreover,  $v$~is defectless in $K$ by Prop.~\ref{prop:deflessredtocent}
 since it is defectless in $A$, whence
\begin{align*}
\DIM{\gr_\alpha(A)}{\gr(F)} \, &= \,
\tsum\limits_{i = 1}^r\,\DIM{\gr_{\eta_i}(A)}{\gr(F)} \, = \,
\tsum\limits_{i=1}^r\,\DIM{\gr_{\eta_i}(A)}{\gr_{v_i}(K)}\,
\DIM{\gr_{v_i}(K)}{\gr(F)}\\
&= \tsum\limits_{i = 1}^r\, \DIM AK \, \DIM{\gr_{v_i}(K)}{\gr(F)}
\, = \, \DIM AK \,\DIM KF \, = \, \DIM AF.
\end{align*}
Hence, $\alpha$ is $v$-gauge.

To prove surjectivity of $\Psi$, we use the following approximation lemma,
for which we use the  notation: let $\Delta_{ij}$ be the convex
subgroup of the divisible hull $\Gamma$ of~$\Gamma_{v}$ associated to $v_{ij}$;
so $v_{ij}$ is a map $K\to \Gamma/\Delta_{ij}\, \cup \{\infty \}$. For any $\delta_1,\ldots, \delta_n\in \Gamma$ we say that the $n$-tuple $(\delta_1,$ \ldots, $\delta_n)$ is {\it compatible} in $\Gamma^n$ if for all $i,j$ we have $\delta_i - \delta_j \in \Delta_{ij}$.
\renewcommand{\qed}{\relax}
\end{proof}

\begin{lemma}
\label{lem:approx}
With the notation just above, fix some $k \in \{1, \ldots, n\}$, and let
$(\delta_1,$ \ldots, $\delta_n)$ be a compatible $n$-tuple in $\Gamma^n$
with $\delta_k = 0$.  Then, there is $c\in K^\times$ with
$v_i (c) >\delta_i$
for each $i \ne k$,  $v_k(c) = 0$, and   $v_k(c-1)>0$.
\end{lemma}

\begin{proof}
For each pair of indices $i$, $j$, we have
$$
0 \, \le \, \big||\delta_i|  - |\delta _j|\big| \, \le \,
|\delta _i - \delta_j|.
$$
Since $\delta_i -\delta_j\in \Delta_{ij}$ and $\Delta_{ij}$ is convex,
it follows that $|\delta_i| - |\delta_j| \in \Delta_{ij}$.
Furthermore, as $\Gamma/\Gamma_{v}$ is a torsion group, there is
$m \in \N$ with each $m|\delta_i| \in \Gamma_{v}$.  Hence,
each $m|\delta_i|\in \Gamma_{v_i}$ and $m|\delta_i|- m|\delta_j| \in
\Delta _{ij}$.  These are the precise conditions needed for
$\big(m|\delta_1|, \ldots,
m|\delta_n|\big)$ to be compatible in $\Gamma_{v_1}\times \ldots\times
\Gamma_{v_n}$ in the terminology of Ribenboim's paper
\cite{R}.  Since this compatibility holds, the general approximation
theorem for incomparable valuations on $K$ \cite[Th.~5]{R}
 says that there exists
$d\in K^\times$ with
$$
v_i(d) \, =\, m \, |\delta_i|, \qquad \text{for all \,$i$}.
$$
In particular, $v_k(d) = m\,|\delta _k| = 0$.  Now,
let $V_i$ be the valuation ring of $v_i$, and
let ${T = V_1\cap \ldots\cap V_n
 \subseteq K}$. A weaker approximation theorem for
incomparable valuations on $K$ (see
~\cite[Th.~3.2.7(3), p.~64]{EP})
says that the canonical map $\rho\colon T\to
\ov K^{\,v_1} \times \ldots \times \ov K^{\,v_n}$ is surjective.  Therefore,
there is $t\in T$ with $\rho(t) =\big( \ov 0 , \ov 0, \ldots,  \ov d\,\inv,
 \ldots,\ov 0\big)$, i.e., $v_i(t)>0$ for $i \ne k$ and $v_k(t) = 0$
with $\ov t = \ov d{\,\inv}$ in $\ov K^{\,v_k}$.  Let $c = td$.  Then
for $i \ne k$ we have
$$
v_i(c) \, =\ v_i(t) + v_i(d) \,> \, 0 + m|\delta_i| \, \ge \delta_i,
$$
 hence $v_i(c) >\delta _i$.  Also, $v_k(c) = v_k(t) + v_k(d) = 0$,
and in $\ov K^{\,v_k}$,
$$
\ov c \, = \, \ov t\cdot \ov d \, = \,\ov d\, \inv \ov d \, = \, \ov 1.
$$
 Thus, $c$ has all the required properties. 
\end{proof}

\begin{proof}[Proof of Th.~\ref{thm:compatible} completed:]
It remains only to prove the surjectivity of the map $\Psi$
in \eqref{gradedpsi}. Fix any $k\in \{1,\ldots, r\}$, and take any $b\in A\setminus \{0\}$. Let
$$
\delta_i \, = \, \eta_k(b) - \eta_i(b) \qquad\text{for \ } i = 1,\ldots, r.
$$
For each pair of indices $i$, $j$, let $\Delta_{ij}$ be the convex
subgroup
of $\divh{\Gamma_F}$ associated to the finest common coarsening
$v_{ij}$ of $v_i$ and $v_j$ on $K$.
Then, since $\eta_i$ and $\eta_j$ are assumed to have the same
$v_{ij}$-coarsening, $\eta_i(b)$ and $\eta_j(b)$ have the same
image in $\Gamma/\Delta_{ij}$; hence,
$$
\delta_i - \delta_j\, = \, \eta_j(b) - \eta_i(b) \,\in\, \Delta_{ij}.
$$
Thus, $(\delta_1, \ldots, \delta_r)$ is a compatible $r$-tuple in
$\divh{\Gamma_F}^r$. Since $\delta_k = 0$, Lemma~\ref{lem:approx} yields $c\in K^\times$ with
$v_i(c) >\delta_i$ for all $i \ne k$ and $v_k(c) = 0$ with
$\ov c = \ov 1$ in
$\ov K^{\,v_k}$. Let $a = cb \in A$.  Then, for $i \ne k$,
$$
\eta_i(a) \,=\, v_i(c) + \eta_i(b) \, > \, \delta_i + \eta_i(b) \, = \,
\eta_k(b),
$$
so $\eta_i(a) > \eta_k(b)$.
But
$$
\eta_k(a) \,=\,v_k(c) + \eta_k(b) \, = \, \eta_k(b).
$$
Hence, $\alpha(a) = \min\big(\eta_1(a), \ldots,\eta_r(a)\big) = \eta_k(b)$.
Moreover,
$$
\eta_k(a-b) \, = \, \eta_k(cb-b) \, = \, v_k(c-1) + \eta_k(b)
\, > \, \eta_k(b).
$$
Thus, for $\tilde a = a+A_{>\alpha(a)}^\alpha \in \gr_\alpha(A)$
and $\tilde b = b+ A_{>\eta_k(b)}^{\eta_k}\in \gr_{\eta_k}(A)$, we have
$$
\Psi(\tilde a) \, = \, \big(a+ A_{>\alpha(a)}^{\eta_1}, \ldots,
a+ A_{>\alpha(a)}^{\eta_r}\big)\, = \,
\big(0, \ldots, 0, \begin{array}[t]{c}\tilde b\\k\end{array}, 0,
\ldots,0\big)\in \tprod\limits_{i=1}^r
\gr_{\eta_i}(A).
$$
Since for arbitrary $k$ and $b$ such elements generate
$\prod_{i=1}^r
\gr_{\eta_i}(A)$, the map $\Psi$ is surjective.  This completes the proof.
\end{proof}

\begin{corollary}\label{cor:viindep}
Suppose $v$ is defectless in $A$,  and suppose the extensions
$v_1, \ldots, v_r$ of $v$
to $K$ are pairwise independent. Take  any $v_i$-gauges $\eta_i$ on $A$,
$i = 1,
\ldots, r$.  Then $\min\big(\eta_1, \ldots, \eta_r\big)$ is a $v$-gauge
on $A$ with
components $\eta_1, \ldots \eta_r$.  In particular, this  holds whenever
$v$ has rank $1$.
\end{corollary}

\begin{proof}
This is immediate from the preceding theorem. For when $v_i$ and $v_j$
are independent valuations, their finest common coarsening
$v_{ij}$ is the
trivial valuation, so the compatibility condition on
$\eta_i$ and $\eta_j$ holds automatically.
\end{proof}

\section{Dubrovin valuation rings and Gr\"ater rings}

In this section we study the connection between gauges and Dubrovin
valuation rings.
The best general reference for Dubrovin theory is
the book \cite{MMU}.
 Let $F$ be a field and let $A$ be a central simple
$F$-algebra. A subring $B$ of $A$ is called a
\emph{Dubrovin valuation ring} of $A$ if there is an ideal $J$ of
$B$ such that the following hold:
\begin{enumerate}
\item[(1)]  $B/J$ is simple Artinian;
\item[(2)] for each $s \in A \setminus B$ there exist $b, c \in B$ such
that $bs, sc \in B \setminus J$.
\end{enumerate}
By \cite[Lemma~5.2, p.~22]{MMU}, the ideal $J$ is the only maximal ideal
of $B$; therefore,
$J  = J(B)$, the Jacobson ideal of $B$. Moreover, the center
$Z(B) = F \cap B$ is a valuation ring of $F$ and
$J \cap F = J(Z(B))$ (see \cite[Lemma~7.1, p.~35]{MMU}).
  We review a few of the nontrivial
properties  of Dubrovin
valuation  rings  that we will    use  repeatedly below. 
Proofs of all of them can be found in Chapters 5 and 6 of
\cite {MMU}.
Let $S$ be an overring of $B$ in $A$, i.e., a ring with
$B\subseteq S \subseteq A$. Then, $S$ is a Dubrovin valuation
ring of $A$; $J(S)$ is a prime ideal of $B$;  $S$ is the left
(and right)
localization of $B$ with respect to the elements of $B$ regular mod $J(S)$;
and $B/J(S)$ is a Dubrovin valuation ring of $S/J(S)$.
Also, $S$ is the central localization $S = B\otimes_{Z(B)}W = {B\!\cdot\!W}$,
where the valuation ring $W = Z(S)$ is the localization of $Z(B)$ at its
prime ideal $J(S) \cap F$.
Moreover, every prime
ideal of $B$ has the form  $J(S)$ for some such $S$.
If $R$ is a subring
of $B$ with $J(B) \subseteq R$, then $R$ is a Dubrovin valuation ring of
$A$ if and only if $R/J(B)$ is a Dubrovin valuation ring of~$B/J(B)$.
Examples of Dubrovin valuation rings include
\begin{itemize}
\item[$\scriptstyle{\bullet}$]
total valuation ring in division algebras $D$, i.e., subrings
$T$ of $D$ such that $d$ or
$d\inv$ lies in $T$ for every $d \in D^\times$;
\item[$\scriptstyle{\bullet}$]
matrix rings $\mat_n(B)$ for any Dubrovin valuation ring $B$;
\item[$\scriptstyle{\bullet}$]
rings $eBe$ for any Dubrovin ring $B$ of $A$ and any nonzero
idempotent $e$ of $A$;
\item[$\scriptstyle{\bullet}$]
 Azumaya algebras over
commutative valuation rings.
\end{itemize}
Associated to a Dubrovin valuation ring $B$ there is
 its \emph{residue ring} $\overline{B} = B/J(B)$, which is simple
Artinian.  This $B$ also has a \emph{value group}
$\Gamma_B = \mathsl{st}(B)/B^\times$, where
$\mathsl{st}(B) = \{x \in A^\times \; | \; xBx^{-1} = B \}$;
the value group
$\Gamma_B$ is a totally ordered abelian group.
It was proved in \cite{W} that
there is a
strong connection between Dubrovin valuation rings and invariant valuation
rings (i.e., the valuation rings associated to a valuation
on a division algebra), which shows up
with passage to the Henselization of the valuation of the center:
 Let $(F_h,v_h)$ be a Henselization
of $(F,v)$. The ring $A \otimes_F F_h$ is a central simple $F_h$-algebra;
hence,
\begin{equation}\label{eq:defnB}
A \otimes_{F} F_h \,\cong\, \mat_{n_B}(D_h),
\end{equation}
 where $n_B$ is a positive
integer and $D_h$ is a central division $F_h$-algebra.  
Thus, $n_B$ is the matrix size of the algebra $A \otimes_F F_h$.
The Henselian valuation
$v_h$ has a unique extension
to a valuation $w$ on~$D_h$. By \cite[Th.~B]{W} or
\cite[Lemma~11.4, p.~59]{MMU},
\begin{equation}\label{eq:gammaB}
\Gamma_B \,=\,  \Gamma_w\quad \text{in} \ \ \divh{\Gamma_v} 
\end{equation} and
\begin{equation}\label{eq:deftB}
\overline{B} \,\cong \,\mat_{t_B}(\overline{D_h}^{\,w}),
\end{equation}
 where  the  positive
integer $t_B$ is the matrix size of $\ov B$.
Moreover, $n_B/t_B$ is always an integer that appears in
the \emph{Ostrowski Theorem} for Dubrovin valuation rings, which says that
\begin{equation}\label{defectdubrovin}
[A:F] = [\overline{B}: \overline{F}^{\,v}]\cdot
|\Gamma_B:\Gamma_v| \cdot \big(n_B/t_B\big)^2 \cdot \partial_B,
\end{equation}
where the \emph{defect} $\partial_B = {\ov p}^d$, for
$\ov p = \charac(\overline{F}^{\,v})$ and $d$ a non-negative integer,
with ${\ov p}^d = 1$ if $\ov p = 0$
(see \cite[Th.~C]{W}). It follows from \eqref{defectdivision} applied to
$D_h$  and
\eqref{defectdubrovin} that $\partial_B = \partial_{D_h/F_h}$.
Hence, $v$ is defectless in $A$ if and only if $\partial_B = 1$ for any
Dubrovin valuation ring $B$ of $A$ extending $V$.

Dubrovin valuation rings have good  properties
with respect to extension from the center. More precisely,
if $V$ is a valuation ring of $F$ then there always exists a Dubrovin
valuation ring $B$ of the central simple
$F$-algebra $A$ extending $V$, i.e., $V = B \cap F$, see
\cite[Th.~9.4, p.~50]{MMU}.
Moreover, by \cite[Th.~A]{W} or \cite[Th.~9.8, p.~52]{MMU}
there is as much uniqueness as possible, i.e.,~all Dubrovin
ring extensions of $V$ to $A$ are
conjugate in $A$. But the number of extensions of $V$ to
Dubrovin valuation rings of $A$ is usually infinite;  the exception
occurs only when
 $A$ is a division algebra and $V$ can be extended to a total
valuation ring $T$ of $A$. In this special case,
 the number of total valuation rings of $A$ extending~$V$ is given
by $n_B/t_B$. In order to obtain a
better understanding of the Ostrowski
Theorem (equivalently, a better interpretation of the integer
$n_B/t_B$ for arbitrary Dubrovin valuation rings),  Gr\"ater introduced in
\cite{G} the Intersection Property for a finite number of Dubrovin
valuation rings as follows: Let $B_1, \ldots, B_n$ be Dubrovin valuation
rings of $A$ and let $R = B_1\cap \ldots \cap B_n$. Let
$\mathcal{B}(B_i)$ denote the set of all overrings of $B_i$ in $A$. Then
$B_1, \ldots,B_n$ have the \emph{Intersection Property} (IP) if the map
\begin{equation}\label{mapIP}
\begin{array}{rcl}
\varphi\colon \mathcal{B}(B_1) \cup \ldots \cup \mathcal{B}(B_n) &
 \rightarrow & \mathsl{Spec}(R)\\
S \qquad\qquad& \mapsto & J(S)\cap R
\end{array}
\end{equation}
is a well-defined order-reversing bijection, where
$\mathsl{Spec}(R)$ is the set of prime ideals of $R$. Actually, if one
supposes only that
$\varphi$ is well-defined (i.e., each ideal $J(S) \cap R$  is 
a prime ideal of $R$),  then  in fact $\varphi$~is an order-reversing bijection
(see \cite{Z}).  Note that if $B_j \supseteq B_i$ for some $i,j$, then we may
delete $B_j$ from the list of $B$'s and the ring $R$
and the domain and target of $\varphi$ are unchanged.  Thus, in working
with the IP, we may delete all such redundant $B_j$ and assume that the
$B_i$ are pairwise incomparable.

At the same time that Intersection Property was introduced by Gr\"ater,
Morandi was working independently on a general approximation theorem for
a finite set of
Dubrovin valuation rings (see~\cite{M3}). He found the following condition:
Let $B_1, \ldots, B_n$ be
pairwise incomparable Dubrovin valuation
rings. For each $i \neq j$,
let $B_{ij}$ be the subring of $A$  generated by  $B_i$ and $B_j$;
so $B_{ij}$ is a Dubrovin valuation ring, since it is an
overring of $B_i$.
Set
$\widetilde{B_i} = B_i/J(B_{ij})$ and $\widetilde{B_j} = B_j/J(B_{ij})$,
which are Dubrovin valuation rings of $\ov{B_{ij}} = B_{ij}/J(B_{ij})$.
The condition needed for the approximation theorem to hold for
$B_1, \ldots, B_n$ is that
the valuation rings $Z(\widetilde{B_i})$ and $Z(\widetilde{B_j})$ be
 independent   in the field $Z(\overline{B_{ij}})$, i.e., 
${Z(\widetilde{B_i})\!\cdot\! Z(\widetilde{B_j})} = 
Z(\overline{B_{ij}})$. Gr\"ater proved that
this condition is equivalent to the Intersection Property
for $B_1, \ldots, B_n$
(see \cite[Cor.~6.2, Prop.~6.3, Cor.~6.7]{G} or \cite[Cor.~16.9,
p.~93]{MMU}).
It follows that
\begin{align}\label{eq:pairwiseIP}
\begin{split} 
&\text{$B_1, \ldots, B_n$ satisfy the IP if and only if each pair
$B_i,B_j$ satisfies the IP,} \\
&\text{for all $i, j\in \{1,\ldots, n\}$
with $i \ne j$. }
\end{split}
\end{align}

We call a subring $R$ of a central simple $F$-algebra $A$
a {\it Gr\"ater ring}  if $R = B_1 \cap \ldots \cap B_n$
where $B_1, \ldots, B_n$ is a family of incomparable
 Dubrovin valuation
rings of $A$ with the IP and $R$ is integral over $Z(R)$.
This name is appropriate because of the remarkable properties
Gr\"ater proved about such rings in \cite{G} and \cite{G2}.
(His results appear also in \cite{MMU}.)
He showed that a Gr\"ater ring is a semilocal B\'ezout ring
whose center is a finite intersection of valuation rings of
$F$. When $R= B_1\cap \ldots \cap B_n$ as above,
then the localizations of the $R$ with respect to its
maximal ideals exist and coincide with the $B_i$.
Thus in particular,
\begin{equation} \label{eq:numofBi}
\text{ the number $n$ of $B_i$} \, = \, \text{the number of
maximal ideals of $R$.}
\end{equation}
Moreover,
if $T$ is any finite intersection of valuation rings of $F$
then there exists a Gr\"ater ring $R$ of $A$ with $Z(R) = T$.
Further, Gr\"ater  proved  a very strong uniqueness property:
If $R'$ is another Gr\"ater ring of $A$ with $Z(R') = T
= Z(R)$, then there is $q\in A^\times$ with $R' = q R q\inv$.
See \cite[Th.~16.14, p.~94; Th.~16.15, p.~ 96]{MMU} for proofs
of these properties.

There are also striking properties when the center
of a Gr\"ater ring is a
valuation ring $V$:  Let  $R = B_1\cap \ldots\cap B_n$ be 
a Gr\"ater ring with center $V$ (and the $B_i$  incomparable);
then, each $B_i$ has
center~$V$, by \cite[Lemma~ 16.13, pp.~93--94]{MMU}.
Moreover, if $B_1', \ldots, B_k'$ is another family of
Dubrovin valuation rings of $A$ each with center $V$
and having the~IP, then the $B_i'$ are incomparable
(since overrings of $B_i'$ are central localizations) and
this family can be enlarged with further Dubrovin valuation rings
$B_{k+1}',\ldots,  B_m'$ each with center $V$ so that
$R' = B_1' \cap \ldots\cap B_m'$ is a Gr\"ater ring with center~$V$
(see \cite[Th.~ 16.14, p.~ 94]{MMU}).
The conjugacy result noted above shows that $R' \cong R$, hence
$m = n$ by~\eqref{eq:numofBi}. Thus, this number $n$ depends only on $V$ and $A$,
and Gr\"ater defined it to be the {\it extension number} of
$V$ in $A$.  Thus, the extension number is the number of
maximal ideals in any Gr\"ater ring of $A$ with center $V$,
and it equals the number of incomparable Dubrovin valuation
rings in any family whose intersection is a Gr\"ater ring
with center $V$ (see \eqref{eq:numofBi}); it is also the maximum number of
Dubrovin rings with center $V$ which have the IP.
 It turns out  that the
extension number  also equals the
quotient  $n_B/t_B$ of the integers $n_B$ and $t_B$ given in
\eqref{eq:defnB} and \eqref{eq:deftB}, for
any Dubrovin valuation ring $B$ of $A$, see \cite[Prop.~19.2, p.~108]{MMU}.
From the conjugacy of such rings $B$, it is clear that $n_B$ and
$t_B$ do not depend on the choice of $B$.
Moreover, it follows
from  the definition of
$n_B$~and~$t_B$ that the extension number depends only on the valuation
ring
$V$ of the center $F$ of $A$ and the  Brauer equivalence class of $A$. Let
\begin{equation}
\xi_{V,[A]} \, = \, \text{the extension number of $V$ to $A$}
\, = \, n_B/t_B.
\end{equation}
 When the algebra  $A$ in question is clear, we write  $\xi_V$ for
$\xi_{V,[A]}$.

Now let $W$ be any (valuation) ring with $V \subseteq W \subseteq F$, 
and let ${S = W\!\cdot\! B}$. Then $S$ is
  a Dubrovin valuation ring of $A$ containing $B$ and $W = S \cap F$.
Let $\widetilde{V} = V/J(W)$, which is a
valuation ring of $\ov W = W/J(W)$,
 and let
$\widetilde{B} = B/J(S)$, which is a Dubrovin valuation ring of
the simple $\ov W$-algebra $\ov S = S/J(S)$.
 Let
\begin{equation}\label{eq:deflVW}
\ell_{V,W} \, = \,\text{the number of extensions of
$\widetilde{V}$ to valuation rings of  $Z(\overline{S})$}.
\end{equation}
 This number does not depend on the choice
of the Dubrovin extensions $B$ of $V$ and $S$ of $W$,
by the conjugacy of $B$ with center $V$ and the property that
overrings of $B$ are obtained by central localization. The following fundamental
relation was given in \cite[Th.~E]{W}:
\begin{equation} \label{teoremaE}
\xi_{V,[A]} \,=\,
\xi_{W,[A]}\, (n_{\widetilde{B}}/t_{\widetilde{B}}) \,\ell_{V,W}
\, = \,\xi_{W,[A]}\, \xi_{Z(\tilde B),[\ov S]} \,\ell_{V,W}.
\end{equation}
Hence,
\begin{equation}\label{eq:xidiv}
\xi_{W,[A]}\, \big| \, \xi_{V,[A]}\qquad\text{ for any valuation ring $W$
with $V \subseteq W \subseteq F$.}
\end{equation}
There is another interpretation of $\ell_{V,W}$ that will be needed later:
Let $w$ be the valuation on $F$ associated to $W$, let $(F_{h,w},w_h)$ be
the Henselization of $(F,w)$,  let $D_{h}$ be the associated division
algebra of $A \otimes _F F_{h,w}$, and let $\ov {D_{h}}^{\,w'}$ be the
residue division algebra of $D_{h}$ for the valuation
$w'$ on~$D_{h}$ extending $w_h$ on $F_{h,w}$.  Let $u$ be the valuation
of $\widetilde V$ on $\ov F^{\,w}$.  Then,
\begin{equation}\label{eq:lVWchar}
\ell_{V,W}\, = \,\text{the number of extensions of $u$ to
valuations of
$Z(\ov {D_{h}}^{\,w'})$}.
\end{equation}
This holds because $\ov S \cong \mat_{t_S}(\ov {D_{h}}^{\, w'})$,
see \eqref{eq:deftB},
  so the $\ov F^{\,w}$-algebras
$\ov S$ and $\ov {D_{h,w}}^{\, w'}$ have isomorphic centers.

Note that the following conditions are equivalent:
\begin{enumerate}
\item[(a)]
Some (so every) Dubrovin valuation ring
$B$ of $A$ extending $V$ is integral over $V$.
\item[(b)]$\xi_{V,[A]} = 1$.
\item[(c)] Every Gr\" ater ring $C$ of $A$ with $C\cap F = V$
is a Dubrovin valuation ring.
\end{enumerate}
These conditions do not hold in general, but for given  $V$ and 
$A$, they hold for some nontrivial coarsening of $V$,
see Prop.~\eqref{prop:maxint} below.

\begin{example}
Let $A$ be a central simple $F$-algebra, and let $V$ be a rank $2$ valuation
ring of $F$.  Let $W$ be the rank $1$ valuation ring with
$V\subsetneqq W \subsetneqq F$. Then $\xi_{W,[A]} = 1$ by 
Prop.~\ref{prop:maxint}(iv) below.
Let $S$ be a Dubrovin valuation ring of $A$ extending $W$.  So,
$S$ is integral over $W$.  Let $\ov S = S/J(S)$, which is a simple
$\ov F = W/J(W)$-algebra.  Suppose the valuation ring $U =
V/J(W)$ of~$\ov F$ has exactly $k$  different extensions
$U_1, \ldots , U_k$ to $Z(\ov S)$.
The map $B \mapsto B/J(S)$ gives a one-to-one correspondence
between the Dubrovin valuation rings $B$ of $A$ extending $V$
and lying in $S$, and the Dubrovin valuation rings~$T$ of
$\ov S$ with $T \cap \ov F = U$.  For each such $T$, the
intersection $T\cap Z(\ov S)$ is a valuation ring of~$J(S)$
extending $U$, so it is one of the $U_j$.  Moreover,
$T$ is integral over $U_j$ since $U_j$ has rank $1$. 
Let $B_1, \ldots, B_n$ be Dubrovin valuation rings of
$A$ extending $V$ with each $B_i \subseteq S$.
Let $\tilde{B_i} = B_i/ J(S)$, and let $\tilde {B_i} \cap
Z(\ov S) = U_{j(i)}$.  Then, $B_1, \ldots, B_n$ have the IP
if and only if $j(1), \ldots, j(n)$ are all different.
The ring $C = \bigcap_{i =1}^n B_i$ is a Gr\"ater ring if and
only if further $C$ is integral over $V$, if and only if
 $\bigcap_{i =1}^n \tilde {B_i}$ is integral over
$U$, if and only if $n = k$ and $\{j(1),j(2), \ldots, j(n)\}
= \{1,\ldots, n\}$.  Thus, $\xi_{V,[A]} = k$.
\end{example}

In general, there is no valuation associated to a Dubrovin valuation ring
$B$ of $A$.  However, a significant exception occurs when $B$ is integral
over $Z(B)$:
Morandi introduced in \cite{M2} a type of value function associated to any
Dubrovin valuation ring integral over its center. Let
${\alpha\colon A \rightarrow \Gamma \cup \{\infty\}}$ be a surmultiplicative
$v$-value function. We say that $\alpha$ is a
\emph{Morandi value function} if
\begin{enumerate}
\item[(1)] $A_0$ is a simple ring;
\item[(2)] $\Gamma_\alpha = \alpha(\mathsl{st}(\alpha)),$ where
$\mathsl{st}(\alpha) =
\{a \in A^\times \; | \; \alpha(a^{-1}) = -\alpha(a)\}.$
\end{enumerate}
If $v = \alpha|_F$, then $v$ is a valuation on $F$ and
we call $\alpha$ a $v$-Morandi value function.
Morandi showed in \cite[Th.~2.4]{M2} (or see \cite[Th.~23.3, p.~135]{MMU}) the following:
If $\alpha$ is a Morandi value function on $A$,  then its associated ring
$$
R_\alpha \,= \,\{x \in A \; | \; \alpha(x) \geq 0 \}
$$ is a
Dubrovin valuation ring integral over its center 
and 
\begin{equation}\label{eq:Morvalgp}
\Gamma_{R_\alpha}\, =\, \Gamma_\alpha. 
\end{equation}
 Conversely, to every
Dubrovin valuation ring~$B$ of $A$ integral over its center, there exists
a Morandi value function $\alpha$ on $A$ such that $R_\alpha = B$;
moreover, $\alpha$ is uniquely determined by $B$ (see \cite[Th.~2.3,
Prop.~2.6]{M2}
or \cite[Th.~23.2, p.~ 134]{MMU}).
(By contrast, we will see in
Example~\ref{ex:notdetermined} below  that if $\beta$ is a gauge on
$A$, then its  gauge ring $R_\beta$ does not always determine $\beta$.)
 The connection between gauges and Morandi value functions
was shown in \cite[Prop.~2.5]{TW1}: If
$A$ is a central simple $F$-algebra and $v$ is a valuation on $F$
defectless in $A$, then a surmultiplicative $v$-value function $\alpha$
is a Morandi value function if and only if $\alpha$ is a gauge with $A_0$
simple. The following result generalizes this for simple  but not
necessarily central simple algebras. If $\alpha$ is a gauge on
a semisimple algebra $A$, then
since $\gr_\alpha(A)$ is graded semisimple, the
degree zero part $A_0$ of $\gr_\alpha(A)$ must be semisimple
(cf.~\cite[Prop.~2.1]{TW1}). But 
$A_0$ is
 not necessarily simple, even if $A$ is
simple, as is illustrated in Ex.~\ref{ex;endgauge}. Let
\begin{equation}\label{eq:defomega}
\omega(\alpha) \, = \, \text{the number of simple components of
$A_0$}.
\end{equation}
\begin{theorem}  \label{morandiIP}
Suppose $v$ is a valuation on $F$ defectless in a simple $F$-algebra $A$.
Let $v_1, \ldots, v_r$ be all the extensions of $v$ to $K=Z(A)$.
\begin{enumerate}
\item[(i)] For $i\in\{1, \dots, r\}$, let $\alpha_i$ be a $v_i$-Morandi
value
function on $A$ and let $B_i$ be the Dubrovin valuation ring associated
to $\alpha_i$. Let $\alpha = \min \big(\alpha_1, \ldots, \alpha_r \big).$
Then $\alpha$ is a $v$-gauge on $A$ if and only if $B_1, \ldots, B_r$
have the IP. When this occurs, $\omega(\alpha) = r$. 
\item[(ii)] Let $\alpha$ be a $v$-gauge on $A$ with $\omega(\alpha)=r$.
Then there are uniquely determined $v_i$-Morandi value functions
$\alpha_i$ for $i \in\{ 1, \ldots, r\}$ such that
$\alpha = \min \big(\alpha_1, \ldots, \alpha_r\big).$
\end{enumerate}
\end{theorem}
\begin{proof}
For each $i$ and $j$, let $V_i$ be the valuation ring of $v_i$,
 and let $V_{ij}
= {V_i\!\cdot\! V_j}$, with its associated valuation $v_{ij}$, which is the finest
common coarsening of $v_i$ and $v_j$.  Let $\tilde {V_i} = V_i/J(V_{ij})
\subseteq \ov {V_{ij}}$, and define $\tilde{V_j}$ analogously.  Note
that $\widetilde {V_i}$ and $\tilde{V_j}$ are independent in
$\ov{V_{ij}}$, i.e., ${\tilde {V_i} \!\cdot\!\tilde{V_j}} = \ov{V_{ij}}$, since~
${V_i\!\cdot\! V_j} = V_{ij}$.

(i) Suppose $\alpha$ is a $v$-gauge on $A$. By Th.~\ref{thm:compatible},
for every pair $i,j$, the gauges
$\alpha_i$ and $\alpha_j$ have the same $v_{ij}$-coarsening,
call it $\alpha_{ij}$.  So, for the gauge rings, we have
$R_{\alpha_i} \subseteq R_{\alpha_{ij}}$ and
$R_{\alpha_j} \subseteq R_{\alpha_{ij}}$.
Since
$B_i = R_{\alpha_i}$ and likewise for $j$, we must have
$B_{ij} \subseteq R_{\alpha_{ij}}$.
 Hence,
$$
Z(B_{ij}) \,=\, B_{ij} \cap K
\,\subseteq\, R_{\alpha_{ij}} \cap K \, = \,
Z(R_{\alpha_{ij}}) \,=\, V_{ij}.
$$
 Since we always
have $V_{ij} = {V_i\!\cdot\! V_j}\subseteq Z(B_{ij})$, we obtain
$V_{ij} = Z(B_{ij})$.
The field $Z(\ov{B_{ij}})$ is a finite-degree extension of the
field $\ov{V_{ij}}$.
Since the valuation rings
$\widetilde{V_i}$ and $\widetilde{V_j}$ are independent in
$\overline{V_{ij}} \subseteq Z(\overline{B_{ij}})$ and
$Z(\widetilde{B_i}) \cap \overline{V_{ij}} = \widetilde{V_i}$, it follows
that the valuation rings
$Z(\widetilde{B_i})$ and $Z(\widetilde{B_i})$ are independent in
$Z(\overline{B_{ij}})$. Hence, $B_i$ and $B_j$ have the IP. Therefore,
$B_1, \ldots, B_r$ have the IP by \eqref{eq:pairwiseIP}.
Moreover, since by Th.~\ref{thm:gaugeismin} 
\begin{equation}\label{eq:grsimple}
\gr_\alpha(A) \,\cong_g\, \gr_{\alpha_1}(A) \times \ldots\times
\gr_{\alpha_r}(A)
\end{equation} 
 and each $A^{\alpha_i}_0$ is the simple residue ring of
$R_{\alpha_i}$, we have $\omega(\alpha) = 
\sum_{i = 1}^r\omega(\alpha_i) = \sum _{i=1}^r 1 = r$.

Conversely, suppose $B_1, \ldots, B_r$ have the IP.
For any distinct $i,j\in \{1,\ldots, r\}$ let
$\Delta_{ij}$ be the convex subgroup of the divisible hull
$\Gamma$ of $\Gamma_v$ associated to $V_{ij}$ and let
$\theta_{ij}\colon \Gamma \rightarrow \Gamma/\Delta_{ij}$ be the canonical
map. Since $B_i$ is the gauge ring $R_{\alpha_i}$,
by Prop.~\ref{coarsergaugering},
$R_{\theta_{ij} \circ \alpha_i} = {B_i\!\cdot\!V_{ij}}$; likewise,
${R_{\theta_{ij} \circ \alpha_j} = {B_j\!\cdot\!V_{ij}}}$. Since 
${B_i\!\cdot\!V_{ij}}$ and
${B_j\!\cdot\!V_{ij}}$ are Dubrovin valuation rings of $A$ integral over $V_{ij}$,
it follows that $\theta_{ij} \circ \alpha_i$ and $\theta_{ij} \circ \alpha_j$
are Morandi value functions. Moreover,
since $B_i$ and $B_j$ have the IP, their overrings
${B_i\!\cdot\!V_{ij}}$ and ${B_j\!\cdot\!V_{ij}}$ also have 
the IP by \cite[Th.~16.8, p.~92]{MMU}. 
But the integrality of ${B_i\!\cdot\!V_{ij}}$  over $V_{ij}$ implies that
$\xi_{V_{ij},[A]} = 1$.  By using this or
\cite[Lemma~16.5, p.~90]{MMU}, it follows 
that ${{B_i\!\cdot\!V_{ij}} = {B_j\!\cdot\!V_{ij}}}$. Therefore,  
$ \theta_{ij} \circ \alpha_i = {\theta_{ij} \circ \alpha_j}$, because a
Morandi value function is completely determined by its associated Dubrovin
valuation ring (see \cite[Prop.~23.6, p.~135]{MMU}). Thus, it follows from 
Th.~\ref{thm:compatible} that $\alpha$~is a $v$-gauge.

(ii) By Th.~\ref{thm:gaugeismin} there exist $v_i$-gauges $\alpha_i$ on
$A$ for $i\in\{1,\ldots, r\}$ such that $\alpha \, =\,
\min\big(\alpha_1, \ldots, \alpha_r\big)$ and
\eqref{eq:grsimple} holds.
Thus $\omega(\alpha) = \omega(\alpha_1) + \ldots + \omega(\alpha_r).$
Since $\omega(\alpha) = r$,  we must have
each $\omega(\alpha_i) = 1$; it follows from \cite[Prop.~2.5]{TW1} that
each $\alpha_i$ is a $v_i$-Morandi value function. The uniqueness
of the $\alpha_i$ follows
from Cor.~\ref{cor:compunique}.
\end{proof}

For  proofs of some of the following theorems  we need the 
notions of jump rank and jump prime ideals, defined as follows:
Let $A$ be a central simple $F$-algebra and let $V$ be the valuation
ring of a nontrivial valuation $v$ on $F$.  For each nonzero prime
ideal $P$ of $V$, let $v_P$ be the valuation on $F$ with valuation 
ring $V_P$, let field $F_{h,P}$ be a Henselization of $F$ with 
respect to $v_P$, and let~$n_P$~be the matrix size of 
$A\otimes_F F_{h,P}$. So, $1\le n_P \le \deg A$.  For prime ideals
$P \subseteq Q$  we have $n_P \le n_Q$ since $F_{h,P}$ embeds in 
$F_{h,Q}$.  We say that $P$ is a {\it jump prime ideal} of $v$ for $A$
if $n_P < n_Q$ for every prime ideal $Q \supsetneqq P$; we then say that
$v_P$ is a {\it jump valuation} of $v$ for $A$.  The {\it
jump rank} of $v$ for $A$ is defined to be 
$$
j(v, A) \, = \, \text{the number of jump prime ideals of $v$ for $A$}.
$$
Note that if a positive integer $m =n_P$ for some $P$, and $\frak P$
is the union of all the prime ideals $Q$ with $n_Q = m$, then 
$\frak P$ is a prime ideal of $V$, since the $Q$'s are linearly 
ordered, and $n_{\frak P} = m$ since $F_{h,\frak P}$ is the direct limit of 
the $F_{h,Q}$; so $\frak P$ is the unique jump prime ideal of $v$
with $n_{\frak P} = m$.  Thus, 
$$
j(v,A) \,=\, \big|\{n_P\mid 
\text{$P$ is a nonzero prime ideal of $V$}\}\big|,
$$
 and $1\le j(v,A) \le \deg A$.  
The jump rank is  useful for induction
arguments, since it is always finite  even when the valuation 
$v$ has infinite rank.  If ${P_1 \subsetneqq P_2 \subsetneqq \ldots 
\subsetneqq P_{j(v,A)}}$ are all the distinct jump prime ideals of $v$ 
for $A$,  we call $P_i$ the $i$-th jump prime of $v$ for $A$.  Note that
$P_{j(v,A)} = J(V)$.  If~$v$~ is the trivial valuation on $F$, we set
$j(v, A) = 0$. 
More information on the jump rank can be found in~ 
\cite {W} or \cite{MMU}.

\begin{proposition}\label{prop:maxint}
Let $F$ be a field, and let $v$ be a nontrivial valuation on $F$
with valuation ring $V$.  Let $A$ be a central simple $F$-algebra,
and let $B$ be a Dubrovin valuation ring of $A$ with center $V$.
\begin{enumerate}
\item[(i)]
There is a unique nonzero prime ideal $\frak P$ of $V$ maximal with the 
property that
${B\!\cdot\! V_{\frak P}}$ is integral over $V_{\frak P}$. 
\item[(ii)]
Let $Q$ be any prime ideal of $V$. Then, ${B\!\cdot\!V_Q}$ is integral over
$V_Q$ if and only if $Q \subseteq \frak P$.
\item[(iii)]
 $\frak P$ is a jump prime ideal of $v$ for $A$.
\item[(iv)] 
If $j(v,A) = 1$, then $\frak P = J(V)$, so $B$ is 
integral over $V$. 
\item[(v)]
If $\frak P \ne J(V)$, then 
$\ell_{V, V_{\frak P}} >1$. 
\end{enumerate} 
\end{proposition}
 
\begin{proof}
(i) and (v) The existence of $\frak P$ with the maximal property
is given in \cite[Prop.~12.4, p.~72]{MMU}, where it is also proved   
that if $\frak P \ne J(V)$ (so $B$ is not integral over
$V$), then $\ell_{V, V_{\frak P}} >1$. The prime ideal $\frak P$
is unique with the maximal property since the prime ideals of 
the valuation ring $V$ are linearly ordered.  
 
(ii) For a prime ideal $Q$ of $V$, if $Q \supsetneqq \frak P$, 
then ${B \!\cdot\! V_Q}$ is not integral over $V_Q$ by the maximality
of~$\frak P$.  But if $Q \subseteq\frak P$, then $V_Q \supseteq
V_{\frak P}$, so by \eqref{eq:xidiv} 
$\xi_{V_Q,[A]} \le \xi_{V_{\frak P},[A]} = 1$.
Hence, $\xi_{V_Q,[A]} = 1$, showing that ${B \!\cdot\!V_Q}$ is integral over
$V_Q$.

(iii) If $\frak P = J(V)$,
then $\frak P$ is a jump prime ideal of $v$ for $A$.  
Assume now that $\frak P \ne J(V)$.
Let $W = V_{\frak P}$
and let $S= {B \!\cdot\! W}$, which is a Dubrovin valuation
ring of $A$ with center $W$.  Let $\tilde B = B/J(S)$, 
which is a Dubrovin valuation ring of $\ov S = 
S/J(S)$.  Note that 
\begin{equation}\label{eq:tStB}
\overline{\widetilde{B}} \,= \,\big(B/J(S)\big)\,\big/\, J(B/J(S)) 
\,\cong\, B/J(B) \,=\, \overline{B}.
\end{equation}
Hence, these residue rings have the same matrix size, i.e., 
$t_{\tilde B} = t_B$.
We have  $\overline{S} =\mat_{t_S}(E)$ 
for some division ring $E$.
If $C$ is a Dubrovin valuation ring of $E$ with center the valuation 
ring $Z(\tilde B)$, then $\mat_{t_S}(C)$ and $\tilde B$ are Dubrovin
valuation rings of $\ov S$ with the same center.  Hence,
$\tilde B \cong \mat_{t_S}(C)$, which implies that 
$\ov{\tilde B} \cong \mat_{t_S}(\ov C)$.  So, for the matrix sizes,
$t_{\tilde B} = t_S\,t_C$.  Thus, we have 
\begin{equation}\label{eq:tbts}
t_S \mid  t_{\tilde B} \,=\,  t_B
\end{equation}
(cf. \cite[Th. E(ii)]{W}). Since $\xi_{V, [A]} = n_B/t_B$, 
formula~\eqref{teoremaE} above can
be restated 
$$
n_B \, = \, n_S(t_B/t_S) (n_{\tilde B}/t_{\tilde B})\,\ell_{V,W}.
$$
Since $t_B/t_S \ge 1$ by \eqref{eq:tbts} and $n_{\tilde B}/t_{\tilde B}
\ge 1$, and $\ell_{V,W} > 1$ from (v), we have $n_B >n_S$.  

Now let~$P'$~be any prime ideal of $V$ with 
$\frak P \subsetneqq P'$, and let $V' = V_{P'}$
and $B' = {B \!\cdot\! V'}$. The prime ideals $Q$ of $V'$ coincide with the prime
ideals of $V$ lying in $P'$, and for any such $Q$, we have
$V'_Q = V_Q$ and ${B'\!\cdot\! V'_Q} = {B\!\cdot\! V_Q}$.  Hence $\frak P$ is 
maximal among the prime ideals $Q$ of $V'$ with ${B'\!\cdot\! V'_Q}$~
integral over~$V'_Q$.  Therefore, the argument just given showing that 
$n_B > n_S$ shows likewise that $n_{B'}> n_S$. Since this is true for every 
$P' \supsetneqq \frak P$, this $\frak P$ is a jump prime ideal of $v$ for $A$.

(iv) If $j(v,A) = 1$, then $J(V)$ is the only jump prime ideal of $v$, 
so $\frak P = J(V)$ and $B = {B\!\cdot\! V_\frak P}$~is integral over
$V_\frak P = V$.
\end{proof}

The following general setup occurs repeatedly in the proofs of the next
three theorems:\\

\begin{setup}\label{setup}
Let $F$ be a field with a valuation $v$ and let $A$ be a central simple
$F$-algebra with a $v$-gauge $\alpha$.  So, $v$ is defectless in $A$
by Prop.~\ref{prop:gaugeimpliesdefectless}.
Let $w$ be a coarsening of $v$.  Then $w$ is also defectless in~$A$
by Prop.~\ref{prop:coarsdefless}.
Let $V$ be the valuation ring of $v$, and $W$ the valuation ring of $w$.
Let $\tilde V = V/J(W)$, which is a valuation ring of
$\ov F^{\,w} = W/J(W)$, and let $u$ be the valuation associated to
$\tilde V$.
Let $\beta$ be the coarsening of
$\alpha$ such that $\beta|_F = w$. By Prop.~\ref{gaugealphazero},
$\alpha$ induces a $u$-gauge $\alpha_0$ on~$A^{\beta}_0$. Let
$C_1, \ldots , C_{\omega(\beta)}$ be the simple components of
$A^{\beta}_0$. Let~$\alpha_0^i$~be the restriction of
$\alpha_0$ to $C_i$ for~$i = 1, \ldots, \omega(\beta)$.
As noted after Prop.~\ref{gaugesemisimple} above, each
$\alpha_0^i$ is a $u$-gauge on $C_i$ and
\begin{equation} \label{alpha0construcao}
\alpha_0(a_1, \ldots, a_{\omega(\beta)}) \ = \
\min\big(\alpha_0^1(a_1) , \ldots,
\alpha_0^{\omega(\beta)}(a_{\omega(\beta)})\big),
\end{equation}
for all $a_1 \in C_1 , \ldots, a_{\omega(\beta)} \in C_{\omega(\beta)}.$
Moreover,
\begin{equation} \label{alphazeroci}
\gr_{\alpha_0}(A^{\beta}_0) \, \cong_g\, \gr_{\alpha_0^1}(C_1)\times
\ldots\times \gr_{\alpha_0^{\omega(\beta)}}(C_{\omega(\beta)}).
\end{equation}
 By \eqref{gradedalphazero}, the associated graded algebras of $\alpha$
and $\alpha_0$ have the same degree zero part.  It then follows from
\eqref{alphazeroci} that
\begin{equation}\label{omega}
\omega(\alpha) \,=\, \omega(\alpha_0) \,=\, \omega(\alpha_0^1) + \ldots +
\omega(\alpha_0^{\omega(\beta)}).
\end{equation}
We next  obtain a description of the simple components $C_i$.
Let $(F_{h,w}, w_h)$ be a Henselization of~$(F,w)$. Consider the scalar
extension $A_h = A \otimes_F F_{h,w} \cong \mat_n(D_h)$, where $D_h$~
is a central division $F_{h,w}$-algebra. Then, $w_h$ is defectless in
$D_h$ by definition as $w$ is defectless in~$A$, and
the Henselian valuation
$w_h$ extends to a valuation $w'$ on $D_h$. Let~$R_{w'}$~be the invariant
valuation ring of $D_h$ associated to $w'$ and let
$\overline{D_h} = R_{w'}/J(R_{w'})$, the residue division ring. On the
other hand, it follows from \cite[Cor.~1.26]{TW1} that 
$\beta_h = \beta \otimes w_h$ is a $w_h$-gauge on $A_h$ and
$\gr_{\beta_h}(A_h) \cong_g \gr_{\beta}(A)$. Write
${A_h = \End_{D_h}(M)}$ for some finite-dimensional $D_h$-vector space~
$M$.
Since $w_h$ is Henselian and is defectless in $D_h$,
\cite[Th.~3.1]{TW1} says that
$\beta_h$ is an End-gauge as in Ex.~\ref{ex;endgauge}, i.e.,  
 there is \mbox{a~$w'$-norm~$\eta$} on $M$ such that 
${\gr_{\beta_h}(A_h) = \grEnd_{\gr_{w'}(D_h)}(\gr_{\eta}(M))}$.
By \cite[Prop.~2.1]{TW1},
\begin{equation}\label{decomposicaoazerobeta}
A_0^\beta \,=\, \big(A_h\big)^{\beta_h}_0\, = \,
\left(\grEnd_{\gr_{w'}(D_h)}(\gr_{\eta}(M))\right)_0
\,\cong\, \tprod_{i=1}^{k} \mat_{r_i}(D_0),
\end{equation}
where  $D_0$ is the degree zero part of $\gr_{w'}(D_h)$, which is
$\overline{D_h}$;
also, $k$ is the number of cosets of~$\Gamma_{w'}$ in~$\Gamma_{\eta}$.
By comparing \eqref{decomposicaoazerobeta} with the
decomposition $A_0^\beta = C_1 \times \ldots \times C_{\omega(\beta)}$
from \eqref{alphazeroci},
we conclude that ${k = \omega(\beta)}$ and that each $C_i = \mat_{r_i}(D_0)$,
after re-indexing if necessary. Thus,
${Z(C_1) = \ldots = Z(C_{\omega(\beta)}) = Z(\overline{D_h})}$. Let
 $K$ denote this common field. (We also have $K = Z(\overline{S})$ 
 for any Dubrovin valuation ring $S$ of $A$ with $Z(S) =W$.) As noted in 
\eqref{eq:lVWchar}, there are
$\ell_{V,W}$~extensions of $u$ to $K$, which we denote
$u_1, \ldots, u_{\ell_{V,W}}$.
Let $U_j$ be the valuation ring of $u_j$.  By
Th.~\ref{thm:gaugeismin} applied to the $u$-gauge $\alpha_0^i$ on $C_i$,
 there exist $u_j$-gauges $\alpha_0^{ij}$ on~
$C_i$ for $j \in\{ 1,
\ldots, \ell_{V,W}\}$ such that
\begin{equation} \label{alpha0i}
\alpha_0^i(a) \ = \ \min\big(\alpha_0^{i1}(a) , \ldots,
\alpha_0^{i\ell_{V,W}}(a)\big)\qquad\text{for all \ } a \in C_i,
\end{equation}
and also
\[
\gr_{\alpha_0^i}(C_i) \ \cong_g \, \gr_{\alpha_0^{i1}}(C_i)\times
\ldots\times \gr_{\alpha_0^{i\ell_{V,W}}}(C_i).
\]
We thus have
\begin{equation}\label{omegai}
\omega(\alpha_0^i) \,= \  \omega(\alpha_0^{i1}) +  \ldots +
\omega(\alpha_0^{i\ell_{V,W}}).
\end{equation}
Hence, with \eqref{omega},
\begin{equation}\label{eq:omegaalpha}
\omega (\alpha)\, =\, \omega(\alpha_0) \,
= \, \tsum_{i = 1}^{\omega(\beta)}\,
\tsum _{j = 1}^{\ell_{V,W}}\omega(\alpha_0^{ij}).
\end{equation}

Because the $C_i$ are Brauer equivalent central simple
$K$-algebras, we have
\begin{equation}\label{xiuij1}
\xi_{U_j, [C_i]} \,= \, \xi_{U_j, [C_1]}\qquad \text{for all }
i \in \{1,\ldots, \omega(\beta)\}, \ j\in \{1, \ldots, \ell_{V,W}\}.
\end{equation}
We claim also that
\begin{equation}\label{xiuij2}
\xi_{U_j, [C_i]} \, = \, \xi_{U_1,[C_i]}\qquad \text{for all }
i \in \{1,\ldots, \omega(\beta)\}, \ j\in \{1,\ldots, \ell_{V,W}\}.
\end{equation}
We check this for
$\ov {D_h}$, which is Brauer equivalent to $C_i$.  
  Recall from \cite[Prop.~1.7]{JW} that ${K = Z(\ov{D_h})}$~is a normal field
extension of $\ov{F_{h,w}}^{\, w_h} = \ov F^{\,w}$, and hence
the Galois group $\Gal(K/\ov F^{\,w})$ acts transitively on the
set $\{u_1, \ldots, u_{\ell_{V,W}}\}$ of all extensions
of $u$ on $\ov F^{\,w}$ to $K$.  Choose $\tau \in \Gal(K/ \ov F^{\,w})$
with $u_j = u_1\circ \tau$.  By \cite[Prop.~1.7]{JW}, there is $d\in D_h^\times$
whose associated automorphism $\ov {\imath}_d\colon \ov {D_h} \to \ov {D_h}$
($\ov a \mapsto \ov {dad\inv}$ for $a\in R_{w'}$) restricts to
$\tau$ on $K$.  Then, $\ov \imath_d$ takes Dubrovin valuation rings of
$\ov {D_h}$ with center $U_j$ to those with center $U_1$;
so $\xi_{U_j, [\ov {D_h}]} = \xi_{U_1, [\ov{D_h}]}$.  Thus, from
the Brauer equivalence of $C_i $ and $\ov {D_h}$,
$$
\xi_{U_j, [C_i]} \, = \,\xi_{U_j, [\ov{D_h}]} \, = \,\xi_{U_1, [\ov{D_h}]}
 \, = \,\xi_{U_1, [C_i]},
$$
as claimed, proving \eqref{xiuij2}
\end{setup}

The following diagram illustrates some of the objects considered here.
We write $v \ge w$ to indicate that $v$ is a refinement of $w$; likewise
for $\alpha \ge \beta$. 

\begin{equation*} \xymatrixcolsep{0 pc}  \xymatrixrowsep{1 pc}
\xymatrix{
 A             & \alpha \ar@{-}[dd]& \geq & \beta \ar@{-}[dd]   \ar@{-}[rrrrd]\\
               &        &      &                       &&&& A_0^\beta  \ar@{-}[d]    
& \alpha_0  &&&&&&& C_i \ar@{-}[d]& \alpha_0^{i1},& \ldots &, \alpha_0^{i\ell_{V,W}} \\
 F \ar@{-}[uu] &  v     & \geq &  w   \ar@{-}[rrrrd]   &&&& Z(A_0^\beta)  \ar@{-}[d] 
&           &&&&&&& K  \ar@{-}[d] &u_1,\ar@{-}[u] \ar@{-}[dr]& \ldots &, u_{\ell_{V,W}} \ar@{-}[u] \\
               &        &      &                       &&&& \overline{W}             
&  u \ar@{-}[uu] &&&&&&& \overline{W} & &u \ar@{-}[ru] &           }
\end{equation*}

\begin{theorem} \label{minimalgauge1}
Let $F$ be a field with a valuation $v$ and associated valuation ring $V$,
and let $A$ be a central simple
$F$-algebra. Then, for any $v$-gauge $\alpha$ on
$A$,
$$
\omega(\alpha)\,  \ge\, \xi_{V,[A]}.
$$
\end{theorem}
\begin{proof}
We write $\xi_V$ for $\xi_{V,[A]}$.
The proof is by induction on $\xi_V$. If $\xi_V = 1$, then we clearly have
$\omega(\alpha) \ge \xi_V$. We can thus assume that 
 $\xi_V > 1.$ By \cite[Th.~16.14, p.~94]{MMU}, there exist Dubrovin valuation
rings $R_1, \ldots , R_{\xi_V}$ of $A$ having the IP such that each
$Z(R_t) = V$ and ${R_1 \cap \ldots \cap R_{\xi_V}}$ is a Gr\"ater
ring integral over~$V$.
 For each $t \in\{ 1,\ldots, \xi_V\}$, 
 by Prop.~\ref{prop:maxint}(i) and (v) there exists a Dubrovin
valuation ring $S_t$ of $A$ containing $R_t$  and minimal with the property
that $S_t$ is integral over ${W_t = Z(S_t)}$ and we have 
$\ell_{V,W_t} \geq 2$. Now by 
\cite[Cor.~16.6, p.~91]{MMU}, we have $S_1 = \ldots = S_{\xi_V}$,
hence
${W_1 = \ldots = W_{\xi_V}}$.  Let $S = S_t$ and $W = W_t$, for all $t$.
The integrality of $S$ over $W$ yields that $\xi_{W,[A]} = 1$.
We use the valuation $w$ of $W$ in Setup~$\ref{setup}$;             
so, $\beta$ is the $w$-coarsening of $\alpha$. 
Let ${\tilde {R_t} = R_t/J(S)}$. Since $\bigcap _{t = 1}^{\xi_V}R_i$~
is integral over $V$, we have $\bigcap_{t=1}^{\xi_V} \tilde {R_t}$
is integral over $\tilde V$.
 The valuation rings
$Z(\widetilde{R_1}), \ldots, Z(\widetilde{R_{\xi_V}})$
must include all the
extensions of $\widetilde{V}$ to $K = Z(\overline{S})$, because
${Z(\widetilde{R_1}) \cap \ldots \cap Z(\widetilde{R_{\xi_V}})}$ is integral
over $\widetilde{V}$. These extensions are thus the valuation rings
$U_1, \ldots, U_{\ell_{V,W}}$ of the extensions $u_1, \ldots ,
u_{\ell_{V,W}}$ of $u$
 in  Setup~\ref{setup}, but with possible 
repetitions. For any ${j \in\{ 1,\ldots, \ell_{V,W}\}}$, choose
an $R_t$ with $Z(\tilde{R_t}) = U_j$.
It follows from \eqref{teoremaE} with $B = R_t$ that
\begin{equation}\label{xiv}
\xi_V \,=\,  \xi_{W,[A]} \,  \xi_{U_j,[\overline{S}]} \, \ell_{V,W}
\, =\,  \xi_{U_j,[\overline{S}]} \, \ell_{V,W}. 
\end{equation}
 Since $\ell_{V,W} \geq 2$, we conclude that each
$\xi_{U_j,[\overline{S}]} < \xi_V$. Since each $C_i$ is Brauer equivalent
to $\overline{S}$, we have $\xi_{U_j,[\overline{S}]} = \xi_{U_j,[C_i]}$.
Since each $\alpha_0^{ij}$ constructed in the Setup~\ref{setup} is a
$u_j$-gauge on $C_i$, we have by the induction hypothesis
$\omega(\alpha_0^{ij}) \geq \xi_{U_j,[C_i]}$. The equality \eqref{xiv}
(or \eqref{xiuij2})
shows that $\xi_{U_j,[C_i]} = \xi_{U_1,[C_i]}=\xi_{U_1,[\ov S]}$ for all 
$i,j$.
 Thus, it follows from \eqref{eq:omegaalpha} that
\[
\omega(\alpha)\,=\,
\tsum_{i = 1}^{\omega(\beta)}\,
\tsum _{j = 1}^{\ell_{V,W}}\omega(\alpha_0^{ij})
 \geq\ \omega(\beta) \, \ell_{V,W} \,
\xi_{U_1,[\overline{S}]} \ \geq\ \ell_{V,W} \,
\xi_{U_1,[\overline{S}]} \,=\, \xi_V,
\]
which completes the proof.
\end{proof}

\begin{definition}\label{def:minimal}
Let $v$ be a valuation on a field $F$, and let $\alpha$ be a 
$v$-gauge on a central simple $F$-algebra $A$.
Then, we call $\alpha$  a {\it minimal gauge} if 
$\omega(\alpha)= \xi_{V,[A]}$.
\end{definition}

\begin{remark}\label{rem:omega1} 
Note that if $\omega(\alpha) = 1$, i.e., $A^\alpha_0$ is a simple 
ring, then by \cite[Prop.~2.5]{TW1} $\alpha$ is a Morandi 
value function on $A$ with associated Dubrovin valuation ring $R_\alpha$.
Conversely,  if $\alpha'$ is a Morandi value function on $A$ 
with $\alpha'|_F = v$ and $v$ is defectless in $A$, then 
by  \cite[Prop.~2.5]{TW1} $\alpha'$~is a $v$-gauge and  
$R_{\alpha'}$ is the Dubrovin valuation ring associated
to $\alpha'$; so, $A_0^{\alpha'}$ is the simple 
ring~$R_{\alpha'}/J(R_{\alpha'})$,
whence $\omega(\alpha') = 1$.

\end{remark}

\begin{theorem} \label{coarserminimal}
Let $F$ be a field with a valuation $v$ and associated valuation
ring $V$.  Let
$A$ be a central simple
$F$-algebra with a $v$-gauge $\alpha$. Let $\beta$ be a coarsening of
$\alpha$, and let $w=\beta|_F$, which is a valuation on $F$ coarser
than $v$.  Let $W$ be the valuation ring of $w$. Then,
$$
\omega(\alpha)/\omega(\beta) \,\ge\, \xi_{V,[A]}/\xi_{W,[A]}. 
$$
Consequently, if $\alpha$ is a minimal gauge, then so is $\beta$.
\end{theorem}
\begin{proof}
We use the notation of Setup~\ref{setup} with the
$v, \alpha, w, \beta$  given here.
By Th.~\ref{minimalgauge1}, we have
$\omega(\alpha_0^{ij}) \geq \xi_{U_j,[C_i]}$ for $i\in\{1,\ldots,
\omega(\beta)\}$
and $j \in\{ 1,\ldots, \ell_{V,W}\}$.
Equations \eqref{xiuij1} and \eqref{xiuij2}  above show that $\xi_{U_j,[C_i]} =
\xi_{U_1,[C_1]}$
for each $i\in\{1,\ldots, \omega(\beta)\}$ and $j \in\{ 1, \ldots,
\ell_{V,W}\}$. It then
follows from~\eqref{eq:omegaalpha} that
\[
\omega(\alpha) \,\geq\, \omega(\beta)\,\ell_{V,W}\,\xi_{U_1,[C_1]}.
\]
Choose a Dubrovin valuation ring $B$ of $A$ extending $V$, and let
$S = {B\!\cdot\!W}$.  Then, for the Dubrovin valuation ring
$\tilde B= B/J(S)$ the valuation ring
$Z(\tilde B)$ is an extension of $V/J(W)$ to $Z(\ov S)$,
so $Z(\tilde B) = U_j$ for some $j$.  Then, 
\eqref{teoremaE}~with this $B$ yields
 $\xi_{V,[A]} = \xi_{W, [A]}\ell_{V,W}\xi_{U_j,[\ov S]}
= \xi_{W, [A]}\ell_{V,W}\xi_{U_1,[C_1]}$. Hence, 
$$
{\omega(\alpha)}\big/{\omega(\beta)} \,\ge\,
\ell_{V,W}\,\xi_{U_1,[C_1]}
\, =\,{\xi_{W, [A]}\,\ell_{V,W}\,\xi_{U_1,[C_1]}}\big/{\xi_{W, [A]}} \,=\,
{\xi_{V,[A]}}\big/{\xi_{W, [A]}}.
$$

Suppose $\alpha$ is a minimal gauge. Then, $\omega(\alpha) = \xi_{V,[A]}$. 
The inequality just proved then shows that $\omega(\beta) \le \xi_{W, [A]}$.
The reverse inequality is given  by Th.~\ref{minimalgauge1}. 
Hence, $\omega(\beta) = \xi_{W, [A]}$, showing that $\beta$ is a minimal gauge.
\end{proof}

The next theorems give the fundamental connection 
between minimal gauges and Gr\"ater rings.

\begin{theorem}\label{teoremaprincipal}
Let $F$ be a field with a valuation $v$ and let $A$ be a central simple
$F$-algebra with a minimal $v$-gauge $\alpha$. Then, the gauge ring
$R_\alpha = \{ a\in A \mid
\alpha(a) \ge 0\}$ is a Gr\"ater ring of $A$ with center
the valuation ring $V$ of $v$.
\end{theorem}

\begin{proof}
We may assume that the valuation $v$ on $F$ is nontrivial.
The proof is by induction on the jump rank $j(v,A)$.
 If $j(v,A) = 1$, then any
Dubrovin valuation ring of $A$ extending $V$ is integral over $V$
by Prop.~\ref{prop:maxint}(iv). Thus, $1 = \xi_{V,[A]} =\omega(\alpha)$,
hence $\alpha$ is a Morandi value function
by Remark~\ref{rem:omega1}. Therefore, $R_{\alpha}$
is a Dubrovin valuation ring integral over $V$,
so $R_\alpha$ is a Gr\"ater ring.

We can thus assume
$j(v,A) > 1$. Let $Q$ be the $(j(v,A)-1)$-st jump prime ideal of $v$ for $A$,
let $W =V_Q$ be the corresponding valuation ring, and let $w$ be the 
valuation on $F$ associated to $W$.  
We write $\xi_V$ for $\xi_{V,[A]}$ and $\xi_W$ for 
$\xi_{W,[A]}$.
The jump prime ideals of 
$w$ for $A$ are the same as the jump prime ideals of $v$ for 
$A$, except that $J(V)$~is excluded. 
 Thus, $j(w,A) = j(v,A)- 1.$ Let $\beta$ be a coarsening of $\alpha$ such
that $\beta|_F = w$. We use these $v,w,\alpha, \beta$ in Setup~\ref{setup}. 
Since $\alpha$ is a minimal gauge, by
Th.~\ref{coarserminimal} $\beta$ is also a
minimal gauge. Thus, by the induction hypothesis,
$R_\beta$~is a Gr\"ater ring with center $W$; hence,
$R_\beta = \bigcap_{i=1}^{\xi_W}
R_{i},$ where $R_1, \ldots , R_{\xi_W}$ are Dubrovin valuation rings
having the IP such that each $Z(R_i) = W.$ By \cite[Lemma~3.2]{M3}, there 
is an isomorphism
\begin{equation}\label{isomorfismoabeta}
{A^{\beta}_0} \,\xrightarrow{\, \sim \,}\, {\tprod_{i=1}^{\xi_W}R_{i}/J(R_{i})}
\ \ \  \text{given  by}\ \ \  {x + J(R_\beta)} \, \mapsto \,
 {\big(x+ J(R_{1}), \ldots, x +
J(R_{\xi_W})\big).}
\end{equation} 

After re-indexing if necessary, we can write
$C_i = R_{i}/J(R_{i}),$ as in  Setup \ref{setup}. By 
\cite[Cor.~E]{W}, jump prime ideals of $u_j$ for $C_i$
pull back to jump primes ideals of $v$ for $A$ properly 
containing $Q$; hence,
\[
j(u_{j}, C_i) \,\le \, j(v,A) - j(w,A) \,=\, 1. 
\]
Therefore, $\xi_{U_j,[C_i]} =1$  for all $i,j$ by Prop.~\ref{prop:maxint}(iv). 
 Thus,  formula \eqref{teoremaE} reduces
to $\xi_V = \xi_W \,\ell_{V,W}$. Since $\omega(\alpha) = \xi_V$ and 
${\omega(\beta) = \xi_W}$, it follows from \eqref{eq:omegaalpha}
that $\omega(\alpha_0^{ij}) = 1$ for all ${i \in\{ 1,\ldots, \xi_W\}}$ and
${j \in\{ 1,\ldots, \ell_{V,W}\}}$. Hence each $\alpha_0^{ij}$ is a Morandi
value function by Remark~\ref{rem:omega1}. Let
$S_{ij} =$ $ \{ x \in C_i \; | \; \alpha_0^{ij}(x) \geq 0\}$, which is a
Dubrovin valuation ring of $C_i$ with $Z(S_{ij})=U_i$.  Let
\[
B_{ij} \,=\, \{a \in A \; | \; a + J(R_i) \in S_{ij} \},
\]
which is a Dubrovin valuation ring of $A$ with $Z(B_{ij}) = V$. We prove
that the set of Dubrovin valuation rings $B_{ij}$ for
$i = 1,  \ldots, \xi_W$ and $j = 1,  \ldots, \ell_{V,W}$ have the IP.
By Th.~\ref{morandiIP}, $S_{i1} , \ldots , S_{i\ell_{V,W}}$ have the IP
for any $i$. Thus, $B_{i1} , \ldots , B_{i \ell_{V,W}}$ have the IP for
each $i$, by \cite[Prop.~16.4, p.~90]{MMU}. Now let 
$i,q \in \{1,  \ldots , \xi_W \}$ with $i \neq q$. Note that
$B_{ij} \subseteq R_i$ and $B_{qr} \subseteq R_q$ for
$j,r \in\{ 1,  \ldots, \ell_{V,W}\}$. Since $R_i$~and~$R_q$ are incomparable 
and
have the IP, $B_{ij}$ and $B_{qr}$ are also incomparable and have the IP
by \cite[Th.~16.8, p.~92]{MMU}.  Therefore, the Dubrovin valuation rings
$\{B_{ij} \}_{i,j}$ are pairwise incomparable and have the IP. 

We claim that 
$R_\alpha = \bigcap_{i=1\  j=1}^{\xi_W \ \,\ell_{V,W}} B_{ij}$.
Since the gauge ring $R_\alpha$ is integral over 
its center $V$, it then follows that the intersection of the 
$B_{ij}$ is a Gr\"ater ring, which  completes the proof that
$R_\alpha$ is a Gr\"ater ring.

To prove the claim, note first that
\begin{equation}\label{inclusaoRalpha}
J(R_\beta) \,\subseteq\, J(R_\alpha) \,\subseteq \,
R_\alpha \,\subseteq\, R_\beta,
\end{equation}
because $\alpha(x)\geq 0$ implies $\beta(x) \geq 0$ and $\beta(x)>0$
implies $\alpha(x)>0$. Likewise, for each
${i \in\{ 1,  \ldots, \xi_W\}}$ and $j \in\{ 1, \ldots, \ell_{V,W}\}$ we
have
\begin{equation*}
J(R_i) \,\subseteq \,J(B_{ij}) \,\subseteq B_{ij}\, \subseteq R_i.
\end{equation*}
Hence, using \eqref{isomorfismoabeta} for the first equality as $A_0^\beta = 
R_\beta/J(R_\beta)$
\begin{equation}\label{inclusaoBij}
J(R_\beta) \,=\, \tbigcap_{i=1}^{\xi_W} J(R_i) \,\subseteq \,
\tbigcap_{i, j = 1}^{\xi_W, \ell_{V,W} } B_{ij} \,\subseteq \,
\tbigcap_{i=1}^{\xi_W} R_i \,= \,R_\beta.
\end{equation}
Let $a \in R_\beta \backslash J(R_\beta)$. In view of  \eqref{inclusaoRalpha} and
\eqref{inclusaoBij}, the proof of the claim will be completed by showing 
that
$a \in R_\alpha$ if and only if
$a \in\bigcap_{i=1 \  j=1}^{\xi_W \ \,\ell_{V,W}} B_{ij}$. We identify
$A_0^\beta$ with ${R_{1}/J(R_{1}) \times \ldots \times
R_{\xi_W}/J(R_{\xi_W})}$, via the isomorphism~\eqref{isomorfismoabeta}.
Thus, by the definition of the gauge $\alpha_0$ on $A_0^\beta = 
R_\beta/J(R_\beta)$ and by~\eqref{alpha0construcao},  
\begin{equation}\label{alphaamin}
\alpha(a)\,=\,\alpha_0\big(a + J(R_\beta)\big)
          \, = \, \alpha_0 \big(a + J(R_1), \ldots , a +J(R_{\xi_W})\big)
          \, = \, \min_{1 \leq i \leq \xi_W}\big(\alpha_0^i(a+ J(R_i))\big).
\end{equation}
Hence, $\alpha(a) \geq 0$ if and only if $\alpha_0^i(a+ J(R_i)) \geq 0$
for $i =
1, \ldots, \xi_W$. By \eqref{alpha0i}, the gauge ring
$R_{\alpha_0^i} = \bigcap_{j=1}^{\ell_{V,W}}S_{ij}.$ Thus,
$a \in R_\alpha$ if and only if
$a + J(R_i) \in \bigcap_{j=1}^{\ell_{V,W}}S_{ij},$ for
$i = 1, \ldots, \xi_W.$ But $a + J(R_i) \in S_{ij}$ if and only if
$a \in B_{ij}$. Therefore, $a \in R_\alpha$ if and only if
$a \in \bigcap_{i=1 \  j=1}^{\xi_W \ \,\ell_{V,W} } B_{ij}$.
This proves the claim, which, as noted above, implies that 
$R_\alpha$ is a Gr\"ater ring.
\end{proof}

\begin{remark}
Note that the Dubrovin valuation rings $B_{ij}$
 of Th.~\ref{teoremaprincipal} are uniquely determined
by~$\alpha$, because they are the localizations of $R_\alpha$ re its
maximal ideals.
\end{remark}

The following result shows that Th.~\ref{teoremaprincipal} has a converse.

\begin{theorem}\label{teoremaprincipalvolta}
Let $F$ be a field with a valuation $v$ and associated valuation 
ring $V$. Let $A$ be a central simple
$F$-algebra such that $v$ is defectless in $A$.  Then a subring 
$C$ of $A$ is a Gr\"ater ring of~$A$ with center $V$
if and only if $C = R_\alpha$ for some 
minimal $v$-gauge~$\alpha$ on~$A$.
\end{theorem}

The proof requires the existence of minimal gauges on defectless
$F$-algebras, which will be proved in \S4.

\begin{proof}
If $\alpha$ is a minimal $v$-gauge on $A$, we have seen in 
Th.~\ref{teoremaprincipal} that $R_\alpha$ is a Gr\"ater ring with 
center~$V$.

For the converse, let $C$ be a Gr\"ater ring of $A$
with center $V$.  Because $v$ is defectless in $A$, 
Th.~\ref{existenceminimalgauges}~below shows that there exists 
a minimal $v$-gauge $\beta$ on $A$.
By Th.~\ref{teoremaprincipal},
its gauge ring $R_\beta$ is a Gr\"ater ring of $A$ with center $V$.
By \cite[Th.~16.15, p.~96]{MMU}, $C$ and $R_\beta$ are conjugate in $A$, 
i.e.,
$C = q R_\beta q^{-1}$ for some $q \in A^\times$. Composition of $\beta$
with the inner automorphism $\imath_q\colon A \rightarrow A$ defined by
$x \mapsto q^{-1} x q$ yields a minimal $v$-gauge
$\alpha = \beta \circ \imath_q$ on~$A$ such that
$R_\alpha = q R_\beta q^{-1} = C$.
\end{proof}

 Note  that the gauge 
$\alpha$ of  the theorem is not uniquely determined by~$C$, as the
example in \S5~below demonstrates.

\section{Existence of minimal gauges}

In this section we prove the existence  of minimal
gauges on defectless semisimple algebras. First, we extend the concept 
of minimal gauge to semisimple
algebras.

\begin{proposition}\label{gaugeminimalsimple}
Let $(F,v)$ be a valued field and $A$ be a finite-dimensional simple
$F$-algebra. Let $v_{1}, \ldots, v_{r}$ be all the extensions of
$v$ to $Z(A)$, and let $V_i$ be the valuation ring of $v_i$. Let $\alpha$
be a $v$-gauge on $A$.
Then,
\begin{equation}\label{omegasimples}
\omega(\alpha) \,\geq\, \xi_{V_1, [A]} + \ldots + \xi_{V_r, [A]}.
\end{equation}
\end{proposition}
\begin{proof}
Let $\alpha_i$ be the $v_i$-component of $\alpha$ for $i = 1, \ldots,
r$.  Then, by the graded algebra isomorphism~\eqref{eq:grisocomps} and 
the inequality of Th.~\ref{minimalgauge1} for each $i$,
\begin{equation}\label{eq:omeganoncentral}
\omega(\alpha) \, = \, \omega(\alpha_1) + \ldots + \omega(\alpha_r)
\, \ge \,\xi_{V_1, [A]} + \ldots + \xi_{V_r, [A]}.
\end{equation}
\end{proof}

\begin{definition}\label{def:minimalss}
Let $v$ be a valuation on a field $F$, and let 
$\alpha$ be a $v$-gauge on a simple (finite-dimensional)
$F$-algebra $A$. We say that $\alpha$ is a {\it minimal} 
 $v$-gauge on  $A$ if we have equality
in~\eqref{omegasimples}. Note that \eqref{eq:omeganoncentral}
shows that $\alpha$ is a minimal $v$-gauge if and only if each 
component $\alpha_i$
of $\alpha$ is a minimal $v_i$-gauge.
More generally, if $A$ is semisimple, say
$A = A_1 \times \ldots \times A_k$ with each $A_i$ simple,
and $\beta$ is a $v$-gauge
on $A$, we say that $\beta$ is a \emph{minimal $v$-gauge} on $A$ if each
$\beta_i = \beta |_{A_i}$ (as in \eqref{betai}) is a minimal $v$-gauge
on $A_i$.
\end{definition}

\begin{theorem} \label{existenceminimalgauges}
If $v$ is a valuation on $F$ defectless in a finite-dimensional
semisimple $F$-algebra $A$, then there
exists a minimal $v$-gauge $\alpha$ on $A$ with
$\Gamma_\alpha \subseteq \divh{\Gamma_v}$.
\end{theorem}

The general method in proving the theorem is to build up 
$\alpha$ inductively from minimal gauges for $A$ for 
valuations on $F$ coarser than $v$.  The proof will begin after 
Prop.~\ref{prop:jump1} and be completed after 
Lemma~\ref{lemma4existence}.

\begin{lemma}\label{lemma2existence}
Let valuation $w$ be a coarsening of $v$ on $F$. Let $A$ be a semisimple 
$F$-algebra with a $w$-gauge $\beta$. Let $F'$ be a field containing $F$ 
with a valuation $w'$ which is an immediate extension of ~$w$, and let 
$v'$ be the extension of $v$ to $F'$ that refines $w'$, so $v'$ is an 
immediate extension of $v$. Let $A' = A \otimes_F F'$ and let 
$\beta' = \beta \otimes w',$ which is a $w'$-gauge on $A'$. Suppose 
$A'$ has a $v'$-gauge $\alpha'$ whose $w'$-coarsening is $\beta'$. 
Then $\alpha = \alpha' |_A$ is a $v$-gauge on $A$ with $w$-coarsening 
$\beta$, and $\gr_\alpha (A) \cong_g \gr_{\alpha'}(A')$.
\end{lemma}
\begin{proof}
Let $u = v/w$, which is the valuation on $\overline{F}^w$ induced  
by $v$ on $F$; likewise let $u'= v'/w'$ on $\overline{F'}^{w'},$ which 
coincides with $u$ under the canonical isomorphism 
$\overline{F}^w \cong \overline{F'}^{w'}.$ Because $w'$ is a immediate 
extension of $w$, by \cite[Cor. 1.26]{TW1}, 
$\beta'= \beta \otimes w'$ is a $w'$-gauge on $A'$ with
\[
\gr_{\beta'}(A') \,\cong_g\, 
\gr_{\beta}(A) \otimes_{\gr_w(F)} \gr_{w'}(F') \,\cong_g\, \gr_\beta(A).
\]
For $\Gamma = \divh{\Gamma_{\alpha'}}$ and 
$\Lambda = \divh{\Gamma_{\beta'}}$, let 
$\varepsilon \colon \Gamma \rightarrow \Lambda$ be the map associated 
to the $w$-coarsening of $v$. For each $\lambda \in \Lambda$ we have 
the $\lambda$-component ${A'}_\lambda^{\beta'}$ of $\gr_{\beta'}(A')$, 
and the $u'$-value function $\alpha'_{\lambda}$ on ${A'}_\lambda^{\beta'}$ 
defined by
\[
\alpha'_\lambda (a + A'^{\beta'}_{> \lambda}) \,=\, \left \{
\begin{array}{cl}
\alpha'(a) & \text{\; if \:} \beta'(a) = \lambda,\\
\infty & \text{\; if \:} \beta'(a)> \lambda.
\end{array}\right.
\]
Since $\alpha'$ is a $v'$-norm, each $\alpha'_\lambda$ is a $u'$-norm 
on ${A'}_\lambda^{\beta'}$ by \cite[Prop. 4.3]{TW2}. Likewise 
$\alpha = \alpha' |_A$ induces the $u$-value function $\alpha_\lambda$ 
on $A_\lambda^\beta$. Since $\beta = \beta' |_A$, we can view 
${A}_\lambda^{\beta} \subseteq {A'}_\lambda^{\beta'}$ via the canonical 
inclusion; then clearly 
$\alpha_\lambda = \alpha'_\lambda |_{{A}_\lambda^{\beta}}$. But since 
the canonical inclusion $\gr_\beta(A) \hookrightarrow \gr_{\beta'}(A')$ 
is a graded isomorphism, we have  
${A}_\lambda^{\beta} = {A'}_\lambda^{\beta'}$.  So, 
$\alpha_\lambda = \alpha'_\lambda$, which is a $u$-norm on 
${A}_\lambda^{\beta}$. Since in addition  $\beta$ is a $w$-norm on $A$, by 
\cite[Prop. 4.3]{TW2} $\alpha$ is a $v$-norm on $A$. Moreover, 
$\alpha = \alpha' |_A$ is surmultiplicative since $\alpha'$ is 
surmultiplicative. There is a canonical algebra monomorphism 
$\iota\colon \gr_\alpha(A) \hookrightarrow \gr_{\alpha'}(A')$. 
Take any~  
$\gamma \in \Gamma$. It follows from the definitions that
$A_{\gamma}^\alpha = 
\big(A_{\varepsilon(\gamma)}^\beta\big)
_\gamma^{\alpha_{\varepsilon(\gamma)}}$.
 Since 
$A_{\varepsilon(\gamma)}^\beta = {A'}_{\varepsilon(\gamma)}^{\beta'}$
and $\alpha_{\varepsilon(\gamma)} = \alpha'_{\varepsilon(\gamma)}$, 
we thus have
\[
A_{\gamma}^\alpha \,=\, 
\big(A_{\varepsilon(\gamma)}^\beta\big)_\gamma^{\alpha_{\varepsilon(\gamma)}}
\, =\,
 \big({A'}_{\varepsilon(\gamma)}^{\beta'}\big)_\gamma^{{\alpha'}_{\varepsilon(\gamma)}} 
\,=\, {A'}_{\gamma}^{\alpha'};
\]
hence, $\iota$ is a graded isomorphism. Thus, $\gr_\alpha(A)$ is graded 
semisimple, since this is true for $\gr_{\alpha'}(A')$, as $\alpha'$ is 
a $v'$-gauge. Therefore, $\alpha$ is a $v$-gauge. The $w$-coarsening of 
$\alpha = \alpha' |_A$ is $\beta'|_A = \beta$.
\end{proof}

\begin{lemma}\label{lemma3existence}
Let $A$ be a central simple $F$-algebra with $j(v,A) = n>1$. Let $P$ be 
the $(n-1)$-st jump prime ideal of $v$ for $A$, and let $W= V_P$ with its 
associated valuation $w$. Let $S$ be a Dubrovin valuation ring of $A$ with 
$Z(S) = W$, and let $\overline{S} = S/J(S)$. Let $u = v/w$, the residue 
valuation on~$\overline{F}^{\,w}$ induced by $v$, and let $u_1, \ldots, u_r$ 
be the valuations on $Z(\overline{S})$ extending $u$. Then 
$u_1, \ldots, u_r$ are pairwise independent valuations.  
\end{lemma}

\begin{proof}
Let $T$ and $B$ be Dubrovin valuation rings of $A$ with 
$B \subseteq T \subsetneqq S$ and $Z(B) = V$. Let $Y = Z(T)$ and let 
$y$ be the valuation of $Y$. We have  
$V \subseteq Y \subsetneqq W$. Let $n_B$ and $t_B$ be as given in 
\eqref{eq:defnB} and \eqref{eq:deftB}. Let
$
\widetilde{B} = B/J(S)
$
which is a Dubrovin valuation ring of $\ov S = S/J(S)$.
Recall from~\eqref{eq:tbts} that $t_{\tilde B} = t_B$.
(The proof of \eqref{eq:tbts} is valid for any overring $S$ of a 
Dubrovin valuation ring $B$.) Using this
and $\xi_{V,[A]} = n_B/t_B$, formula \eqref{teoremaE}
yields
$$
n_B\, = \,\xi_{W,[A]}\, n_{\tilde B}\,\ell_{V,W}.
$$ 
Likewise, by replacing $B$ by $T$, we have 
$$
n_T\, = \,\xi_{W,[A]}\, n_{\tilde T}\,\ell_{Y,W}.
$$ 
Hence, 
\begin{equation}\label{eq:TvsB}
n_B/n_T \, = \, \big(n_{\tilde B}/n_{\tilde T}\big) \,\big(\ell_{V,W}/
\ell_{Y,W}\big).
\end{equation}
Because $y$ is a coarsening of $v$, the Henselization 
$F_{h,y}$ embeds in $F_{h,v}$. Hence $n_B/n_T$ is a positive
integer, as likewise is $n_{\tilde B}/n_{\tilde T}$.  Also, because 
the valuation ring $\tilde V = V/J(W)$ is a refinement  of 
$\tilde Y = Y/J(W)$ in $\ov W$, this  $\tilde V$ has at least as many extensions
to $Z(\ov S)$ as $\tilde W$, i.e., $\ell_{V,W}\ge\ell_{Y,W}$.  But,
since there are no jump prime ideals between $J(Y)$ and $J(V)$, we 
have $n_B = n_T$. Thus, in \eqref{eq:TvsB} the left side equals 
$1$ and the right side is a product of positive integers.  Hence,
$\ell_{V,W} = \ell_{Y,W}$.

The last equality says that the number of 
extensions $U_i$ of $U = V/J(W)$ to $Z(\overline{S})$ equals the number 
of extensions of the coarser valuation $Y/J(W)$ to $Z(\overline{S})$. 
Hence, any two distinct $U_i$ and $U_j$ have distinct coarsenings to 
extensions of $Y/J(W)$. Because this is true for every valuation ring 
$Y$ with $V \subseteq Y \subsetneqq W$, the finest common coarsening 
of $U_i$ and $U_j$ must be the trivial valuation ring. Hence, $U_i$ and 
$U_j$ are independent valuation rings in $Z(\overline{S})$, so their 
corresponding valuations $u_i$ and $u_j$ are independent.
\end{proof}

The next proposition is the most difficult step in the proof of 
Th.~\ref{existenceminimalgauges}.  Here is the setup for the proposition:
Let $A$ be a central simple $F$-algebra with $v$ defectless in $A$ and 
$j(v,A)=2$. Let $W = V_P$, where $P$ is the first jump prime ideal of 
$v$ for $A$, and let $w$ be the valuation of $W$. Assume 
that $w$ is Henselian. 
Let $\beta$ be a $w$-gauge of $A$ with 
$\Gamma_\beta \subseteq \divh{\Gamma_w}$. Write $A = \End_D(M)$, where 
$D$~is the division algebra associated to $A$, and $M$ is a 
finite-dimensional right $D$-vector space. Let~$y$~be the valuation on 
$D$ extending $w$ on $F$, and let $\overline{D} = \overline{D}^{\,y}$. Let 
$u$ be the residue valuation~$v/w$ on~$\overline{F}^{\,w}$ induced by $v$, 
and let $u_1, \ldots, u_r$ be the extensions of $u$  to $Z(\overline{D})$. 
Let field $S$ be the separable closure of $\overline{F}^{\,w}$ in 
$Z(\overline{D})$. Recall from \cite[Prop.~1.7]{JW} that $Z(\ov D)$ is normal
over $\ov F^{\,w}$ and that $S$ is abelian Galois over $\ov F^{\,w}$.   
Let $K$ be the decomposition field of $u_1 |_S$ over 
$u$ (so $K$ is also the decomposition field of each $u_i |_S$ over $u$, 
as $\Gal(S/\overline{F}^{\,w})$ is abelian). For basic properties of
decomposition fields, see \cite[pp.~133-136]{Ef}. 
Let $L$ be a subfield of $D$ 
that is an inertial lift of $K$ over $F$.  That is, $\ov L^{\,y} = K$ and
$\DIM LF = \DIM K{\ov F^{\,w}}$. 
Such an $L$ exists (and is 
unique up to isomorphism) because  $w$ is Henselian and $K$ is separable 
over $\ov F^{\,w}$,  cf.~\cite[p.~135]{JW}. 
 Let $v_1, \ldots, v_r$  
~be the extensions of $v$ to $L$. Let~$C$~be the centralizer  $C_D(L)$.

\begin{proposition}\label{prop1existence} 
In the situation just described,
let $\alpha$ be a $v_1$-gauge on $C = C_D(L)$
with ${\Gamma_\alpha \subseteq \divh{\Gamma_v}}$.
Then, $A$ has a $v$-gauge $\varphi$ with 
$w$-coarsening $\beta$ such that 
$\Gamma_\varphi \subseteq \divh{\Gamma_v}$ and
\[
\omega(\varphi) \,=\, r\hsp \omega(\alpha) \hsp\omega(\beta),
\]
where $r$ is the number of extensions of 
$u$ to $Z(\ov D)$. Moreover $v_1$ on $L$ 
has extension number $1$ in~$C$. Hence, $\alpha$ exists and can be chosen 
with 
$\omega(\alpha) = 1$.
\end{proposition}
\begin{proof}
Let $A' = \End_C(M) \cong A \otimes_F L$, which contains 
$A = \End_D(M)$ canonically, as an $F\text{-}$subalgebra. We will build 
$\varphi$ as the restriction to $A$ of a suitable End-gauge on $A'$.

Since $K$ is Galois over $\overline{F}^{\,w}$, its inertial lift $L$ is 
Galois over $F$, with $\Gal(L/F) \cong \Gal(K/\overline{F}^{\,w})$, 
cf.~\cite[p.~135]{JW}. Let $G =  \Gal(L/F)$.  Since $\Gal(S/K)$ is 
the decomposition group for $u_1|_S$ over $u$ and every valuation on $S$
has a unique extension to the purely inseparable field extension $Z(\ov D)$
of $S$, we have    
\begin{align*}
r \, &= \, \big|\{\text{extensions of $u$ from $\ov F^{\,w}$ to $Z(\ov D)$}\}\big| \, = \,
\big|\{\text{extensions of $u$ from $\ov F^{\,w}$ to $S$}\}\big|\\
 &= \, \IND{\Gal(S/\ov F^{\,w})}{\Gal(S/K)} \, = \, |G| \, = \, \DIM LF.
\end{align*}
For each $i \in \{1,2, \ldots, r\}$, since $K$ is the decomposition field 
of each $u_i|_S$ over $u$, this $u_i|_S$ is the unique extension of 
$u_i|_K$ to $S$, cf.~\cite[Prop.~15.1.2(b), p.~134]{Ef}.
Hence, $u_1|_K, \ldots, u_r|_K$ are all distinct valuations
of $K$.   
Let $w_L = y|_L$, which is the unique extension of 
the Henselian valuation $w$ to $L$. The 
valuations $v_i$ of $L$ extending $v$ on $F$ are the composite valuations 
$v_i = u_i|_K \ast w_L$. Therefore, there are $r$ distinct $v_i$, since 
the $u_i|_K$ are distinct. The group 
$G$ acts transitively on the $v_i$ since it acts transitively on the 
$u_i |_K$, and this action is simply transitive as $|G| = r$. Since 
$\DIM LF = r$ the Fundamental Inequality shows that each $v_i$ is an 
immediate extension of $v$. Thus,
\begin{equation}\label{eq:grL}
\gr_{v_i}(L) \,\cong_g\, \gr_{v}(F) \qquad \text{ for } i=1, \ldots, r.
\end{equation}
Note that since $D$ is Brauer equivalent to $A$, $j(v,D) = j(v,A) = 2$
and $v$ has the same jump prime ideals for $D$ as for $A$; these
are $j(W)$ and $j(V)$.  The valuation ring $Y$ of $y$ on $D$ is a
Dubrovin valuation ring of $D$ with $Z(Y) = Y \cap F = W$ and 
$Y/J(Y) = \ov D$.  By Lemma~\ref{lemma3existence} (with $D$ (resp.~$Y$) for the 
$A$ (resp.~$S$) of the lemma)
the valuation rings $u_1, \ldots, u_r$ of $Z(\ov D)$ are 
pairwise independent; hence, $u_1|_K, \ldots, u_r|_K$ are pairwise independent;
hence, the finest common coarsening 
of any distinct $v_i$ and $v_j$ is $w_L$.  
 Let $\Gamma = \divh{\Gamma_v}$ 
and $\Lambda = \divh{\Gamma_w}$. So each $\Gamma_{v_i} \subseteq 
\Gamma$, and for the valuation $y$ on $D$ 
extending $w$ we have $\Gamma_y \subseteq \Lambda$.
Since $w$ is a coarsening of $v$
on $F$, there is an epimorphism $\Gamma_v \to \Gamma_w$.  Let 
$\varepsilon\colon \Gamma \rightarrow \Lambda$ be the unique 
extension of this 
map to $\Gamma$; then,
$\varepsilon$ is surjective.  
 Let $\Gamma_{C,y}$ 
be the value group of $y|_C$.

By Skolem-Noether, for each $\rho \in G$ there is $d_\rho \in D^\times$ 
with $d_\rho \ell d_\rho^{-1} = \rho(\ell)$  for all $\ell \in L$. Since 
$d_\rho L d_\rho^{-1} = L$, we have $d_\rho C d_\rho^{-1} = C$. Moreover,
\[
D \,=\,\tbigoplus_{\rho \in G} d_\rho C \,=\, \tbigoplus_{\rho \in G} C d_\rho.
\]
(This is a standard fact about generalized crossed product algebras 
(see \cite[Th.~1.3]{T} or \cite[p.~156]{JW})).
Since $y$ is a valuation on the division algebra $D$, there 
is a canonical epimorphism $\Gamma_y \to 
\Gal\big(Z(\overline{D})/\overline{F}^{\,w}\big)$ induced by 
conjugation by elements of $D^\times$ (see 
\cite[Prop.~1.7]{JW}). Let ${\zeta\colon
\Gamma_y \to \Gal(K/\overline{F}^{\,w})}$ be the composition
of epimorphisms
\[
\zeta\colon \Gamma_y \,\longrightarrow\,
 \Gal\big(Z(\overline{D})/\overline{F}^{\,w}\big) \,\longrightarrow\, 
\Gal(K/\overline{F}^{\,w}).
\]
Since $K = \overline{L}$, we have $\Gamma_{C,y} \subseteq \ker(\zeta)$. 
Thus,
\[
r \,=\, [L:F] \,=\, [D:C] \,\geq\, | \Gamma_y:\Gamma_{C,y}|\, \geq \,
| \Gamma_y:\ker(\zeta)| \,=\, | \Gal(K/\overline{F}^{\,w})| \,=\,
 [K:\overline{F}^{\,w}] \,=\, r.
\]
So, equality holds throughout, showing that $D$ is totally ramified 
over $C$ and $\ker(\zeta) = \Gamma_{C,y}$. Since $\zeta(y(d_\rho)) =\rho$ 
for all $\rho \in G$, the values $y(d_\rho)$ are distinct modulo 
$\Gamma_{C,y}$. Thus, there is a disjoint union decomposition
\begin{equation}\label{eq:Gammay}
\Gamma_y \,=\, \tbigsqcup_{\rho \in G} y(d_\rho)+ \Gamma_{C,y}.
\end{equation}
Let $\delta_\rho = y(d_\rho) \in \Gamma_y$ for all $\rho \in G$. 
The disjoint union decomposition for $\Gamma_y$ in \eqref{eq:Gammay} 
shows that for any $d = \sum_{\rho \in G} d_\rho c_\rho \in D$ with all 
$c_\rho \in C$,
\[
y\big(\tsum_{\rho \in G} d_\rho c_\rho\big) \,=\, 
\min\limits_{\rho \in G}\big(y(d_\rho)+ y(c_\rho)\big)\,=\, 
\min\limits_{\rho \in G}\big(y(c_\rho)+ \delta_\rho\big).
\]
Take any $\rho \in G$, and define  $v_\rho \colon L \to 
\Gamma\cup \{ \infty\}$ by 
\[
v_\rho(\ell) \,=\, v_1(d_\rho^{-1}\ell d_\rho) \,=\, v_1(\rho^{-1}(\ell))
\qquad\text{for all} \ \ell \in L.
\]
Then, $v_\rho$ is a valuation of $L$ extending $v$ on $F$. Since $G$ acts 
simply transitively on $\{v_1, \ldots, v_r\}$ it follows that 
$\{v_1, \ldots, v_r\} = \{v_\rho \, | \, \rho \in G\}$ and the $v_\rho$ 
are distinct for distinct choices of $\rho$.

Because $w$ on $F$ is Henselian, the $w$-gauge on $A = \End_D(M)$ is an 
End-gauge by \cite[Th.~3.1]{TW1}, i.e., $\beta = \End(\theta)$ 
as in Ex.~\ref{ex;endgauge} for  
some $y$-norm $\theta\colon M \rightarrow \Lambda \cup \{\infty\}.$ Let 
$(m_1, \ldots , m_n)$ be a splitting base for $\theta$ of the $D$-vector 
space $M$, and let
\[
\pi_j \,=\,\theta(m_j) \qquad \text{for }\, j=1, \ldots, n.
\]
So $\theta( \sum_{j=1}^n m_jd_j) = 
\min\limits_{1 \leq j \leq n}(\pi_j + y(d_j))$ for all $d_j \in D$. Hence, 
for any $c_{j\rho} \in C$ for $j=1,\ldots, n$ and $\rho \in G$,
\begin{eqnarray*}
\theta\big(\tsum_{j=1}^n \tsum_{\rho \in G} m_j d_\rho c_{j\rho}\big) 
& = & 
\min\limits_{1 \leq j \leq n} 
\big(\pi_j +y(\tsum_{\rho \in G}d_\rho c_{j\rho}) \big)\\
& = & 
\min\limits_{1 \leq j \leq n} 
\big(\pi_j + \min\limits_{\rho \in G} 
\big(\delta_\rho + y(c_{j\rho})\big)\big)\\
& = & 
\min\limits_{1 \leq j \leq n,\; \rho \in G}
 \big(\pi_j + \delta_\rho + y(c_{j\rho})\big).
\end{eqnarray*}
This shows that $(m_jd_\rho)_{j=1, \,\rho \in G}^n$ is a splitting base 
for $\theta$ as a $y|_C$-norm on $M$. Since we can adjust the~$m_j$ by 
multiplication by any element of $D^\times$, we may assume that 
$\pi_i = \pi_j$ whenever ${\pi_i + \Gamma_y = \pi_j+ \Gamma_y}$.

For each $\rho \in G$, pick $\gamma_\rho \in \Gamma$ with
\[
\varepsilon(\gamma_\rho) \,=\, \delta_\rho.
\]
For $j=1, \ldots, n$, pick $\mu_j  \in \Gamma$ with
\[
\varepsilon(\mu_j) \,=\, \pi_j.
\]
Choose the $\mu_j$ so that $\mu_i = \mu_j$ whenever $\pi_i = \pi_j$. 
We now use the $v_1$-gauge $\alpha$ on $C$ to define an 
\lq\lq$\alpha$-$v_1$-norm" $\eta\colon M \rightarrow 
\Gamma \cup \{\infty\}$ as in Lemma~\ref{lemma1existence}. For all $c_{j\rho} \in C$, set
\[
\eta\big(\tsum_{j=1}^n \tsum_{\rho \in G} m_j d_\rho c_{j\rho}\big) 
=\, \min\limits_{1 \leq j \leq n,\; \rho \in G} \big(\mu_j + \gamma_\rho + \alpha(c_{j\rho})\big).
\]
Since $\alpha$ is a $v_1$-gauge on $C$, its coarsening 
$\varepsilon \circ \alpha$ is a $w_L$-gauge on $C$ by 
\cite[Prop.~4.3]{TW2}.  But since the valuation $w_L$ on 
$L$  extends to a valuation on $C$, by \cite[Cor.~3.2]{TW1}
that valuation is the only $w_L$-gauge on $C$;    
 hence, $\varepsilon \circ \alpha = y|_C$. Thus,
\[
\varepsilon \circ \eta
\big(\tsum_{j=1}^n \tsum_{\rho \in G} m_j d_\rho c_{j\rho}\big) 
\,=\, \min\limits_{1 \leq j \leq n,\; \rho \in G} 
\big(\pi_j + \delta_\rho + y(c_{j\rho})\big) \,=\,
 \theta\big(\tsum_{j=1}^n \tsum_{\rho \in G} m_j d_\rho c_{j\rho}\big),
\]
i.e., $\varepsilon \circ \eta = \theta$. Now, let $\psi = \End(\eta)$, 
the $v_1$-$\End$-gauge on $A' = \End_C(M)$ determined by $\eta$, as in 
Lemma~\ref{lemma1existence}. So, for $f \in A'$,
\[
\psi(f) \,=\, \min\limits_{m \,\in M \setminus\{0\}}
\big(\eta(f(m)) - \eta(m)\big) \,=\, 
\min\limits_{1 \leq j \leq n, \, \rho \in G}
\big(\eta(f(m_jd_\rho)) - \mu_j - \gamma_\rho\big).
\]
Let $\varphi = \psi|_A$. We will
show that $\varphi$ is the desired  $v$-gauge on $A$.

We claim first  that $\varphi$ is a $v$-norm. For this, take
any~$f \in A$ and 
write
\begin{equation}\label{eq:fmj}
f(m_j) \,=\, \tsum_{i=1}^n m_i d_{ij} \,=\, 
\tsum_{i=1}^n \tsum_{\sigma \in G} m_i d_\sigma c_{ij\sigma},
\end{equation}
where each $d_{ij} \in D$ and each $c_{ij\sigma}\in C$. For any 
$\sigma, \rho \in G$, we have 
$d_{\sigma\rho}^{-1}d_{\sigma}d_{\rho}$ centralizes $L$, so lies in~$C$. 
That is,
$d_{\sigma}d_{\rho} = d_{\sigma\rho}t_{\sigma, \rho}$, for some 
$t_{\sigma, \rho} \in C^\times$. Then,
\begin{align}\label{eq:varphif}
\begin{split}
\varphi(f) \, = \, \psi(f) \, &= \  
\min\limits_{1 \leq j \leq n, \, \rho \in G}
\Big(\eta\big(\big(\tsum_{i=1}^n \tsum_{\sigma \in G} 
m_i d_\sigma c_{ij\sigma}\big)d_\rho\big) - \mu_j - \gamma_\rho\Big) 
\\
&= \ \min\limits_{1 \leq j \leq n, \, \rho \in G}
\Big(\eta\big(\tsum_{i=1}^n \tsum_{\sigma \in G} 
m_i d_{\sigma\rho}(t_{\sigma, \rho}d_\rho^{-1} c_{ij\sigma}d_\rho)\big) 
- \mu_j - \gamma_\rho\Big)
\\
&= \ \min\limits_{1 \leq j \leq n, \, \rho \in G}
\Big( \min\limits_{1 \leq i \leq n, \, \sigma \in G}
\big( \alpha(t_{\sigma, \rho}d_\rho^{-1} c_{ij\sigma}d_\rho) 
+ \mu_i + \gamma_{\sigma\rho}\big) - \mu_j - \gamma_\rho\Big)
\\
&= \ \min\limits_{1 \leq i, j \leq n; \, \rho, \sigma \in G}
\big(\alpha(t_{\sigma, \rho}d_\rho^{-1} c_{ij\sigma}d_\rho) 
+ \mu_i  - \mu_j  + \gamma_{\sigma\rho}-  \gamma_\rho\big).
\end{split}
\end{align}

The choice of $D$-base $(m_j)_{j=1}^n$ of $M$ gives an isomorphism 
$A = \End_D(M) \cong \mat_n(D)$, which we use to interpret formula 
\eqref{eq:varphif}. In $A$ we have the \lq\lq matrix units" $e_{ij}$ 
for $i,j \in \{1,\ldots, n\},$ defined by
\[
e_{ij}(m_j) \,=\, m_i \qquad \text{ and } \qquad e_{ij}(m_k) \,=\, 0 \  
\text{ for } k \neq j.
\]
We also have an embedding $\lambda\colon D \rightarrow A$ given by 
$d \mapsto \lambda_d$, where $\lambda_d (m_j) = m_j d$ for all $j$. 
(So, $\lambda_d (m_jb) = m_j db$ for all $b \in D$, so  
$\lambda_d \circ  \lambda_b =  \lambda_{db}$. Clearly  
$\lambda_d \circ e_{ij} = e_{ij} \circ  \lambda_d$ for all $i,j$ and 
all $d \in D$.) This $\lambda$ corresponds to the diagonal embedding 
of $D$ in $\mat_n(D)$.
 The $f$ in \eqref{eq:fmj} above can be described as 
$f = \sum_{i=1}^n  \sum_{j=1}^n \sum_{\sigma \in G} 
e_{ij} \circ \lambda_{d_\sigma} \circ \lambda_{c_{ij\sigma}}$. Thus, 
formula \eqref{eq:varphif} becomes
\begin{equation}\label{eq:varphisum}
\varphi\big(\tsum_{i=1}^n  \tsum_{j=1}^n \tsum_{\sigma \in G} e_{ij} 
\circ \lambda_{d_\sigma} \circ \lambda_{c_{ij\sigma}}\big) \,=\,
 \min\limits_{1 \leq i, j \leq n; \, \rho, \sigma \in G}
\big(\alpha(t_{\sigma, \rho}d_\rho^{-1} c_{ij\sigma}d_\rho) + \mu_i  - \mu_j  + \gamma_{\sigma\rho}-  \gamma_\rho\big).
\end{equation}
This holds for all choices of $c_{ij\sigma}$ in $C$.  In particular,
fixing $i,j, \sigma$, 
\[
\varphi( e_{ij} \circ \lambda_{d_\sigma} \circ \lambda_c ) \,=\, 
\min\limits_{ \rho \in G}
\big(\alpha(t_{\sigma, \rho}d_\rho^{-1} c d_\rho) 
+ \mu_i  - \mu_j  + \gamma_{\sigma\rho}-  \gamma_\rho\big), 
\quad \text{for all }  c \in  C,
\]
so formula \eqref{eq:varphisum} shows that
\begin{equation}\label{eq:varphisum2}
\varphi\big(\tsum_{i=1}^n  \tsum_{j=1}^n \tsum_{\sigma \in G} 
e_{ij} \circ \lambda_{d_\sigma} \circ \lambda_{c_{ij\sigma}}\big) 
\,=\, \min\limits_{1 \leq i, j \leq n; \, \sigma \in G}
\varphi( e_{ij} \circ \lambda_{d_\sigma} \circ \lambda_{c_{ij\sigma}} ).
\end{equation}
Now, $A$ is a right $C$-vector space via $\lambda$, i.e., 
$f \cdot c = f \circ \lambda_c$ for $f \in A$ and $c \in C$, and 
${(e_{ij}\circ\lambda_{d_\sigma})_{1 \leq i, j \leq n; \, \sigma \in G}}$ 
is a $C$-base of $A$. Formula \eqref{eq:varphisum2} shows that the direct 
sum decomposition of $A$ into $1$-dimensional right $C$-vector spaces  
$A = \bigoplus_{i, j,\sigma}(e_{ij} \circ \lambda_{d_\sigma})\cdot C$ is 
a splitting decomposition for the $v$-value function $\varphi$ on $A$,
i.e., 
$$
\gr_\varphi(A) \,\cong_g \,\tbigoplus\limits_{i,j,\sigma} \gr_\varphi
\big((e_{ij} \circ \lambda_{d_\sigma})\cdot C\big).
$$  
Therefore, to show that $\varphi$ is a $v$-norm on $A$, it suffices to 
show that its restriction to each $1$-dimensional $C$-subspace 
$(e_{ij} \circ \lambda_{d_\sigma})\cdot C$ is a $v$-norm.
For this, fix $i,j \in \{1,\ldots, n\}$ and $\sigma \in G$, let 
$h = e_{ij} \circ \lambda_{d_\sigma}$, and let $H = {h \!\cdot\! C}$. For 
$c \in C$, we have
\begin{equation}\label{eq:phiH} 
\varphi(h \!\cdot\! c) \,=\, \varphi(h \circ \lambda_c) \,=\, 
  \min\limits_{ \rho \in G}( \alpha_\rho (c)),
\end{equation} 
where
\[
\alpha_\rho (c) \,=\, 
\alpha(t_{\sigma, \rho}d_\rho^{-1} c d_\rho) + \tau_\rho, \quad 
\text{with } \ 
\tau_\rho\,=\,  \mu_i  - \mu_j  + \gamma_{\sigma\rho}-  \gamma_\rho.
\]
Since $\alpha$ is a $v_1$-value function on $C$, we have for any 
$c \in C$ and $\ell \in L$,
\[
\alpha_\rho (c\ell) \,=\, 
\alpha(t_{\sigma, \rho}d_\rho^{-1} c d_\rho d_\rho^{-1} \ell d_\rho) 
\,=\, \alpha_\rho(c) + v_1(d_\rho^{-1} c d_\rho) \,=\, 
\alpha_\rho(c) + v_\rho(\ell).
\]
Hence, $\alpha_\rho$ is a $v_\rho$-value function on $C$. The function 
$g\colon C \rightarrow C$ given by 
$c \mapsto t_{\sigma, \rho}d_\rho^{-1} c d_\rho$ is an $F$-vector space 
isomorphism with $\alpha_\rho(c) = \alpha(g(c)) + \tau_\rho$ for all 
$c\in C$.
So, for any $\gamma \in \Gamma$, $g$~maps 
$C^{\alpha_\rho}_{ \geq \gamma}$ bijectively to 
$C^{\alpha}_{ \geq (\gamma - \tau_\rho)}$ and 
$C^{\alpha_\rho}_{ > \gamma}$ bijectively to 
$C^{\alpha}_{ > (\gamma - \tau_\rho)}$. Hence, $g$ induces a 
bijection $\gr(g)\colon \gr_{\alpha_\rho}(C) \rightarrow 
\gr_{\alpha}(C)(-\tau_\rho)$, in the grade shift notation of 
\eqref{eq:shift}; clearly $g$ is
 a graded 
$\gr_v(F)$-vector space isomorphism. Since 
$\gr_{v_1}(L) = \gr_v(F) = \gr_{v_\rho}(L)$ by \eqref{eq:grL},
and $\alpha$ is a $v_1$-norm on $C$, we have
\begin{eqnarray*}
 \dim_{\gr_{v_\rho}(L)} \gr_{\alpha_\rho}(C) 
& = &  
\dim_{\gr_{v}(F)} \gr_{\alpha_\rho}(C)  \ = \   
\dim_{\gr_{v}(F)} \big(\gr_{\alpha}(C)(-\tau_\rho)\big)\\
& = &  
\dim_{\gr_{v}(F)} \gr_{\alpha}(C)  \ = \   
\dim_{\gr_{v_1}(L)} \gr_{\alpha}(C) \  = \  \dim_L C.
\end{eqnarray*}
Thus, $\alpha_\rho$ is a $v_\rho$-norm on $C$. Hence, 
$\varphi_\rho \colon H \rightarrow \Gamma \cup \{\infty\}$ given by 
$ \varphi_\rho ({h \!\cdot\! c}) = \alpha_\rho(c)$ is a $v_\rho$-norm on  
the $1$-dimensional $C$-vector space $H$, for every $\rho \in G$.

We work back from the $\varphi_\rho$ to $\varphi|_H$.  Since 
$\varphi|_H = \min_{\rho\in G}(\varphi_\rho)$ by \eqref{eq:phiH},
there is a graded $\gr_v(F)$-vector space monomorphism
$$
\Phi\colon \gr_{\varphi}(H) \, \hookrightarrow \, \tbigoplus_{\rho\in G}
\gr_{\varphi_{\rho}}(H)
$$
given by 
$$
\tilde b^{\,\varphi} \, = \, b + H^\varphi_{>\varphi(b)} \, \mapsto \, 
\big( \ldots, \,b+H^{\varphi_\rho}_{>\varphi(b)}, \, \ldots\big) \qquad 
\text{for all $b \in H\setminus\{0\}$}.
$$
Once we verify that $\Phi$ is surjective, we will have 
\begin{align*}
\dim_{\gr_v(F)} \gr_\varphi(H) \, &=\,  \tsum_{\rho\in G}\dim_{\gr_v(F)}
\gr_{\varphi_\rho}(H) \, = \, \tsum_{\rho\in G}
\dim_{\gr_{v_\rho(L)}}\gr_{\varphi_{\rho}}(H)\\ 
&= \,\tsum_{\rho\in G}
\dim_L H \, = \, |G|\dim_L H \, = \, \dim_F H. \qquad\qquad\qquad\qquad
\end{align*}
Hence, $\varphi|_H$ is a $v$-norm on $H$.

For the surjectivity of $\Phi$, consider  the $w_L$-coarsening of $\varphi_\rho$. We have seen that 
$\varepsilon \circ \alpha = y|_C$ as $w_L$-gauges on $C$. Also the 
equation $d_{\sigma}d_{\rho} = d_{\sigma\rho}t_{\sigma, \rho}$ yields for 
the valuation $y$,
\[
y(t_{\sigma, \rho}) + y(d_{\sigma\rho})   - y(d_{\rho}) 
\,=\, y(d_{\sigma}) \quad \text{for all } \sigma, \rho \in G.
\]
Hence, for $c \in C$,
\begin{eqnarray*}
\varepsilon \circ \varphi_\rho ({h\! \cdot\! c}) 
& = & 
\varepsilon \circ \alpha (t_{\sigma, \rho}d_\rho^{-1} c d_\rho) 
+ \varepsilon (\tau_\rho)  \ = \   
y(t_{\sigma, \rho}d_\rho^{-1} c d_\rho) + \pi_i  - \pi_j  
+ y(d_{\sigma\rho})-  y(d_\rho)\\
& = &  
y(c) +  \pi_i  - \pi_j  + y(d_{\sigma}).
\end{eqnarray*}
Thus, $\varepsilon \circ \varphi_\rho$ is the same for each 
$\rho \in G$. Since $w_L$ is the finest common coarsening of 
$v_\rho$~and~$v_{\rho'}$ for all distinct $\rho, \rho' \in G$,
as noted at the beginning of the proof,
 it follows by an argument just like that 
for surjectivity of $\Psi$ in the proof of Th.~\ref{thm:compatible} that
$\Phi$ is surjective.  Hence $\varphi|_H$ is a $v$-norm on $H$, 
as noted above. 
Since this is true for each subspace 
$H = (e_{ij} \circ \lambda_{d_\sigma})\cdot C$ in the splitting 
decomposition of $A$ for $\varphi$, this $\varphi$ must be a $v$-norm on $A$,
as claimed. 

To see that $\varphi$ is not only a $v$-norm but actually a $v$-gauge, 
observe  that 
$\varphi = \psi|_A$ is surmultiplicative on $A$ since the  
$v_1$-gauge $\psi$ is surmultiplicative on all of $A'$. Moreover, the 
inclusion $A \hookrightarrow A'$ yields a canonical graded monomorphism 
$\iota\colon \gr_\varphi(A) \hookrightarrow  \gr_\psi(A')$. Because 
$\varphi$ is a $v$-norm and $\psi$ is a $v_1$-norm and 
$\gr_v(F) =  \gr_{v_1}(L)$, we have
\[
[\gr_\varphi(A): \gr_v(F)] \,=\, [A:F] \,=\, [A':L] \,=\,
 [\gr_\psi(A'):\gr_{v_1}(L)] \,=\, [\gr_\psi(A'):\gr_v(F)].
\]
Hence, $\iota$ is a graded isomorphism. Therefore, $\gr_\varphi(A)$ is 
graded semisimple, since this is true for $\gr_\psi(A')$, as $\psi$ is 
a gauge. Thus, $\varphi$ is a $v$-gauge on $A$.

Since $\psi$ and $\beta$ are End-gauges, we have, for all $f \in A$,
\[
\varphi(f) =\, \psi(f) \,=\, \min\limits_{m \,\in M \setminus\{0\}}
\big(\eta(f(m)) - \eta(m)\big) \quad \text{and} \quad 
\beta(f) \,=\, \min\limits_{m \,\in M \setminus\{0\}}
\big(\phi(f(m)) - \phi(m)\big).
\]
We observed earlier that $\varepsilon \circ \eta = \theta$. It follows 
that $\varepsilon \circ \varphi = \beta$, i.e., $\beta$ is the 
$w$-coarsening of $\varphi$. Also,
from \eqref{eq:varphisum},
\[
\Gamma_\varphi \,\subseteq\, 
 \tbigcup_{i,j=1}^n \tbigcup_{\sigma, \rho \,\in G} 
\big(\mu_i  - \mu_j  + \gamma_{\sigma\rho}-  
\gamma_\rho + \Gamma_\alpha \big) \,\subseteq\, \Gamma \,=\, \divh{\Gamma_v}.
\]

Since $\gr_\varphi(A) \cong_g \gr_\psi(A')$, we have 
$\omega(\varphi) = \omega(\psi)$. We compute $\omega(\psi)$. 
Partition $\{1,\ldots, n\}$ into a disjoint union 
$\bigsqcup_{\ell = 1}^k S_\ell$ according to the coset of 
$\Gamma_y$ containing $\theta(m_j)$. That is, if $j \in S_\ell$, then 
$S_\ell = \{j' \in \{1,\ldots, n\} \, |
 \, \theta(m_{j'}) - \theta(m_j) \in \Gamma_y\}$. Recall that the 
$\pi_j = \theta(m_j)$ 
have been chosen so that $\pi_j = \pi_{j'}$ if and only if 
 $j$ and $j'$ lie in the 
same $S_\ell$. Since $\beta = \End (\theta)$, by the comments with
\eqref{eq:end0} the number 
of sets $S_\ell$ equals $\omega(\beta)$.

We use the $S_\ell$ to decompose $\gr_\eta(M)$. Since 
$(m_jd_\rho)_{j=1, \,\rho \in G}^n$ is the $C$-base of $M$ used in 
building~$\eta$ and 
$\widetilde{m_jd_\rho} = \widetilde{m_j}\hsp\widetilde{d_\rho}$, by 
Lemma~\ref{lemma1existence} 
$(\widetilde{m_j}\hsp\widetilde{d_\rho})_{j=1, \,\rho \in G}^n$  is a 
$\gr_\alpha(C)$-base of $\gr_\eta(M)$. For $\ell = 1,  \ldots , k$ and~ 
$\rho \in G$, let
\[
\grN_{\ell\rho} \,=\, \tbigoplus_{j \in S_\ell} 
\widetilde{m_j}\hsp\widetilde{d_\rho}\gr_\alpha(C).
\]
So, $\gr_\eta(M) = \tbigoplus_{\ell = 1}^k \tbigoplus_{\rho \in G} 
\grN_{\ell\rho},$ and each $\grN_{\ell\rho}$ is a graded right
$\gr_\alpha(C)$-submodule of $\gr_\eta(M)$. We show that their grade 
sets do not overlap. If $\gamma \in \Gamma_{\grN_{\ell\rho}}$, then 
$\gamma = \deg (\widetilde{m_j}\hsp\widetilde{d_\rho}\hsp \widetilde{c})$ for 
some $j \in S_\ell$ and  $c \in C^\times$ with~ 
$\widetilde{m_j}\hsp\widetilde{d_\rho}\hsp \widetilde{c} \neq 0$. So, 
$\gamma = \mu_j + \gamma_\rho + \alpha(c) \in \Gamma$. Then, as 
$\varepsilon(\gamma_\rho) = \delta_\rho = y(d_\rho)$ and 
$\varepsilon \circ \alpha = y|_C$,
\[
\varepsilon(\gamma) \,=\, \pi_j + \delta_\rho+ y(c) \,=\,
 \pi_j + y(d_\rho c) \,\in\, \pi_j + \Gamma_y.
\]
Now likewise let $\gamma' \in \Gamma_{\grN_{\ell'\rho'}}$, say 
$\gamma' = 
\deg (\widetilde{m_{j'}}\hsp\widetilde{d_{\rho'}}\hsp \widetilde{c'}) 
= \mu_{j'} + \gamma_{\rho'} + \alpha(c')$. If $\gamma = \gamma'$, then 
$\varepsilon(\gamma) = \varepsilon(\gamma')$, so 
$\pi_j + \Gamma_y = \pi_{j'} + \Gamma_y$. Hence, $S_\ell = S_{\ell'}$, 
so $\ell = \ell'$ and $\pi_j = \pi_{j'}$. Then the equality 
$\varepsilon(\gamma) = \varepsilon(\gamma')$  yields 
$y(d_\rho) + y(c) = y(d_{\rho'}) + y(c').$ Since the $y(d_\rho)$ are 
distinct modulo $\Gamma_{C,y}$, it follows that $\rho = \rho'$. Thus, 
$\grN_{\ell\rho} = \grN_{\ell'\rho'}$ whenever their grade sets intersect. 
Hence, $\Gamma_{\gr_\eta(M)} =
 \bigsqcup_{\ell =1}^k \bigsqcup_{\rho \in G} \Gamma_{\grN_{\ell\rho}}$, 
a disjoint union, so each homogeneous element of $\gr_\eta(M)$ lies in 
some $\grN_{\ell\rho}$. 

Let $\grE = \grEnd_{\gr_\alpha(C)}(\gr_\eta(M)) \cong_g \gr_\psi(A')$ by 
Lemma~\ref{lemma1existence}. If $f \in \grE_0$, then $f$ is a 
degree-preserving map, so $f$ must map each $\grN_{\ell\rho}$ to itself, 
by the disjointness of the grade sets $\Gamma_{\grN_{\ell\rho}}$. Thus,
\begin{equation}\label{eq:E0}
\grE_0 \,\cong\, \tprod_{\ell =1}^k \tprod_{\rho \in G}  
\big(\grEnd_{\gr_\alpha(C)}(\grN_{\ell\rho})\big)_0.
\end{equation}
Since the $\pi_j$ are the same for all $j$ in $S_\ell$, the $\mu_j$ 
are likewise the same by hypothesis, hence the base elements 
$\widetilde{m_j}\,\widetilde{d_\rho}$ of $\grN_{\ell \rho}$ all have the 
same degree $\mu_j+\gamma_\rho$. Therefore, as graded $\gr_\alpha(C)$-modules,
\[
\grN_{\ell \rho} \,\cong_g\,
 \gr_\alpha(C)^{|S_\ell||G|}(\mu_j + \gamma_\rho) \qquad 
\text{for any } j \in S_\ell,
\]
in the grade shift notation of \eqref{eq:shift}.
Hence,
\[
\grEnd_{\gr_\alpha(C)} \big( \grN_{\ell \rho} \big) \,\cong_g\, 
\grEnd_{\gr_\alpha(C)} \big( \gr_\alpha(C)^{|S_\ell||G|}  \big) 
\,\cong_g\, \mat_{|S_\ell||G|}(\gr_\alpha(C)).
\]
So, in degree $0$, 
$\big(\grEnd_{\gr_\alpha(C)} \big( \grN_{\ell \rho} \big)\big)_0 
\cong \mat_{|S_\ell||G|}(\gr_\alpha(C)_0)$, and the number of its simple  
components coincides with the number of simple components of 
$\gr_\alpha(C)$, which is $\omega(\alpha)$. So, from \eqref{eq:E0},
\[
\omega(\varphi) \,=\, \omega(\psi) \,=\, 
\text{ number of simple components of } \grE_0 \,=\,
 k \,|G|\, \omega(\alpha) \,=\, \omega(\beta) \,r\, \omega(\alpha). 
\]

To complete the proof, we show that $\alpha$ can be chosen with 
$\omega(\alpha) = 1$. Let $T$ be a Dubrovin valuation ring of 
$D$ with $Z(T) = V$.  Then $T \subseteq Y$ since the valuation ring $Y$
of $y$ on $D$ is the unique Dubrovin valuation ring of $D$ with 
center $W$.  Let $\tilde T = T/J(Y)$, which is a Dubrovin
valuation ring of $Y/J(Y) = \ov D$.  The valuation of the 
valuation ring $Z(\tilde T)$ is an extension to $Z(\ov D)$
of $u = v/y$. So this valuation is one of the $u_i$; after
renumbering if necessary, we may assume that it is $u_1$.
Since jump prime ideals
of $u_1$ for $\ov D$ pull back to 
jump prime ideals of $v$ for $D$ strictly containing~$J(W)$ by~
\cite[Cor.~E]{W}, and since by hypothesis $J(V)$ is the only such 
jump prime ideal for $D$, we have $j(u_1,\ov D) = 1$.  Hence,
by Prop.~\ref{prop:maxint}(iv) $\tilde T$ is integral over its 
center, which is the valuation ring $U_1$ of~$u_1$.  
Now  let $B$ be a Dubrovin valuation ring of $C$ with center $Z(B) = V_1$, 
the valuation ring of~ $v_1$ on $L = Z(C)$. Then $B \subseteq Y \cap C$, 
since the 
valuation ring $Y\cap C$ is the unique Dubrovin valuation ring of $C$ with 
center $W_L$, the valuation ring of $w_L$. 
Note that   
$\ov C^{\, y}= \overline{D}$, since $D$ is totally ramified over $C$ re $y$.
Let $\widetilde{B} = B/J(Y\cap C)$, 
which is a Dubrovin valuation ring of $\ov C^{\, y}$, so of
$\ov D$. The valuation
of the valuation ring $Z(\tilde B)$ restricts on $\ov L^{\,y} = K$
to $v_1/w_L = u_1|_K$.  Hence, $Z(\tilde B) = U_1 = Z(\tilde T)$,
as $u_1$ is the unique extension of $u_1|_K$ to $Z(\ov D)$.   
Because $\tilde B$ and $\tilde T$ are Dubrovin valuation rings
of $\ov D$ with the same center, we have $\tilde B \cong \tilde T$.  
Therefore $\tilde B$ is integral over $U_1$, since this is true 
for $\tilde T$. But also $U_1$ is integral over 
$U_1\cap K$, as $u_1$ is the unique extension of 
$u_1|_K$ to~$Z(\ov D)$; hence, $\tilde B$ is integral over 
$U_1 \cap K$.      
Moreover, by \cite[p.~390]{W3} 
 the valuation  
ring~$Y\cap C$ of the division ring $C$ is integral over its center 
$Y\cap L = W_L$,
 It follows by  
\cite[Prop.~12.2, p.~70]{MMU} that $B$ is integral over $V_1$. 
Let $\alpha$ 
be the Morandi value function on $C$  with $R_\alpha = B$. 
To see that $\alpha$ is 
a $v_1$-gauge we must check
 that $v_1$~on~$L$ is defectless in~$C$. 
For this, note that by~ 
Th.~\ref{tensorhens}, there is an $F_{h,v}$-isomorphism 
\[
L \otimes_F F_{h,v} \cong L_{h,v_1} \times \ldots \times L_{h, v_r}.
\]
Since $\DIM LF = r$, each factor $L_{h,v_i}$
must be 1-dimensional over $F_{h,v}$, i.e.,
isomorphic to $F_{h,v}.$ 
Note that as $C = C_D(L)$, the algebras $D\otimes_F L$ and 
$C$ are Brauer equivalent.  Hence, $C \otimes_L L_{h,v_1}$
is Brauer equivalent to $D \otimes_F L_{h,v_1} \cong
D\otimes_F F_{h,v}$.
Since $v$ on $F$ is defectless in $D$, $v_1$ on $L$ is defectless in 
$C$. It follows by Remark~ \ref{rem:omega1} that $\alpha$ is a $v_1$-gauge 
on $C$ with 
$\omega(\alpha) =1$.
\end{proof}

There is a easier version of Prop.~\ref{prop1existence} for jump rank 1 
that we will need later:

\begin{proposition}\label{prop:jump1}
 Let $A$ be a central simple $F$-algebra. Suppose 
the valuation $v$ on $F$ is defectless in $A$. Let  valuation $w$ 
on $F$ be a nontrivial coarsening of $v$, and assume that $w$ is Henselian. 
Suppose $j(v,A) = 1$. Let $\beta$ be a $w$-gauge on $A$ with 
$\Gamma_\beta \subseteq \divh{\Gamma_w}$. Then, there is a $v$-gauge 
$\varphi$ on $A$ with $w$-coarsening $\beta$ such that 
$\omega(\varphi) = \omega(\beta)$ and 
$\Gamma_\varphi \subseteq \divh{\Gamma_v}$. 
\end{proposition}

\begin{proof}
View 
$A =\End_D(M)$ where $D$ is a division ring and $M$ is a finite-dimensional 
right $D$-vector space. Since $w$ is Henselian, it has a unique extension 
to a valuation $y$ on $D$. Moreover, by \cite[Th.~3.1]{TW1}, $\beta$ is an 
End-gauge, as in Ex.~\ref{ex;endgauge}, say 
$\beta = \End(\theta)$ for some $y$-norm $\theta$ on $M$. Let 
$(m_1, \ldots, m_n)$ be a splitting base of $M$ for $\theta$, and let 
$\pi_j = \theta(m_j)$ for $j =1, \ldots, n$. Let 
$\Lambda = \divh{\Gamma_w}$ and $\Gamma = \divh{\Gamma_v}$, and 
$\varepsilon\colon \Gamma \rightarrow \Lambda$ the epimorphism induced by 
the canonical map $ \Gamma_v \rightarrow \Gamma_w$.  
From~\eqref{eq:Gammaalpha}, $\Gamma_\beta  =  
\Gamma_{\gr_\beta(A)} = \bigcup_{i, j = 1}^n \pi_i - \pi _j  + 
\Gamma_w$.  By changing $\theta$ by replacing each $\pi_j$ by 
$\pi_j - \pi_1$ (which does not change $\End(\theta)$) we may 
assume that each $\pi_j \in \Gamma_\beta \subseteq 
 \divh{\Gamma_w}$. Also, by adjusting the~$m_j$ by multiples in 
$D^\times$, we may assume that $\pi_i = \pi_j$ whenever 
$\pi_i + \Gamma_y = \pi_j + \Gamma_y$. So, the number of distinct 
$\pi_i$ equals the number of cosets of $\Gamma_y$ in $\Gamma_\beta$. 
This number equals $\omega(\beta)$ by~\eqref{eq:end0}.

Because $j(v,A) = 1$, we have $\ind(A \otimes_F F_{h,v}) = 
\ind(A \otimes_F F_{h,w})$ where $F_{h,v}$ (resp.~$F_{h,w}$)
 is a Henselization of $F$ with respect to $v$ (resp.~$w$).
Since $F_{h,w} = F$ as $w$ is assumed Henselian, it follows that 
 $\ind(A \otimes_F F_{h,v}) = \ind(A)$. Hence, by  
Morandi's theorem \cite[Th.~3]{M1} 
the valuation $v$ on~$F$ extends to a valuation $z$ on $D$. Pick any 
$\mu_1, \ldots, \mu_n \in \Gamma$ with $\varepsilon (\mu_j) = \pi_j$ 
for all $j$ and $\mu_j = \mu_i$ whenever $\pi_j = \pi_i$. Let $\eta$ be 
the $z$-norm on $M$ given by
\[
\eta \big(\tsum_{j=1}^n m_j d_j \big) \,=\, 
\min\limits_{1  \leq j \leq n} \big( \mu_j + z(d_j)\big) 
\qquad \text{for all } d_1, \ldots, d_n \in D.
\]
Let $\varphi = \End(\eta)$, which is a $v$-gauge on $A$ since $v$ is
defectless in $A$ (see Ex.~\ref{ex;endgauge}). Then $\varepsilon \circ \varphi = \beta$ since 
$\varepsilon \circ \eta = \theta$. Also, if 
$\pi_i + \Gamma_y = \pi_j + \Gamma_y$, then $\pi_i = \pi_j$, so 
$\mu_i  =\mu_j$ by the choice of the $\mu_i$, so 
$\mu_i + \Gamma_z = \mu_j + \Gamma_z$. Thus, by \eqref{eq:end0},
\[
\omega(\varphi) \,=\, 
\big| \{ \text{cosets of } \Gamma_z \text{ in } \Gamma_\varphi\}\big|
\, =\, \big| \{ \text{cosets of } \Gamma_y \text{ in } \Gamma_\theta\}\big| 
\,=\, \omega(\beta).
\]
Because each $\mu_j \in \Gamma = \divh{\Gamma_v}$ and  
$\Gamma_z \subseteq \divh{\Gamma_v}$, we have by 
\eqref{eq:Gammaalpha}
$$
\qquad \qquad\qquad \qquad
\Gamma_\varphi\, =\, \tbigcup_{i,j = 1}^n (\mu_i - \mu_j) + 
\Gamma_z \,\subseteq \,\divh{\Gamma_v}. \qquad\qquad\qquad\qquad
\qquad \qquad\qquad\qquad
\qedhere
$$
\end{proof} 

\begin{proof}[Proof of Th.~\ref{existenceminimalgauges} 
$($Central simple case$)$]
Suppose $A$ is central simple. We argue by induction on the jump rank 
$j(v,A)$. 
If $j(v,A) = 0$, then  $v$ is the trivial valuation on $F$, and 
the trivial gauge  on~$A$ is a minimal $v$-gauge. Ê
If $j(v,A) = 1$, then 
$\xi_{V,[A]}=1$ by Prop.~\ref{prop:maxint}(iv); so for any Dubrovin valuation 
ring~$B$ of $A$ with $Z(B) = V,$ $B$ is integral over $V$. Let 
$\alpha$ be the associated Morandi value function of~$B$. 
Then $\Gamma_\alpha = \Gamma_B$ by \eqref{eq:Morvalgp}, and 
$\Gamma_B \subseteq \divh{\Gamma_v}$ by \eqref{eq:gammaB}; hence
$\Gamma_\alpha \subseteq \divh{\Gamma_v}$.
Since $v$ is defectless in $A$, by Remark \ref{rem:omega1}
$\alpha$ is a $v$-gauge with  
 $\omega(\alpha) =1$, so $\alpha$ is a minimal $v$-gauge.

Now suppose $j(v,A) = n >1$. Let $P$ be the $(n-1)$-st jump prime ideal 
of $v$ for $A$, and let ${W = V_P}$, with associated valuation $w$. Then 
$w$ is defectless in $A$ since $v$ is defectless in~$A$, by~ 
Prop.~\ref{prop:coarsdefless}. Since $j(w,A) = n-1$, by induction 
there is a minimal $w$-gauge $\beta$ on $A$ 
with~${\Gamma_\beta \subseteq \divh{\Gamma_{w}}}$.
Let~$(F', w_h)$~be the Henselization of $(F,w)$, let $v'$ be the 
valuation of $F'$ refining $w_h$ and restricting to $v$ on $F$. Let 
$A' = A \otimes_{F}F'$, which is a central simple $F'$-algebra, and let 
$\beta' = \beta \otimes w_h$, which is a $w_h$-gauge on~$A'$ with 
$\gr_{\beta'}(A') \cong_g \gr_{\beta}(A)$ by \cite[Cor.~1.26]{TW1}. Note  
however that $\beta'$ 
need not be a minimal gauge even though $\beta$ is minimal.

We claim that $j(v', A') = 2$. To see this, let valuation $y$ on $F'$ be 
any coarsening of $v'$, and let $z = y|_F$, which is a coarsening of $v$. 
If $y =w_h$ or $y$ is coarser than $w_h$, then $y$ is Henselian, as  
$w_h$~is Henselian (see \cite[Cor.~4.1.4, p.~ 90]{EP}). Then the 
Henselization $F'_{h,y} = F'$, so of course  
$\ind(A'\otimes_{F'} F'_{h,y}) = \ind(A')$. Suppose instead that $y$ is 
properly finer than $w_h$, so $z$ is a refinement of $w$. We show that 
then the Henselizations $(F'_{h,y}, y_h)$ and $(F_{h,z}, z_h)$ are 
isomorphic. For this, note first that since the $w\text{-}$coarsening 
$w_1$ of $z_h$ in $F_{h,z}$ is Henselian and restricts to $w$ in $F$, 
there is an $F$-homomorphism $\eta_1\colon F' \rightarrow F_{h, z}$ with 
$w_h = w_1 \circ \eta_1$. The Henselian valuation $z_h$ on $F_{h,z}$ pulls 
back to a valuation $z'$ on $F'$ that refines $w_h$ with 
$z'|_F = z = y|_F$. Hence, $z'= y$. Because $z_h$ is Henselian and pulls 
back to $y$, there is an $F$-monomorphism 
$\eta_2\colon F'_{h,y} \rightarrow F_{h, z}$ with $y_h = z_h \circ \eta_2$. 
But also since $y_h$ is Henselian with $y_h|_F = y|_F =z$, there is an 
$F$-monomorphism $\eta_3\colon F_{h, z} \rightarrow F'_{h,y}$ with 
$z_h = y_h \circ \eta_3$. Thus, $\eta_2 \circ \eta_3$ is an 
$F$-homomorphism $F_{h, z} \rightarrow F_{h, z}$ with 
$z_h = z_h \circ \eta_2 \circ \eta_3$. By the uniqueness in the universal 
property for the Henselization (recalled in the Appendix below),  
we must have 
$\eta_2 \circ \eta_3 = \id_{F_{h,z}}$, so $\eta_3$~is an isomorphism 
$(F_{h,z}, z_h) \cong (F'_{h,y}, y_h)$. In particular, 
$(F_{h,v}, v_h) \cong (F'_{h,v'}, v'_h)$. So, as there are no jump 
prime ideals of $v$ for $A$  between $P = J(W)$ and~$J(V)$,
\[
\ind(A'\otimes_{F'} F'_{h,y}) \,=\, \ind(A\otimes_F F_{h,z}) \,=\, 
\ind(A\otimes_F F_{h,v}) \,=\,  \ind(A'\otimes_{F'} F'_{h,v'});
\]
 this value is strictly smaller than $\ind(A') = 
\ind(A\otimes_{F} F')$ since $P$ is  a jump prime ideal. Thus,  
$j(v', A') = 2$, as claimed. The calculation also shows that $w_h$ is 
the first jump valuation for $v'$ in~$A'$. Note also that $v'$ is 
defectless in $A'$, since this depends on the defectlessness of the 
associated division algebra of $A'\otimes_{F'} F'_{h,v'}$ re $v_h'$; 
but  $A'\otimes_{F'} F'_{h,v'} \cong A\otimes_F F_{h,v}$ and $v$ is defectless
in $A$. Thus, the hypotheses  
of Prop.~\ref{prop1existence} are satisfied for the field $F'$ with 
valuations $v'$ and $w_h$ and central simple $F'$-algebra $A'$ with 
$w_h$-gauge $\beta'$.

By Prop.~\ref{prop1existence}, $A'$ has a $v'$-gauge $\varphi$ whose 
$w_h$-coarsening is $\beta'$ with $\omega(\varphi) = r \omega(\beta')$ 
and $\Gamma_{\varphi} \subseteq \divh{\Gamma_{v'}}$. Let $D'$ be the 
associated division algebra of $A'$, let $w_{D'}$ be the valuation on 
$D'$ extending $w_h$ on $F'$, and let 
$\overline{D'} = \overline{D'}^{\,w_{D'}}$. Let $u'$ be the residue 
valuation $v'/w_h$ on $\overline{F'}^{\,w_{h}}$ determined by $v'$. The 
integer~$r$ in the formula for $\omega(\varphi)$ is the number of 
extensions of $u'$ to $Z(\overline{D'})$. Hence, by \eqref{eq:lVWchar}, 
$r = \ell_{V,W}$. Let $\alpha = \varphi|_A$. By Lemma~\ref{lemma2existence},
 $\alpha$ is a $v$-gauge on $A$ with $w$-coarsening $\beta$, and 
$\gr_\alpha(A) \cong_g \gr_\varphi(A')$. Hence,
\begin{equation}\label{eq:omegavarphi}
\omega(\alpha) \,=\, \omega(\varphi) \,=\, r \omega(\beta) \,=\,
 \ell_{V,W}\, \omega(\beta) \,=\,\ell_{V,W}\,\xi_{W,[A]}.
\end{equation}
The last equality in \eqref{eq:omegavarphi} holds since $\beta$ is a 
minimal gauge. Since $\xi_{V,[A]} \geq \xi_{W,[A]}\ell_{V,W}$ by 
\eqref{teoremaE}, it follows from \eqref{eq:omegavarphi} that 
$\omega(\alpha) \leq \xi_{V,[A]}$.  But we always have 
$\omega(\alpha) \geq \xi_{V,[A]}$ by Th.~\ref{minimalgauge1}; so 
$\omega(\alpha)= \xi_{V,[A]}$, showing that $\alpha$ is a minimal gauge. 
Furthermore, $\Gamma_\alpha = \Gamma_\varphi \subseteq \divh{\Gamma_{v'}} 
= \divh{\Gamma_{v}}$.  This completes the proof of the central simple case
of Th.~\ref{existenceminimalgauges}.  The rest of the proof of the theorem
will be given after Lemma~\ref{lemma4existence} below.
\end{proof}

\begin{proposition}\label{prop2existence}
Let $A$ be a central simple $F$-algebra with $v$ defectless in $A$. 
Let $w$ be a valuation on $F$ that is coarser than $v$, and let 
$\beta$ be a minimal $w$-gauge on $A$ with 
$\Gamma_\beta \subseteq \divh{\Gamma_w}$. Then there is a minimal 
$v$-gauge $\alpha$ with $w$-coarsening $\beta$ and 
$\Gamma_\alpha \subseteq \divh{\Gamma_v}.$
\end{proposition}
\begin{proof}
Let $V$ (resp.~$W$) be the valuation ring of $v$ (resp.~$w$). 
We argue by induction on the  
jump rank $j(v,A)$. Clearly, $j(w,A) \leq j(v,A)$. If $j(w,A) = 0$, then 
$w$ is the trivial valuation on $F$ and $\beta$~is the trivial $w$-gauge 
on $A$.  By the case of Th.~\ref{existenceminimalgauges} already proved,
there is a minimal $v$-gauge $\alpha$ on~$A$; then the trivial gauge 
$\beta$ is a coarsening of $\alpha$. Thus, we may assume that
$1 \le j(w,A) \le j(v,A)$. 

Suppose $j(v,A) = 1$. 
Then $j(w,A) =1$, 
 so $\xi_{W,[A]} = 1$ by Prop.~\ref{prop:maxint}(iv). Hence, 
$\omega(\beta) =1$ because $\beta$ is assumed to be a minimal gauge. 
Therefore, by Remark~\ref{rem:omega1} $R_\beta$ is a Dubrovin 
valuation ring integral over its center, 
which is  $W$, and $\beta$ is the Morandi value function of $R_\beta$. 
Let $B$ be any Dubrovin valuation ring of $A$ with $Z(B) = V$ and 
$B \subseteq R_\beta$. Such a $B$ is obtainable as the inverse 
image in $R_\beta$ of a Dubrovin valuation ring of 
$R_\beta / J(R_\beta)$ whose center is a valuation ring  $U$
of $Z(R_\beta / J(R_\beta))$ satisfying $U \cap (W/J(W)) = V/J(W)$.
Since $j(v,A) = 1$, by Prop.~\ref{prop:maxint}(iv) $B$ is integral over~$V$. 
Let $\alpha$ be the Morandi value function of $B$. Then 
$\Gamma_\alpha = \Gamma_B \subseteq \divh{\Gamma_v}$
(see \eqref{eq:gammaB}).  Since 
$v$ is assumed defectless in $A$, by Remark~\ref{rem:omega1} 
$\alpha$ is a $v$-gauge with
$\omega(\alpha) =1$; so $\alpha$ is  minimal gauge. 
Because $R_\alpha \subseteq R_\beta$ and these Dubrovin valuation rings 
determine their associated gauges, $\beta$~is a coarsening of $\alpha$.

Now suppose $j(v,A) = n>1$. Let $y$ be the $(n-1)$-st jump valuation of 
$v$ for $A$. Suppose first that $w$ is coarser than $y$ or $w =y$. Then, 
as $j(y,A) = n-1$ and $y$ is defectless in $A$ since $v$ is defectless 
in $A$, by induction there is a minimal $y$-gauge $\eta$ with 
$\Gamma_\eta \subseteq \divh{\Gamma_y}$ such that $\beta$ is the 
$w$-coarsening of $\eta$. The proof of 
Th.~\ref{existenceminimalgauges} (central simple case) shows that there 
is a minimal $v$-gauge~  $\alpha$ with 
$\Gamma_\alpha \subseteq \divh{\Gamma_v}$ and $y$-coarsening $\eta$. 
The $w$-coarsening of $\alpha$ is then the $w$-coarsening of~$\eta$, 
which is $\beta$.

Suppose instead that $w$ is properly finer than $y$. Let $(F',w_h)$ be 
the Henselization of $(F,w)$, and let $v'$ be the valuation on $F'$ 
refining $w_h$ and restricting to $v$ on $F$. Let 
$\beta' = \beta \otimes w_h$, which is a $w_h$-gauge on 
$A' = A \otimes_F F'$, but not necessarily minimal, with 
$\Gamma_{\beta'} = \Gamma_{\beta} \subseteq \divh{\Gamma_w} 
= \divh{\Gamma_{w_h}}$. Moreover, we claim that $j(v', A') = 1$. To see 
this, suppose $z$ is any nontrivial valuation of $F'$ with $z$ coarser 
than $v'$ or $z = v'$. If $z$ is coarser than $w_h$ or $z = w_h$, then 
$z$ is Henselian, so the Henselization $F'_{h,z} = F'$ and 
$\ind (A' \otimes_{F'} F'_{h,z}) = \ind (A')$. If $z$ is finer than 
$w_h$, then as in the proof of 
Th.~\ref{existenceminimalgauges} (central simple case) above 
$F'_{h,z} = F_{h, z|_F}$, the Henselization of $F$ re $z|_F$. Since 
there are no jump valuations between $w$ and $v$, so none between $z|_F$ 
and $v$,
\[
\ind (A' \otimes_{F'} F'_{h,z}) \,=\, \ind (A \otimes_{F} F_{h, z|_F}) 
\,=\, \ind (A \otimes_{F} F_{h,v}) \,=\, \ind (A \otimes_{F} F') 
\,=\, \ind (A').
\]
Since the indices are the  same for all $z$, $j(v', A') = 1$, as claimed.

By Prop.~ \ref{prop:jump1}, applied to $F', v', w', \beta'$, there is 
a $v'$-gauge $\alpha'$ of $A'$ with $w_h$-coarsening $\beta'$, such that 
$\omega(\alpha') = \omega(\beta')$, and 
$\Gamma_{\alpha'} \subseteq \divh{\Gamma_{v'}} = \divh{\Gamma_v}$. 
Let $\alpha = \alpha' |_A$. By Lemma~\ref{lemma2existence}, $\alpha$ is 
a $v$-gauge on $A$ with $\gr_\alpha (A) \cong_g \gr_{\alpha'}(A')$, 
so
\[
\omega(\alpha) \,=\, \omega(\alpha') \,=\, \omega(\beta') 
\,=\, \omega(\beta).
\]
Since $\beta$ is a minimal gauge and $\xi_{V,[A]} \geq \xi_{W,[A]}$, 
$\alpha$ must also be a minimal gauge. Also, 
${\Gamma_{\alpha} = \Gamma_{\alpha'} \subseteq \divh{\Gamma_v}}$. 
Finally, since $\beta'$ is the $w_h$-coarsening of $\alpha'$, the 
$w$-coarsening of $\alpha = \alpha' |_A$ is $\beta' |_A = \beta$.
\end{proof}

Let $F \subseteq L$ be fields  with $\DIM LF \leq \infty$, and let $v$ be a
nontrivial 
valuation of $F$ with valuation ring $V$. For each nonzero prime ideal 
$P$ of $V$, let $s(P)$ be the number of valuation rings of~$L$ extending 
the valuation ring $V_P$ of $F$. Clearly, $1 \leq s(P) \leq  \DIM LF$. Also 
for prime ideals ${P \subseteq P'}$ we have ${s(P) \leq s(P')}$. Moreover, 
if $P = \bigcup_{j \in J} P_j$ for prime ideals $P$ and $P_j$, $j \in J,$ 
then ${s(P) = \max_{j \in J} s(P_j).}$ We call $P$ a \emph{splitting prime 
ideal}  of $V$ in $L$ if $s(P) < s(P')$ for all prime ideals 
$P'\supsetneqq P$. Define the \emph{splitting rank} of $v$ in $L$ to be
\[
\srk (v,L)\, =\, \text{ the number of splitting prime ideals of $V$ in $L$.}
\]
Note that the maximal ideal $J(V)$ is always a splitting prime ideal of 
$V$ in $L$, so $\srk (v,L) \geq 1$. If $P$ is a splitting prime ideal, we 
call the valuation of $V_P$ a \emph{splitting valuation} of $V$ in $L$.
If $v$ is the trivial valuation on $F$, define $\srk (v,L) = 0$.

\begin{lemma}\label{lemma4existence}
With the notation above, let $v_1$ and $v_2$ be two different extensions 
of $v$ to $L$. Then either $v_1$ and $v_2$ are independent or the finest 
common coarsening $v_{12}$ of $v_1$ and $v_2$ restricts to a splitting 
valuation for $v$ in $L$.
\end{lemma}
\begin{proof}
We argue by induction on $n = \srk (v,L)$. If $n =1$, let $w$ be the 
trivial valuation on $F$. If $n > 1$, let $w$ be the $(n-1)$-st splitting 
valuation of~$F$ for $v$ in $L$. Let $y$ be a valuation of~$F$ coarser 
than $v$ and strictly finer than $w$. Because there are no splitting 
valuations between $y$ and $v$, $y$ must have the same number of 
extensions to $L$ as $v$. Hence, $v_1$ and $v_2$ have distinct 
$y$-coarsenings. If~$n=1$, this shows that $v_{12}$ must be the trivial 
valuation, i.e., $v_1$ and $v_2$ are independent valuations. If instead
$n >1$, this shows that $v_{12} |_F$ either equals or is coarser than $w$. 
Since $w$ is a splitting valuation, it suffices to consider the case when 
$v_{12} |_F$ is strictly coarser than $w$. Then the $w$-coarsenings 
$w_1$ of $v_1$ and $w_2$ of $v_2$ are distinct. Hence, the finest common 
coarsening $w_{12}$ of $w_1$ and $w_2$ is coarser than or equals to 
$v_{12}$. But $v_{12}$ is coarser than each of $w_1$ and $w_2$, so coarser 
than or equal to $w_{12}$. Hence, $v_{12} = w_{12}$. Since 
$\srk(w,L) = n-1$, the conclusion of the lemma holds for $w_1$ and $w_2$ 
by induction; hence it also holds for $v_1$ and $v_2$.
\end{proof}

\begin{proof}[Proof of Th.~\ref{existenceminimalgauges} $($completed$)$]
It suffices to prove the theorem for the simple components of a
semisimple $F$-algebra. So, assume $A$ is simple. Let $L =Z(A)$, so 
$\DIM LF \leq \infty$. Let $v_1, \ldots, v_r$~be the valuations on $L$  
extending $v$ on $F$ and let $V_i$ be the valuation ring of $v_i$. The  
argument is by induction on $n = \srk(v,L)$. 
If $n = 0$, then $v$ is the trivial valuation, and the trivial gauge
on $A$ is a minimal $v$-gauge.
Assume now that $n=1$. 
Lemma~\ref{lemma4existence} then shows that the $v_i$ are pairwise 
independent. By the central simple case of 
Th.~\ref{existenceminimalgauges} proved above, for each $i$ there is a 
minimal $v_i$-gauge $\alpha_i$ of $A$ with 
$\Gamma_{\alpha_i} \subseteq \divh{\Gamma_{v_i}} = \divh{\Gamma_v}$. 
(Note that each $v_i$ is defectless in $L$, by 
Prop.~\ref{prop:deflessredtocent}.) Let 
$\alpha = \min(\alpha_1, \ldots, \alpha_r)$. Because the $v_i$ are 
pairwise independent, Cor.~\ref{cor:viindep} shows that $\alpha$ is 
a $v$-gauge on $A$. Since each $\alpha_i$ is a minimal $v_i$-gauge, 
$\alpha$ is a minimal $v$-gauge (see the comments in~ 
Def.~\ref{def:minimalss}).
 Also, 
$\Gamma_{\alpha} \subseteq  \bigcup_{i=1}^r \Gamma_{\alpha_i} 
\subseteq  \divh{\Gamma_v}$.

Now assume  that $n >1$. Let 
$P_1 \subsetneqq P_2 \subsetneqq \ldots  \subsetneqq P_{n-1}
\subsetneqq P_n$ be the splitting prime ideals of $V$ in~$L$, and let 
$W = V_{P_{n-1}}$ with associated valuation $w$. Because $v$ is defectless 
in $A$, $w$ is defectless in $A$, by Prop.~\ref{prop:coarsdefless}. Since 
$\srk(w,L) = n-1$, by induction there is a minimal $w$-gauge $\beta$ on 
$A$ with $\Gamma_\beta \subseteq  \divh{\Gamma_w}$. Let 
$w_1, \ldots , w_\ell$ be the valuations of $L$ extending $w$, and for 
each $j$ let $\beta_j$ be the $w_j$-component of $\beta$. The construction 
of the $\beta_j$ in Th.~\ref{thm:gaugeismin} shows that 
$\Gamma_{\beta_j} \subseteq \Gamma_\beta \subseteq  \divh{\Gamma_w}$ for 
each $j$. By Th.~\ref{thm:compatible}, $\beta_j$ and $\beta_k$ have the 
same $w_{jk}$-coarsening for all $j,k \in \{1,\ldots, \ell \}.$ Since 
$\beta$ is a minimal $w$-gauge and 
$\beta = \min (\beta_1, \ldots, \beta_\ell )$, each $\beta_j$ is a 
minimal $w_j$-gauge.

For each $i \in  \{1,\ldots, r \}$ let $j(i) \in  \{1,\ldots, \ell \}$ 
be the index such that $w_{j(i)}$ is the $w$-coarsening of~$v_i$. For 
each $i$, Prop.~\ref{prop2existence} applied to the valuations $v_i$ 
and $w_{j(i)}$ on $L$ shows that there is a minimal $v_i$-gauge 
$\alpha_i$ with $w_{j(i)}$-coarsening $\beta_{j(i)}$ and 
$\Gamma_{\alpha_i} \subseteq  \divh{\Gamma_{v_i}} =  \divh{\Gamma_v}$. 
Let $\alpha = \min(\alpha_1, \ldots, \alpha_r)$. To see that $\alpha$ is 
a $v$-gauge, we must check that the $\alpha_i$ satisfy the compatibility 
condition of~Th.~\ref{thm:compatible}. For this, take any distinct 
$i,k \in \{1,\ldots, r\}$. If $j(i) = j(k),$ then $w_{j(i)} = w_{j(k)},$ 
which is coarser than both $v_i$ and $v_k$, so coarser than (or equal to) 
$v_{ik}$. Since $v_{ik}|_F$ is a splitting valuation of~$v$ in $L$, the 
choice of $w$ implies that $v_{ik} =  w_{j(i)} = w_{j(k)}.$ So, the 
$v_{ik}$-coarsening of $\alpha_i$ is $\beta_{j(i)},$ which is the same 
as the $v_{ik}$-coarsening $\beta_{j(k)}$ of $\alpha_k$. Suppose now 
instead that $j(i) \neq j(k),$ so $w_{j(i)} \neq w_{j(k)}$. The proof of 
Lemma~\ref{lemma4existence} shows that in this case $v_{ik} = w_{j(i)j(k)}$. 
Hence, the $v_{ik}\text{-}$coarsening of $\alpha_i$ is the 
$w_{j(i)j(k)}\text{-}$coarsening of the $w_{j(i)}$-coarsening 
$\beta_{j(i)}$ of $\alpha_i$. By Th.~\ref{thm:compatible}, this coincides 
with the $w_{j(i)j(k)}\text{-}$coarsening of the $w_{j(k)}$-coarsening 
$\beta_{j(k)}$ of $\alpha_k$. Thus, $\alpha_i$ and $\alpha_k$ have the 
same $v_{ik}\text{-}$coarsening. Since the compatibility condition thus 
holds in all cases, Th.~\ref{thm:compatible} shows that $\alpha$ is a 
$v$-gauge. It is a minimal gauge since each $\alpha_i$ is a minimal gauge.
Moreover,
$\Gamma_\alpha \subseteq \bigcup_{i = 1}^r\Gamma_{\alpha_i}
\subseteq \divh{\Gamma_v}$.
\end{proof}

\section{An example}

In this section we construct an example of a central simple
algebra with multiple non-isomorphic minimal
gauges all having the same gauge ring.

\begin{example}\label{ex:notdetermined}
Let $L$ be a field with
$\charac(L) \ne 2$, let $x$ be transcendental over $L$,
and let ~${F = L(x)((y))}$,
the Laurent series field in one variable over $L(x)$.
Let $w$ be the complete discrete (so Henselian) $y$-adic
valuation on $F$, with $\Gamma_w = \Z$ and
$\ov F^{\,w} = L(x)$.  Let $W$ be the power series
ring $L(x)[[x]]$, which is the valuation ring of $w$.
Let $v$ be the rank $2$ valuation on $F$ that is the composite
of $w$ with the discrete $x$-adic valuation on $\ov F^{\,w}$.
Equivalently $v$ is the valuation on~$F$
obtained by restriction
from the standard rank $2$ Henselian valuation
on $L((x))((y))$.  Thus, $\ov F^{\,v}= L$, $\Gamma_{v} =
\Z \times \Z$ with  right-to-left lexicographic
ordering, $v(x) = (1,0)$, $v(y) = (0,1)$, and
$\gr_v(F) = L[X,X\inv,Y,Y\inv]$, a twice
iterated Laurent polynomial ring, where $X = \tilde x$
and $Y = \tilde y$.  Let $V$ be the valuation ring of $v$.
Note that
 $w$ is the rank~$1$ coarsening of $v$,
and the epimorphism $\varepsilon\colon \Gamma_{v} \to
\Gamma_{w}$
given by $v(c) \mapsto w(c)$ for~$c\in F^\times$ is
the projection $(\ell, m) \mapsto m$.  Since we will be
working primarily with $v$, we write
$\ov F$ for~$\ov F^{\,v}$.

Let
$$
D \, = \, \big (1+x,y/F),
$$
a quaternion  algebra over $F$ with its standard $F$-base $(1, i,j,k)$,
where $i^2 = 1+x$, $j^2 = y$, and $k = ij = -ji$.
Because $w(1+x) = 0$ and $\ov{1+x}^{\,w}$ is not a square in
$\ov F^{\,w}$,   the
valuation $w$ has a unique and inertial extension to
the field $F(i)$.  Therefore, every norm from $F(i)$ to $F$
has $w$-value in $2\Gamma_w$.  Since $w(y) = 1$, $y$ is not
such a norm.  Hence, $D$ is a division algebra.  The Henselian
valuation $w$ on $F$ therefore has a unique extension to a
valuation $\beta$ on $D$, with $\beta(1) = \beta(i) =0$
and $\beta(j) = \beta(k) = \half$.  Since $\beta(j) \notin
\Gamma_w$ one can see that $D$ is totally ramified over
$F(i)$ for $\beta$, while $F(i)$ is inertial over $F$; indeed,
 $\Gamma_\beta = \half \Z$ and $\ov D^{\,\beta}
= \ov {F(i)}^{\,\beta} = L(x)(\sqrt{1+x})$, and for
all $a,b,c,d\in F$,
\begin{equation}\label{eq:betaformula}
\textstyle
\beta(a + bi + cj+dk) \, = \, \min\big(
\beta(a + bi), \beta(cj+dk)\big) \, =\,
 \min\big(w(a), w(b), w(c) + \half, w(d) + \half\big).
\end{equation}

Let $K = F(t)$, where $t^2= 1+x$.  Thus,
$K$ is a quadratic extension field of
$F$, and since $\ov {1+x} = \ov 1$ in $\ov F$,
$v$ has two extensions to $K$ that are distinguished by whether
$\ov t = \ov 1$ or $\ov{-1}$ in $\ov K$.  Let $v'$ denote
the extension of $v$ to $K$ with $\ov t = 1$.
Then, $\ov K = \ov F = L$ and
$\Gamma_{v'} = \Gamma_v = \Z\times \Z$, so
$\gr(K) = \gr(F)$.  Note
that as $x = (t-1)(t+1)$ and $v'(t+1) =0$, we
have $v'(t-1) = v(x) = (1,0)$. The rank~$1$ coarsening
of $v'$ is the unique, unramified, extension $w'$ of
$w$ to $K$, with $\ov K^{w'} = L(x)(\sqrt{1+x})$
and $\Gamma_{w'} = \Z$.
Also, $K$ is a splitting
field of~$D$, as $K \cong F(i)$, which is a maximal
subfield of $D$.  Explicitly, let $S = \mat_2(K)$, and
view~$D$ as an
$F$-subalgebra of $S$ by identifying
$$
1 \, =\,  \left(\begin{smallmatrix} 1\  0\\
\vphantom {2^X}0\  1
 \end{smallmatrix}\right), \quad
i \, =\,  \left(\begin{smallmatrix} t \ \ \, 0\\
\vphantom {2^X}0\  -t
 \end{smallmatrix}\right), \quad
j \, =\,  \big(\begin{smallmatrix} 0 \  y\\
\vphantom {2^X}1\  0
\end{smallmatrix}\big), \quad
k \, =\,  \big(\begin{smallmatrix}  \ \,0 \ \ t y\\
\vphantom {2^X}-t\   \, \,0
\end{smallmatrix}\big).
$$
Give $\Q \times \Q$ the right-to-left lexicographic
ordering.  Fix any $\gamma \in \Q$ with $0<\gamma< \half$,
and let $\delta = (\gamma, \half) \in \Q \times \Q$.
Let $\alpha'$ be the $v$-gauge on $S$ given by
$$
\alpha'\big(\begin{smallmatrix} p \ q\\
\vphantom {2^X}r\   \, s
\end{smallmatrix}\big)\, = \, \min\big(v'(p),
\ v'(q) - \delta,\
v'(r) + \delta, \ v'(s)\big).
$$
Indeed, let $M = K\text{-} \textsl{span}\big\{
\big(\begin{smallmatrix} 1\\ \vphantom {2^X}0
\end{smallmatrix}\big),
\big(\begin{smallmatrix} 0\\ \vphantom {2^X}1
\end{smallmatrix}\big)\big\}$, and identify
$S = \End_K(M)$.  Then $\alpha'$ is the $v'$-gauge
$\End(\eta)$
as in Ex.~\ref{ex;endgauge}, where $\eta\colon M \to \Q \times \Q
\cup \{\infty\}$ is the $v'$-norm on $M$
given by $\eta\big(\begin{smallmatrix} p\\ \vphantom {2^x}q
\end{smallmatrix}\big) = \min\big(v'(p),v'(q)+\delta\big)$.
Thus,
$$
\gr_{\alpha'}(S) \, =\, \grEnd_{\gr(K)}(\gr(M)) \,\cong_g
\, \mat_2(\gr(K))(0,\delta),
$$
in the notation of \eqref{eq:shiftmatrix}.
Let $\alpha = \alpha'|_D$, which is a surmultiplicative
$v$-value function on $D$ since the gauge $\alpha'$
on $S$ is surmultiplicative.
While $\alpha'$ is a $v'$-gauge, we must still
verify that $\alpha$ is a $v$-gauge.  For this,
note that  for any
$z = a + bi +cj+dk\in D$ with $a,b,c,d\in F$, we have
$z = \big( \begin{smallmatrix}a+bt\ \, (c+dt)y
\\ \!\!\!\!c-dt\vphantom {2^X}\ \ \,\, a-bt
\end{smallmatrix} \big)$ in $S$. Thus,  as
$v(y) = (0,1)$,
\begin{align}\label{eq:alphaeq}
\begin{split}
\alpha(z) \, &= \, \min\big(v'(a+bt), \, v'((c+dt)y)-
\delta,\,
v'(c-dt) +\delta, \, v'(a-bt)\big)\\
&=\, \min\big(v'(a+bt), \, v'(a-bt), \,
v'(c+dt)+(-\gamma, \textstyle\half),
\,v'(c-dt)+ (\gamma, \half)\big).
\end{split}
\end{align}
So, $\alpha(1) = \alpha(i) = 0$ and
$$
\textstyle
\alpha(j) \, = \, \alpha(k) \, = \,
\min\big((-\gamma, \half), (\gamma, \half) \big)\,
= \, (-\gamma, \half).
$$
Since $v'(1-t) = (1,0)$, we have
$$
\textstyle
\alpha(j-k) \, = \, \min \big((-\gamma, \half) + (1,0),
(\gamma, \half)\big) \, = \, (\gamma,\half),
$$
 as $\gamma< \half$. So, in $\gr_\alpha(D) \subseteq
\gr_{\alpha'}(S)$,
\begin{align*}
&\tilde {1+i} \, = \big(\begin{smallmatrix}2\ 0\\
\vphantom {2^X}0 \ 0\end{smallmatrix}\big) \in D_0,
\qquad\tilde {1-i} \, = \big(\begin{smallmatrix}0\ 0\\
\vphantom {2^X}0 \ 2\end{smallmatrix}\big) \in D_0,\\
&\tilde j = \tilde k =
\big(\begin{smallmatrix}0\ \tilde y\\
\vphantom {2^X}0 \ 0\end{smallmatrix}\big)
\in D_{(-\gamma, \half)},\quad\text{ and }
\quad\tilde{j-k} =
\big(\begin{smallmatrix}0\ 0\\
\vphantom {2^X}2 \ 0\end{smallmatrix}\big)
\in D_{(\gamma,\half)}.
\end{align*}
Since  $1+i, \,1-i, \,j, \,j-k$ have images in
$\gr_\alpha(D)$ which are clearly $\gr(F)$-independent,
they comprise a splitting base of $\alpha$ as a
$v$-value function; this shows that $\alpha$ is a
$v$-norm on $D$.   Moreover,
${\DIM{\gr_\alpha(D)}{\gr(F)} = 4 =
\DIM{\gr_{\alpha'}(S)}{\gr(K)}
= \DIM{\gr_{\alpha'}(S)}{\gr(F)}}$, since
${\gr(K) = \gr(F)}$.
Hence, ${\gr_\alpha(D) = \gr_{\alpha'}(S)}$, which is
graded simple.
Thus, $\alpha$~is a $v$-gauge on $D$.
Note that
\begin{equation}\label{eq:GammaDalpha}
\Gamma_{\alpha} \,=\, \Gamma_{\alpha'}
\,=\,\, \Z^2 \cup (\delta + \Z^2) \cup (-\delta + \Z^2).
\end{equation}
Also,   $D_0 = S_0 = L\times L$, so $\omega(\alpha) = 2$.

From \eqref{eq:alphaeq}, we have
\begin{multline*}
\textstyle
\qquad R_\alpha\, = \, \{ a +bi +cj+dk \in D \mid
v'(a+bt) \ge 0, \ v'(a-bt) \ge 0,\\
\textstyle \ v'(c+dt) \ge (\gamma, -\half), \,
v'(c-dt) \ge (-\gamma, -\half) \}.\qquad
\end{multline*}
Let $v'(c+dt) = (\ell, m) \in \Gamma_{v'} = \Z \times \Z$.
Then, $v'(c+dt) \ge (\gamma, -\half)$ if and only
if $m \ge 0$, i.e., $w'(c+dt) \ge 0$ where $w'$
is the rank $1$ coarsening of $v'$.
Likewise, $v'(c-dt) \ge (-\gamma, -\half)$ if and only
if $w'(c-dt) \ge 0$.
Therefore, each of the infinitely many choices of
$\gamma$ yields the same gauge ring for the associated
$v$-gauge $\alpha$.  But different choices of
 $\gamma$ yield non isomorphic gauges since the gauges
have
 different value sets
(see \eqref{eq:GammaDalpha}).  Thus, the gauge ring
$R_\alpha$ does not determine $\alpha$.

We still must show that $\alpha$ is a minimal gauge.
We will show as well that $R_\alpha$ is an intersection
of two total valuation rings $B_1$,~$B_2$, which are the only
Dubrovin valuation rings of $D$
with center $V$.

Since $\charac(\ov F) \ne 2$ and $v'(t) = 0$,
we have $\min\big(v'(a+bt),
v'(a-bt)\big) = \min\big(v(a),v(b)\big)$.
Similarly, $\min(w'(c+dt),w'(c-dt)) =
\min\big(w(c),w(d)\big)$. Thus,
the description of
$R_\alpha$ simplifies to
$$
R_\alpha\, = \, \{a+bi+cj+dk\mid a,b \in V,
 \ c,d \in W \}.
$$
The ring $R_\alpha$ lies in the invariant valuation ring
$R_\beta$ of the valuation $\beta$ on $D$
extending $w$ on $F$.  From \eqref{eq:betaformula}, we have
\[
R_\beta \, = \, \{a+bi+cj+dk\mid a,b, c, d \in W \} \quad \text{and} \quad J(R_\beta) \, = \, \{a+bi+cj+dk\mid a,b \in J(W), \, c,d \in W \}.
\]
Let
\[
\pi\colon R_{\beta} \rightarrow R_{\beta}/J(R_{\beta}) =  L(x)(\sqrt{1+x})
\]
be the canonical projection. Let $U_1$,~$U_2$ be the two valuation
rings of $L(x)(\sqrt{1+x})$
extending the valuation ring $U$ of the $x$-adic valuation on $L(x)$.
 The commutative valuation rings $U_\ell$ for $\ell = 1,2$ are Dubrovin valuation
rings. Since the invariant valuation ring $R_\beta$ is a Dubrovin valuation
ring, the pullback rings $B_\ell = \pi^{-1}(U_\ell)$ for $\ell=1,2$
 are  Dubrovin valuation rings of $D$ with center $V$ and $\tilde{B_\ell}
= B_\ell /J(R_\beta) = U_\ell$.
Moreover, since $R_\beta$ and $U_\ell$ are valuation rings, it is easy
to check that $B_\ell$ is a total valuation ring, i.e., for any
$d\in D \setminus\{0\}$, $d\in B_\ell$ or $d\inv \in B_\ell$.
Let $B = B_1$.  Since $B/J(B) \cong U_1/J(U_1)$, which is a field,
$t_B = 1$.  Also, if $(F_h,v_h)$ is a Henselization of $(F,v)$, then 
$1+x\in F_h^2$, since $\ov {1+x} = \ov 1$ in
$\ov{F_h}^{\,v_h} = \ov F^{\,v}$. Therefore, $F_h$ splits $D$, which
shows that $D\otimes_F F_h \cong \mat_2(F_h)$, hence $n_B = 2$.  Thus,
$$
\xi_{V,[D]} \, = \, n_B/t_B = 2 = \omega(\alpha),
$$
showing that $\alpha$ is a minimal gauge.
Another way to calculate  the extension number  $\xi_{V,[D]}$
 is by using \eqref{teoremaE} with $R_\beta$ for $S$:  Since the valuation
rings $R_\beta$ and
$U_1=\tilde B$ have extension number $1$ and the residue valuation ring
$U = V/J(W)$ has two extensions to $R_\beta/J(R_\beta)$,
$$
\xi_{V,[D]} \, = \, \xi_{R_\beta,[D]}\, \ell_{V,W}  \,
\xi_{U_1,[R_\beta/J(R_\beta)]}\, = \, 1\cdot 2\cdot 1 \, = \,2.  
$$

Any inner automorphism $\iota$ of $D$ maps the invariant valuation ring
$R_\beta$ to itself.  The automorphism  induced by $\iota$ on
$R_\beta/J(R_\beta)$  is one of the two $\ov F^{\,w} = L(x)$-automorphisms
of $L(x)(\sqrt{1+x})$, so it either preserves or interchanges $U_1$ and $U_2$.
Hence, the set of conjugates of $B_1$ in $D$ is $\{B_1,B_2\}$. The $B_\ell$
are therefore the only Dubrovin valuation rings of $D$ with center
$V$.  Since the $B_\ell$ are not integral over $V$ (as  $\xi_{V,[D]} \ne 1$)
the only possible Gr\" ater ring of $D$ with center $V$ is $B_1\cap B_2$.
Since $R_\alpha$ is the gauge ring of a minimal $v$-gauge, it is a  Gr\"ater ring
with center $V$ by Th.~\ref{teoremaprincipal}.   Hence, $R_\alpha = B_1 \cap B_2$.
(This equality can also be verified directly after first showing that
$V + Vi$ is the integral closure of~$V$ in $F(i)$ and  that
the $B_\ell\cap F(i)$ are the valuation rings of $F(i)$ extending
$V$ in $F$.)

Note that our example required a valuation of rank at least $2$.
For if $v$ is a rank $1$ valuation on a field $F$ and $A$ is a
central simple $F$-algebra, then for the valuation ring
$V$ of $v$, we have $j(v,A) = 1$, so $\xi_{V,[A]}= 1$
by Prop.~\ref{prop:maxint}(iv).  Hence, for any minimal  
$v$-gauge $\alpha$ on $A$, we have $\omega(\alpha) = 1$, so 
by Remark~\ref{rem:omega1} the gauge ring $R_\alpha$
is a Dubrovin valuation ring integral over its center $V$, and
$\alpha$ is the Morandi value function determined by $R_\alpha$.
\end{example}

\appendix

\section{Tensor product and Henselization}

It is well known that if $L/F$ is a finite degree field extension, $v$ is
a discrete (rank $1$) valuation on~$F$, and $v_1, \ldots , v_r$ are all
the extensions of $v$ to $L$, then
$L \otimes_F \widehat{F} \cong \prod_{i=1}^r \widehat{L_i}$, where
$\widehat{F}$ (resp.~$\widehat{L_i}$) is the completion of $F$ (resp.~$L$)
with respect to  $v$ (resp.~$v_i$) (see \cite[Ch.~ VI, \S 8, no.~6,
Prop.~11]{B}). In this appendix we prove an analogous
result for valuations of arbitrary rank, replacing the completion by the
Henselization. For separable field extensions, this result is implicit
in \cite[Th.~17.17, p.~135]{E}. We give a full proof here, since this 
result is essential for many of the arguments in this paper. 

Let $F$ be a field with a
valuation $v$. A \emph{Henselization} of $(F,v)$ is a valued field
extension $(F_h, v_h)$ of $(F,v)$ such that $v_h$ is Henselian and for
any extension $(K, w)$ of $(F,v)$ with $w$ Henselian there is a unique
$F$-homomorphism $\eta\colon (F_h,v_h) \rightarrow (K,w)$ such that
$v_h = w \circ \eta$. We thus refer to $(\eta(F_h), w|_{\eta(F_h)})$ as
the {\it Henselization of $(F,v)$ within $(K,w)$}. It is clear from the
definition that a Henselization of $(F,v)$ is unique up to unique
isomorphism. Thus, we sometimes say that $(F_h,v_h)$
is ``the Henselization'' of $(F,v)$. A proof of the existence of a
Henselization can be found in \cite[Th.~5.2.2, p.~121]{EP}.  

\begin{theorem}\label{tensorhens}
 Let $F$ be a field with a valuation $v$. Let $K$ be a finite-degree field
extension of $F$, and let  $v_1, \ldots, v_r$ be all the extensions of
$v$ to $K$. Let $(F_h,v_h)$ be a Henselization of $(F,v)$, and let
$(K_{h,i}, v_{i,h})$ be a Henselization of $(K_{i}, v_{i})$ for
$i =1,  \ldots, r$. Then,
 \[
 K \otimes_F F_h \,\cong\, K_{h,1} \times \ldots \times K_{h,r}.
  \]
\end{theorem}

The proof will use the following two lemmas:

\begin{lemma}\label{lemadoublecosets}
Let $F \subseteq N$ be fields with $N$ Galois over $F$ $($possibly of
infinite degree$)$, and let ${G = \mathcal{G}(N/F)}$. Let $K$ and $E$ be
subfields of $N$ containing $F$, with $\DIM KF < \infty$. Let
${H = \mathcal{G}(N/K) \subseteq G}$ and $Z = \mathcal{G}(N/E) \subseteq G$.
Let $\tau_1, \ldots, \tau_r$ be representatives of the distinct $Z$-$H$
double cosets of $G$. $($So, $G = \bigsqcup_{i=1}^r Z \tau_i H,$ a disjoint
union.$)$ Then
\[
K \otimes_FE \ \cong\, {\tau_1(K)\!\cdot\! E} \times \ldots \times
{\tau_r(K)\!\cdot\! E}.
\]
\end{lemma}
\begin{proof}
Since $K$ is separable over $F$, we have $K = F(a)$ for some $a$.
Let $f$ be the minimal polynomial of $a$ over $F$. Then $f$ splits
over $N$, as $N$ is normal over $F$, say $f = (X-a_1) \ldots
(X - a_n) \in N[X]$ where the $a_i$ are distinct and $a_1 = a$. Let
$\mathcal{A} = \{a_1, \ldots,a_n\}$. Let $f = g_1\ldots g_r$ be the
irreducible factorization of $f$ in $E[X],$ and fix a root $b_i$ of
$g_i$ for $i = 1,  \ldots, r.$ The Galois group $G$ acts transitively on
$\mathcal{A}$, but $\mathcal{A}$ decomposes into $r$ disjoint $Z$-orbits,
$\mathcal{A} = \bigsqcup_{i=1}^r \mathcal{B}_i$, where
\[
\mathcal{B}_i \, =\, Z \!\cdot\! b_i \,=\, \{ \text{roots of } g_i 
\text{ in } N \}.
\]
For each $i$, choose $\tau_i' \in G$ with $\tau_i'(a) = b_i$. Then,
as $H = \{ \sigma \in G \; | \; \sigma(a) =a \}$, we have
\[
Z \tau_i' H \,=\, \{ \sigma \in G \; | \; \sigma(a) \in \mathcal{B}_i \},
\quad \quad \text{for } i=1,\ldots,r.
\]
Thus, $Z \tau_1' H, \ldots, Z \tau_r' H$ are all the distinct $Z$-$H$ double
cosets in $G$. Moreover, we may assume that the double coset
representatives $\tau_i'$ coincide with the $\tau_i$ of the lemma, by
replacing $b_i$ by $\tau_i(a)$. Since  $\gcd(g_i,g_j) = 1$ for
$i \neq j$, the Chinese Remainder Theorem yields
\begin{align*}
K \otimes_F E \, &  \cong \, F[X]/(f) \otimes_F E \ \cong\, E[X]/f E[X]\\
              & \cong \, E[X]\big/\big((g_1)\ldots (g_r)\big) \,\cong\, E[X]/(g_1)
\times \ldots \times E[X]/(g_r)\\
              & \cong \, E(b_1) \times \ldots \times E(b_r) \,\cong \,
{\tau_1(K)\!\cdot\! E} \times \ldots \times {\tau_r(K) \!\cdot\! E}.
\end{align*}
\end{proof}

\begin{lemma}\label{henselizacaocompositum}
Let $(F_h,v_h)$  be a Henselization of the valued field $(F,v)$. Let $K$
be any extension field of $F$ lying in the algebraic closure of $F_h$,
and let $w$ be the unique extension of $v_h$ to the compositum ~
${K \!\cdot\! F_h}$. Then, $({K \!\cdot\! F_h}, w)$ is a Henselization of
$(K, w|_K)$.
\end{lemma}
\begin{proof}
The valuation $w$ on ${K \!\cdot\!F_h}$ is Henselian since $v_h$ is Henselian.
Let $(K_h,w_h)$ be a Henselization of $(K,w|_K)$. The universal
property shows
that $(K_h,w_h)$ embeds in $({K \!\cdot\! F_h},w)$. Thus, we may assume that
$K \subseteq K_h \subseteq {K \!\cdot\! F_h}$ and
$w_h = w|_{K_h}$. So, $w_h|_F = w|_F =v$.
Since $(F_h, v_h)$ is the Henselization of $(F,v)$
within ${(K \!\cdot\! F_h,w)}$, it is also the Henselization of $(F, v)$ within
$(K_h,w_h)$, by the uniqueness in the universal property for the
Henselization; so, $F_h \subseteq K_h$. Since also $K \subseteq K_h$, we
have ${K \!\cdot\! F_h} \subseteq K_h \subseteq {K \!\cdot\! F_h}$. Hence,
${K_h = K \!\cdot\! F_h}$ and $w = w_h$, showing that $({K \!\cdot\! F_h},w)$ is a
Henselization of $(K,w|_K)$. 
\end{proof}

\begin{proof}[Proof of Th.~\ref{tensorhens}:] Assume first that $K$ is 
separable over $F$. Let  $F_\sep$ be a separable closure of $F$ containing
$F_h$, and let $v_\sep$ be the unique valuation on $F_\sep$ extending the
Henselian valuation $v_h$. Let $G = \mathcal{G}(F_\sep/F)$,
$H = \mathcal{G}(F_\sep/K) \subseteq G$, and
$Z = \mathcal{G}(F_\sep/F_h) \subseteq G$. By \cite[Th.~5.2.2, p.~121]{EP} and 
the universal property of the Henselization, $F_h$ is the decomposition
field for $v_\sep$ over
$v$, so $Z$ is the decomposition subgroup of $G$, i.e.,
\[
Z \,=\,  \{ \sigma \in G \; | \; v_\sep \circ \sigma = v_\sep \}.
\]
Let $\Omega$ be the set of all valuations on $F_\sep$ extending $F$.
Then, $G$ acts transitively on $\Omega$ (see \cite[Th.~3.2.14, p.~68]{EP}), while
the distinct $H$-orbits
of $\Omega$ are $\Omega_1, \ldots, \Omega_r$, where
$\Omega_i = \{w \in \Omega \; | \; w|_K = v_i \}$. For $i=1, \ldots,r$,
choose $\tau_i \in G$ with $v_\sep \circ \tau_i|_K = v_i$. Then,
\[
\{ \sigma \in G \; | \; v_\sep \circ \sigma|_K = v_i \} \,=\,
\{ \sigma \in G \; | \; v_\sep \circ \sigma \in \Omega_i \} \,=\, Z\tau_iH.
\]
So, $G = \bigsqcup_{i=1}^r Z \tau_i H$ is the disjoint $Z$-$H$ double coset
decomposition of $G$. We now apply Lemma~\ref{lemadoublecosets} with
$N =F_\sep$ and $E = F_h$. (So, the $K$, $H$ and $Z$ of the lemma are
the $K$, $H$ and $Z$ here.) By the lemma,
\begin{equation}\label{hensdirectproduct}
K \otimes_F F_h \ \cong\, \tprod_{i=1}^r {\tau_i(K)\!\cdot\! F_h}
\ \cong \,\tprod_{i=1}^r {K \!\cdot\! \tau_i^{-1}(F_h)},
\end{equation}
where the second isomorphism follows by applying $\tau_i^{-1}$ to the
$i$-th factor. Note that the $F$-isomorphism $\tau_i^{-1}$ maps
$(F_h,v_h)$ to $(\tau_i^{-1}(F_h), v_h \circ \tau_i)$. Hence
$(\tau_i^{-1}(F_h), v_h \circ \tau_i)$ is a Henselization of $(F,v)$. The
unique extension of the Henselian valuation $v_h \circ \tau_i$ to
${K\!\hsp \cdot\!\hsp \tau_i^{-1}(F_h)}$ must be
$v_\sep \circ \tau_i|_{K \cdot \tau_i^{-1}(F_h)}$, whose restriction to
$K$ is $v_i$ by the choice of $\tau_i$. Therefore,
${K \!\cdot\! \tau_i^{-1}(F_h)} \cong K_{h,i}$ by~
Lemma~\ref{henselizacaocompositum}. The theorem
(for $K$ separable over $F$) then follows from
\eqref{hensdirectproduct}.

If $K$ is not separable over $F$, let $S$ be the separable closure of
$F$ in $K$, and let $y_i = v_i|_S$ for $i =1, \ldots, r$. Then,
$y_1, \ldots,y_r$ are all the extensions of $v$ to $S$. Since
valuations extend uniquely from $S$ to its purely inseparable extension
$K$, we have $y_i \neq y_j$ for $i \neq j$. As we just proved,
$S \otimes_F F_h \cong \prod_{i=1}^r S_{h,i}$, where $(S_{h,i},y_{i,h})$
is a Henselization of $(S,y_i)$ in $F_\alg$, the algebraic closure
of~$F$. Therefore,
\begin{equation}\label{hensdirectproduct2}
K \otimes_F F_h \, \cong \, K \otimes_S(S \otimes_F F_h) \, \cong
\,\tprod_{i=1}^r K \otimes_F S_{h,i}.
\end{equation}
Because $K$ is purely inseparable over $S$ while $S_{h,i}$ is separable
over $S$, the  fields $K$ and $S_{h,i}$
are linearly disjoint over $S$; so, $K \otimes_S S_{h,i}$ is a field,
which is isomorphic to the
compositum ${K \!\cdot\! S_{h,i}}$ in $F_\alg$. The Henselian valuation
$y_{i,h}$ on $S_{h,i}$ has a unique
extension to the field $K \otimes_S S_{h,i}$ whose restriction to~$K$ is
the unique extension of
$y_i$ to $K,$ which is $v_i$. By Lemma~\ref{henselizacaocompositum},
$K \otimes_S S_{h,i}$ is a Henselization of $K$ with respect to $v_i$.
The theorem thus follows from \eqref{hensdirectproduct2}.
\end{proof}

\end{document}